\documentclass{amsart}
\usepackage{graphicx}
\usepackage{hyperref}
\usepackage{amssymb}
\usepackage{amsfonts}
\usepackage{euscript}
\usepackage[all]{xy}
\usepackage{bbm}

\numberwithin{equation}{section}
\numberwithin{figure}{section}
\renewcommand{\subsection}[1]{\hspace{-\parindent}\refstepcounter{subsection}{\bf (\arabic{section}\alph{subsection}) #1.}\addcontentsline{toc}{subsection}{\bf #1.}}

\newenvironment{nouppercase}{%
  \renewcommand{\uppercasenonmath}[1]{}}{}

\theoremstyle{plain}
\newtheorem{thm}{Theorem}[section]
\newtheorem{theorem}[thm]{Theorem}

\newtheorem{addendum}[thm]{Addendum}

\newtheorem{corollary}[thm]{Corollary}

\newtheorem{prop}[thm]{Proposition}
\newtheorem{assumption}[thm]{Assumption}

\newtheorem{remark}[thm]{Remark}

\newtheorem{convention}[thm]{Convention}
\newtheorem{proposition}[thm]{Proposition}
\newtheorem{example}[thm]{Example}

\newtheorem{lemma}[thm]{Lemma}

\newtheorem*{claim*}{Claim} 
\newtheorem*{lemma*}{Lemma}
\newtheorem*{theorem*}{Theorem}
\newtheorem*{conjecture*}{Conjecture}
\newtheorem*{references}{References}

\newcommand{\bC}{{\mathbb C}}

\newcommand{\bK}{{\mathbb K}}

\newcommand{\bN}{{\mathbb N}}

\newcommand{\bQ}{{\mathbb Q}}
\newcommand{\bR}{{\mathbb R}}

\newcommand{\bZ}{{\mathbb Z}}

\newcommand{\scrA}{\EuScript A}

\newcommand{\scrG}{\EuScript G}
\newcommand{\scrH}{\EuScript H}

\newcommand{\scrJ}{\EuScript J}

\newcommand{\scrM}{\EuScript M}

\newcommand{\scrP}{\EuScript P}

\newcommand{\scrS}{\EuScript S}

\newcommand{\scrY}{\EuScript Y}

\newcommand{\frakg}{\mathfrak{g}}

\newcommand{\half}{{\textstyle\frac{1}{2}}}
\newcommand{\quarter}{\textstyle\frac{1}{4}}
\newcommand{\iso}{\cong}
\newcommand{\htp}{\simeq}
\newcommand{\smooth}{C^\infty}
\newcommand{\Id}{\mathbbm{1}}

\newcommand{\qabla}{\nabla\mkern-12mu/\mkern2mu}

\title[LEFSCHETZ FIBRATIONS]{\Large\larger\rm Fukaya $A_\infty$-structures associated to\\ Lefschetz fibrations. IV}
\author{Paul Seidel}

\begin{document}
\begin{nouppercase}
\maketitle
\end{nouppercase}

\begin{abstract}
We consider Hamiltonian Floer cohomology groups associated to a Lefschetz fibration, and the structure of operations on them. As an application, we will (under an important additional assumption) equip those groups with connections, which differentiate with respect to the Novikov variable.
\end{abstract}

\section{Introduction\label{sec:intro}}
This paper continues a discussion \cite{seidel12b,seidel14b,seidel15,seidel16} of pseudo-holomorphic curve theory as applied to Lefschetz fibrations. Ultimately, our sights are set on Fukaya categories; but here, we remain in the ``closed string'' context of Hamiltonian Floer cohomology. 

Given a symplectic Lefschetz fibration (over the disc), one can single out certain classes of time-dependent Hamiltonians on the total space, which give rise to Floer cohomology groups that are invariants of the fibration. More concretely, the outcome is an infinite family of such groups, indexed by the ``rotation number at infinity'' $r \in \bZ$. All of them vanish if the Lefschetz fibration is trivial (has no critical points). Based on intuition from mirror symmetry, one expects these groups to carry a rich algebraic structure. The initial breakthrough in constructing the desired operations is due to Abouzaid-Ganatra (unpublished), who equipped the $r = 1$ group with a ring structure. Here, we streamline and generalize their insight, by introducing a suitable TQFT framework. 
Eventually, one wants to compare Hamiltonian Floer cohomology with certain twisted Hochschild cohomology groups of the associated Fukaya category, and thereby relate our operations to ones in homological algebra; however, to carry that out, one needs a specially adapted version of closed-open string maps, which is beyond the scope of the present paper (the literature contains several constructions relating the open and closed string sectors for Lefschetz fibrations, but none of them seems really practical for this purpose).

Our main motivation for pursuing this direction comes from \cite{seidel16}. There, symplectic cohomology was shown, under a fundamental additional assumption, to carry a one-parameter family of connections, which differentiate with respect to the Novikov formal parameter. That construction relied on the structure of symplectic cohomology as a BV algebra. Using some of the operations we have made available, a version of the same construction can be carried out for Lefschetz fibrations. The outcome, under a similar assumption as in \cite{seidel16}, is that each of our Floer cohomology groups carries a connection; for different $r$, these correspond to different instances of the family of connections on symplectic cohomology. Ultimately, one wants to use these connections, and closed-open string maps, to carry over the enumerative results from \cite{seidel16} to Fukaya categories of Lefschetz fibrations; we refer the reader to \cite[Section 4]{seidel15} for the conjectural outcome and its context.

The structure of this paper may deserve some comment. Floer-theoretic constructions and identities have often been thought of as realizations of abstract TQFT arguments, and that is also true here. Usually, the formal TQFT part is straightforward, and the main work goes into its Floer-theoretic implementation. In our situation, the weight shifts: the TQFT arguments are far from obvious, whereas the Floer theory is quite standard (except possibly for some compactness arguments). Therefore, after outlining our results in Section \ref{sec:results}, the bulk of the paper (Sections \ref{sec:sl2}--\ref{sec:elliptic}) is devoted to TQFT considerations and their geometric prerequisites. After that, we explain the necessary basic Floer theory (Sections \ref{sec:maps-to-the-disc}--\ref{sec:floer}). Putting together the two parts is a routine process, and we will only describe it briefly (Section \ref{sec:end}).

{\em Acknowledgments.} This work was supported by the Simons Foundation, through a Simons Investigator award; by NSF grant DMS-1500954; by the Institute for Advanced Study, through a visiting appointment supported by grants from the Ambrose Monell Foundation and the Simonyi Endowment Fund; and by Columbia University, through an Eilenberg Visiting Professorship. I would like to thank Mohammed Abouzaid and Sheel Ganatra for generously sharing their seminal insights; Nick Sheridan for illuminating conversations on further developments; and Alexander Sukhov for patiently answering my questions concerning pseudoconvexity (see Remark \ref{th:sukhov}).

\section{Main constructions\label{sec:results}}
This section summarizes the structures that will arise from our analysis of Floer-theoretic operations, including connections. This will be done with only the minimal amount of technical details, and without any real attempt at explaining the underlying geometry. At the end, we relate part of our constructions to more familiar ones in symplectic cohomology.

\subsection{Background\label{subsec:background}}
Let
\begin{equation} \label{eq:lefschetz}
\pi: E \longrightarrow B
\end{equation}
be a proper symplectic Lefschetz fibration whose base is the open unit disc $B \subset \bC$, and whose total space is a $2n$-dimensional symplectic manifold. For simplicity, we assume that
\begin{equation} \label{eq:cy}
c_1(E) = 0,
\end{equation}
and choose a trivialization of the anticanonical bundle (the complex line bundle representing the first Chern class), up to homotopy. The smooth fibres are then closed $(2n-2)$-dimensional symplectic manifolds, again with trivialized anticanonical bundle. 

All Floer cohomology groups under discussion will be finite-dimensional $\bZ$-graded vector spaces over the Novikov field $\bK$ with real coefficients. Elements of $\bK$ are of the form
\begin{equation} \label{eq:novikov}
f(q) = c_0 q^{d_0} + c_1 q^{d_1} + \cdots, \quad c_i \in \bR, \;\; d_i \in \bR, \;\; \textstyle \lim_i d_i = +\infty.
\end{equation}
Concretely, to \eqref{eq:lefschetz} we associate a family of such Floer cohomology groups, denoted by
\begin{equation} \label{eq:floer-cohomology}
\mathit{HF}^*(E,r), \;\; r \in \bZ.
\end{equation}
The $r = 0$ group admits a description in purely topological terms, as the cohomology of $E$ ``relative to a fibre at $\infty$''. If we compactify \eqref{eq:lefschetz} to $\bar\pi: \bar{E} \rightarrow \bar{B}$, then the statement is that
\begin{equation} \label{eq:floer-0}
\mathit{HF}^*(E,0) \iso H^*(\bar{E},\bar{E}_w) \quad \text{for some $w \in \partial \bar{B}$.}
\end{equation}
The other groups \eqref{eq:floer-cohomology}, $r \neq 0$, are not classical topological invariants; they encode symplectic information related to the monodromy around $\partial \bar{B}$. In the simplest case $r = 1$, if $\mu$ is that monodromy (a symplectic automorphism of the fibre), with its fixed point Floer cohomology $\mathit{HF}^*(\mu)$, there is a long exact sequence
\begin{equation} \label{eq:b-sequence}
\cdots \rightarrow \mathit{HF}^*(\mu) \longrightarrow H^*(E;\bK) \stackrel{}{\longrightarrow} \mathit{HF}^*(E,1) \longrightarrow \mathit{HF}^{*+1}(\mu) \rightarrow \cdots
\end{equation}
This idea can be generalized to other $r>0$, leading to a version of the spectral sequence from \cite{mclean12}, involving powers of $\mu$ (see \cite[Lemma 6.14]{seidel14b} for the statement). The situation for $r<0$ is dual: there are nondegenerate pairings
\begin{equation} \label{eq:pairing}
\mathit{HF}^*(E,-r) \otimes \mathit{HF}^{2n-*}(E,r) \longrightarrow \bK.
\end{equation}

\subsection{Operations\label{subsec:operations}}
Our main topic is the structure of operations on \eqref{eq:floer-cohomology}. We start with linear maps. Each group $\mathit{HF}^*(E,r)$ has a canonical automorphism $\Sigma$, which generates an action of $\bZ/r$. For $r = 0$, this automorphism admits a classical interpretation as the relative monodromy map, moving $w$ around $\partial \bar{B}$ in \eqref{eq:floer-0}. For $r \neq 0$, one can think of $\mathit{HF}^*(E,r)$ as the Floer cohomology of the $r$-th power of a Hamiltonian automorphism, and then $\Sigma$ is induced by $1/r$ rotation of loops. We also have continuation maps which increase $r$, and which are invariant under the action of $\Sigma$ (on either the left or right):
\begin{align} \label{eq:continuation}
& C: \mathit{HF}^*(E,r) \longrightarrow \mathit{HF}^*(E,r+1), \\
\label{eq:continuation-2}
& D: \mathit{HF}^*(E,r) \longrightarrow \mathit{HF}^{*-1}(E,r+1).
\end{align}
The structure of multilinear operations is more interesting, and does not seem to admit a single concise description. Two simple consequences of the general theory are the following: 

\begin{prop} \label{th:operations-1}
$\mathit{HF}^*(E,1)$ carries the structure of a Gerstenhaber algebra (a commutative product and a bracket of degree $-1$, with suitable relations between them). 
\end{prop}

\begin{prop} \label{th:operations-1b}
Each $\mathit{HF}^*(E,r)$ is a Gerstenhaber module over $\mathit{HF}^*(E,1)$, in a $\bZ/r$-equi\-va\-riant way (and reducing to the diagonal module for $r = 1$).
\end{prop}

The Gerstenhaber algebra structure includes the Abouzaid-Ganatra product. Propositions \ref{th:operations-1} and \ref{th:operations-1b} assign a distingushed role to $\mathit{HF}^*(E,1)$. Operations of this kind are the most important ones for our applications, but they by no means exhaust the structures that are present. Here are two other easily stated results that can be extracted from the general theory.

\begin{prop} \label{th:operations-2}
There is the structure of a bigraded associative algebra on 
\begin{equation} \label{eq:bigraded-1}
\bigoplus_{s \geq 0} \mathit{HF}^*(E,s+1).
\end{equation}
\end{prop}

\begin{prop} \label{th:operations-3}
There is the structure of a bigraded Lie algebra on 
\begin{equation} \label{eq:bigraded-lie}
\bigoplus_{s \geq 0} \mathit{HF}^{*+1}(E,s+1)^{\bZ/(s+1)}.
\end{equation}
\end{prop}

The product from Proposition \ref{th:operations-2} contains the algebra structure from Proposition \ref{th:operations-1} (by restricting to $s = 0$), and also the module structure over that algebra from Proposition \ref{th:operations-1b}, for $r>0$. It may be worth while spelling out what Proposition \ref{th:operations-3} says, in view of the various degree shifts involved. $\mathit{HF}^i(E,r)^{\bZ/r}$ appears in our graded Lie algebra in (cohomological) degree $i-1$, and this determines the signs in the antisymmetry and Jacobi relations. In addition, the bracket is homogeneous with respect to $s = r-1$, hence is given by maps
\begin{equation}
\mathit{HF}^{i_1}(E,r_1)^{\bZ/r_1} \otimes \mathit{HF}^{i_2}(E,r_2)^{\bZ/r_2} \longrightarrow
\mathit{HF}^{i_1+i_2-1}(E,r_1+r_2-1)^{\bZ/(r_1+r_2-1)}.
\end{equation}
The Lie structure from Proposition \ref{th:operations-1} is contained in that of Proposition \ref{th:operations-3}, and so is the $\bZ/r$-invariant part of the Lie module structure from Proposition \ref{th:operations-1b}, for $r>0$. 

The product structures from Propositions \ref{th:operations-1}--\ref{th:operations-2} have a unit
\begin{equation} \label{eq:e-}
e \in \mathit{HF}^0(E,1),
\end{equation}
which can be thought of as the image of the unit in ordinary cohomology under the map from \eqref{eq:b-sequence}. Another important ingredient for us will be the image of the symplectic class, or more precisely of $q^{-1}[\omega_E] \in H^2(E;\bK)$, under that map. We call this image the Kodaira-Spencer class, and write it as
\begin{equation} \label{eq:k-k}
k \in \mathit{HF}^2(E,1).
\end{equation}
Following the model of \cite{seidel16}, vanishing of the Kodaira-Spencer class gives rise to another operation, which is not $\bK$-linear, but instead differentiates with respect to the Novikov parameter $q$:

\begin{prop} \label{th:diff-q}
Suppose that \eqref{eq:k-k} vanishes. Choose a bounding cochain for its cocycle representative, in the Floer cochain complex underlying $\mathit{HF}^*(E,1)$. That choice determines a connection on each group \eqref{eq:floer-cohomology}, which means a map
\begin{equation} \label{eq:nabla-r}
\begin{aligned}
& \nabla: \mathit{HF}^*(E,r) \longrightarrow \mathit{HF}^*(E,r), \\
& \nabla(f x) = f \nabla x + (\partial_q f) x \quad \text{for $f \in \bK$.}
\end{aligned}
\end{equation}
Changing the bounding cochain by adding a cocycle representing some class in $\mathit{HF}^*(E,1)$ has the effect of subtracting the Lie action of that class (in the sense of Proposition \ref{th:operations-1b}) from \eqref{eq:nabla-r}.
\end{prop}

The theory producing all these results is modelled on an abstract notion of TQFT for surfaces equipped with framed $\mathit{PSL}_2(\bR)$-connections (``framed'' means that we consider only connections on the trivial bundle, and do not quotient out by the gauge group). Here, $\mathit{PSL}_2(\bR)$ is thought of as acting on the base of \eqref{eq:lefschetz} by hyperbolic isometries. Of course, this action cannot be lifted to symplectic automorphisms of the total space, but lifts exist infinitesimally if one only considers the Lefschetz fibration near infinity, and that is sufficient for our purpose. More precisely, the products and Lie brackets mentioned above are all constructed using flat $\mathit{PSL}_2(\bR)$-connections. More generally, one can allow connections whose curvature belongs to the nonnegative cone inside $\mathfrak{sl}_2(\bR)$. Nonnegatively curved connections are required to produce the maps \eqref{eq:continuation}, \eqref{eq:continuation-2}, as well as the classes \eqref{eq:e-}, \eqref{eq:k-k} (and, naturally in view of the latter, play an important role in Proposition \ref{th:diff-q}).

%
At this point, we want to make up for a slight lack of precision in the discussion above. There are many choices involved in setting up $\mathit{HF}^*(E,r)$, but the essential one (other than $r$, of course) is that of a hyperbolic element $g \in \mathit{PSL}_2(\bR)$. One uses the action of $g$ on the disc to determine the structure at infinity of a Hamiltonian automorphism of $E$ (and then, $r$ is used to choose an extension of that automorphism over all of $E$). Different choices of $g$ give rise to isomorphic Floer cohomology groups. However, these isomorphisms are generally unique only up to composition with powers of $\Sigma$ (it doesn't matter which side the composition is taken on, since the isomorphisms are $\Sigma$-equivariant). As an elementary illustration, let's take a second look at the $r = 0$ case. A more precise statement of \eqref{eq:floer-0} would require that $w$ should lie in the interval in $\partial \bar{B}$ bounded by the repelling fixed point of $g$ (on the left) and its attracting fixed point (on the right). For two choices of $g$ within the same one-parameter semigroup (or equivalently, which have the same attracting and repelling fixed points), the resulting relative cohomology groups $H^*(\bar{E},\bar{E}_w)$ can be identified canonically; otherwise, one needs to make a choice, of how to move one $w$ point to the other (or more intrinsically, of a path inside the hyperbolic locus connecting the two choices of $g$). In fact, this is also the situation for general $r$. We will prove that the algebraic structures from Proposition \ref{th:operations-1} and \ref{th:operations-1b} are compatible with those isomorphisms, and a similar argument would apply to Proposition \ref{th:diff-q}. In the interest of brevity, we will not address the parallel uniqueness issue for the structures from Propositions \ref{th:operations-2} and \ref{th:operations-3}, since those are less central for our intended applications.

\begin{remark} \label{th:open-q}
The way in which the various pieces fit together (together with the parallel situation in homological algebra \cite[Section 2f]{seidel14b}) suggests the following possible further developments (none of which will be attempted in this paper).

(i) The relation between Proposition \ref{th:operations-2} and the module structures from Proposition \ref{th:operations-1b} seems to indicate that the algebra structure \eqref{eq:bigraded-1} ought to be extended to $s \in \bZ$.

(ii) There might be an extension of \eqref{eq:bigraded-lie} which would add an $s = -1$ term of the form $\mathit{HF}^*(E,0)[[u]]$, for $u$ a formal variable of degree $2$.

(iii) Take $r>0$. By using a suitably chosen chain level representative of $\Sigma$, one can define the algebraic mapping cone of $\mathit{id}-\Sigma$, which is a graded space $S^*(E,r)$ fitting into a long exact sequence
\begin{equation}
\cdots \rightarrow S^*(E,r) \longrightarrow \mathit{HF}^*(E,r) \xrightarrow{\mathit{id}-\Sigma} \mathit{HF}^*(E,r) \rightarrow \cdots
\end{equation}
Extend this to $r = 0$ by setting $S^*(E,0) = \mathit{HF}^*(E,0)$. Then, the expectation is that
\begin{equation} \label{eq:k-space}
\bigoplus_{s \geq -1} S^*(E,s+1)
\end{equation}
carries the structure of a bigraded Batalin-Vilkovisky (BV) algebra. The Lie structure from Proposition \ref{th:operations-3}, with its extension from (ii), would be thought of as the ``cyclic'' counterpart of the bracket on \eqref{eq:k-space} (compare \cite[Remark 2.10]{seidel14b}). The maps $S^*(E,r) \rightarrow \mathit{HF}^*(E,r)$ would form an algebra homomorphism from \eqref{eq:k-space} to (an extended version of) \eqref{eq:bigraded-1}.
%
\end{remark}

\subsection{Relation with symplectic cohomology}
Symplectic cohomology is a well-known invariant of noncompact symplectic manifolds, which can be easily adapted to our situation \eqref{eq:lefschetz}.

\begin{prop} \label{th:limit}
In the limit $r \rightarrow \infty$, the maps \eqref{eq:continuation} recover the symplectic cohomology of $E$:
\begin{equation} \label{eq:direct-lim}
\underrightarrow{\lim}_r\,\mathit{HF}^*(E,r) \iso \mathit{SH}^*(E).
\end{equation}
\end{prop}

We can describe the situation a little more precisely, using versions of symplectic cohomology with ``finite slope''. These are groups
\begin{equation} \label{eq:floer-cohomology-2}
\mathit{HF}^*(E,r), \;\; r \in \bR \setminus \bZ.
\end{equation}
As the notation indicates, they fit in with \eqref{eq:floer-cohomology}, in the sense that there are continuation maps which increase $r$, and can mix the integer and non-integer cases (we will denote all such maps by $C$). By definition, symplectic cohomology is the direct limit of $\mathit{HF}^*(E,r)$ using only non-integer values of $r$. One proves Proposition \ref{th:limit} by interleaving this with the integer values. Another familiar property of \eqref{eq:floer-cohomology-2} is that it comes with a natural map $H^*(E;\bK) \rightarrow \mathit{HF}^*(E,r)$ for positive $r$. If we take $r \in (0,1)$, then composing this map with a continuation map $\mathit{HF}^*(E,r) \rightarrow \mathit{HF}^*(E,1)$ recovers the map from \eqref{eq:b-sequence}.

\begin{remark} \label{th:relate-operations}
It is well-known (see e.g.\ \cite{seidel07, ritter10}) that symplectic cohomology carries the structure of a BV algebra (this notion has already been mentioned in Remark \ref{th:open-q}). The basic operations are: the (graded commutative) product; and the BV operator $\Delta$, of degree $-1$. From them, one constructs a bracket of degree $-1$,
\begin{equation}
[x_1,x_2] = \Delta(x_1 \cdot x_2) - (\Delta x_1) \cdot x_2 - (-1)^{|x_1|} x_1 \cdot (\Delta x_2),
\end{equation}
which yields a Gerstenhaber algebra structure. It is natural to ask how \eqref{eq:direct-lim} relates the operations on \eqref{eq:floer-cohomology} with those in $\mathit{SH}^*(E)$. For instance, it is easy to show that \eqref{eq:continuation-2} turns into $\Delta$ in the limit (see Proposition \ref{th:mixed-continuation-2}(iv) for a sketch of the geometry underlying that argument). While we will not pursue such questions further here, it may be worthwhile summarizing what one can reasonably expect:

(i) The map $\mathit{HF}^*(E,1) \rightarrow \mathit{SH}^*(E)$ should be a homomorphism of Gerstenhaber algebras. 

(ii) For the Gerstenhaber module structure from Proposition \ref{th:operations-1b}, the expected behaviour is slightly more complicated. The product is still compatible with that on symplectic cohomology, but the bracket should fit into a commutative diagram
\begin{equation} \label{eq:lie-lie}
\xymatrix{
\mathit{HF}^*(E,1) \otimes \mathit{HF}^*(E,r) \ar[rr]^-{[\cdot,\cdot]} 
\ar[d]
&& \mathit{HF}^{*-1}(E,r) \ar[d] \\
\mathit{SH}^*(E) \otimes \mathit{SH}^*(E) \ar[rr]^-{[\cdot,\cdot]_c} && \mathit{SH}^*(E)
}
\end{equation}
where the bottom $\rightarrow$ is the following operation on symplectic cohomology, for $c = r-1$:
\begin{equation} \label{eq:two-brackets}
[x_1,x_2]_c = [x_1,x_2] - c (\Delta x_1) \cdot x_2 = \Delta(x_1 \cdot x_2) - (c+1) (\Delta x_1) \cdot x_2 - (-1)^{|x_1|} x_1 \cdot (\Delta x_2).
\end{equation}
We will see the same phenomenon in Proposition \ref{th:two-connections} below (which is indeed closely related).

(iii) One expects similar compatibility statements for Propositions \ref{th:operations-2} and (less obviously) \ref{th:operations-3}.
\end{remark}

For our intended applications, the most important aspect of the relation between \eqref{eq:floer-cohomology} and \eqref{eq:floer-cohomology-2} concerns connections. Let's first summarize the outcome of the construction from \cite{seidel16}:

\begin{prop} \label{th:elliptic-connection}
Fix $r_1 \in \bR^{>0} \setminus \bZ$ and $r_2 \in \bR \setminus \bZ$, such that $r_1 + r_2 \notin \bZ$.
Suppose that the image of $q^{-1}[\omega_E]$ in $\mathit{HF}^2(E,r_1)$ vanishes, and choose a bounding cochain. This gives rise to a family of ``connections'' parametrized by $c \in \bK$, which are maps
\begin{equation} \label{eq:nabla-c}
\begin{aligned}
& \nabla_c: \mathit{HF}^*(E,r_2) \longrightarrow \mathit{HF}^*(E,r_1+r_2), \\
& \nabla_c( fx) = f \nabla_c x + (\partial_q f) C(x),
\end{aligned}
\end{equation} 
related to each other as follows. Associated to the bounding cochain is an $a \in \mathit{HF}^0(E,r_1)$, and 
\begin{equation}
\nabla_c x = \nabla_1 x + (c-1)\, a \cdot x,
\end{equation}
where $a \cdot x$ is the pair-of-pants product $\mathit{HF}^*(E,r_1) \otimes \mathit{HF}^*(E,r_2) \rightarrow \mathit{HF}^*(E,r_1+r_2)$. The maps \eqref{eq:nabla-c} commute with the continuation maps that increase $r_2$. Hence, in the limit, they induce a family of actual connections on $\mathit{SH}^*(E)$.
\end{prop}

To relate this to Proposition \ref{th:diff-q}, one needs to restrict attention to integer values of the parameter $c$. We will show the following (which is only one of the possible ways of formulating the relationship, and not necessarily the most conceptually satisfying one, but sufficient for our purpose):

\begin{prop} \label{th:two-connections}
Suppose that the assumption of Proposition \ref{th:diff-q} holds; then, so does that of Proposition \ref{th:elliptic-connection}, for any $r_1>1$. Choose $r_1 \in (1,2)$ and $r_2 \notin \bZ$, such that $r_1+r_2 \notin \bZ$ and the interval $(r_2,r_1+r_2)$ contains exactly one integer $r \in \bZ$. Then, for suitably correlated choices of bounding cochains, \eqref{eq:nabla-c} for $c = r-1$ factors through \eqref{eq:nabla-r}, as follows:
\begin{equation} \label{eq:two-connections}
\xymatrix{
\ar@/_1.5pc/[rrr]_-{\nabla_c} 
\mathit{HF}^*(E,r_2) \ar[r]^-{C} & \mathit{HF}^*(E,r) 
\ar[r]^-{\nabla}
& \ar[r]^-{C} \mathit{HF}^*(E,r) & 
\mathit{HF}^*(E,r_1+r_2).
}
\end{equation}
\end{prop}

%
\section{$\mathfrak{sl}_2(\bR)$\label{sec:sl2}}
The Lie algebras which occur naturally in symplectic geometry are those of Hamiltonian vector fields, which are complicated even in two dimensions. Fortunately, to see the structures emerging in our particular application, it is enough consider a finite-dimensional Lie subalgebra. This section contains the necessary elementary material, much of it borrowed from \cite{lalonde-mcduff97}.

\subsection{Nonnegativity and the rotation number}
Set $G = \mathit{PSL}_2(\bR)$ and $\frakg = \mathfrak{sl}_2(\bR)$.  Call $\gamma \in \frakg$ nonnegative if the associated quadratic form is positive semidefinite:
\begin{equation} \label{eq:quadratic-form}
\mathrm{det}(v, \gamma v) \geq 0 \quad \text{for all $v \in \bR^2$,}
\end{equation}
where $(v, \gamma v)$ is the matrix with columns $v$ and $\gamma v$; an equivalent notation would be $v \wedge \gamma v \geq 0$. Concretely, nonnegative $\gamma$ are of the form
\begin{equation} \label{eq:abc}
\gamma = \left( \begin{smallmatrix} \epsilon & \delta-\alpha \\ \delta+\alpha & -\epsilon \end{smallmatrix} \right), \quad
\alpha \geq \sqrt{\delta^2+\epsilon^2}.
\end{equation}
If we think of $G$ as acting on $\bR P^1$, an element of $\frakg$ is nonnegative iff the associated vector field does not point in clockwise direction anywhere. Write $\frakg_{\geq 0} \subset \frakg$ for the cone of nonnegative elements, which is invariant under the adjoint action. Similarly, we call $\gamma$ positive if \eqref{eq:quadratic-form} is positive definite; those elements constitute the interior $\frakg_{>0} \subset \frakg_{\geq 0}$. Correspondingly, there is the opposite cone $\frakg_{\geq 0} = -\frakg_{\leq 0}$ of nonpositive elements, whose interior consists of the negative ones. 

There is a related notion for parabolic elements $g \in G$ (the terminology in this paper is that a nontrivial element of $G$ is either elliptic, parabolic, or hyperbolic; the identity $\Id$ is considered to be neither). Namely, given such an element, choose the preimage in $\mathit{SL}_2(\bR)$ with $\mathrm{tr}(g) = 2$. We say that $g$ is a nonnegative parabolic if $\mathrm{det}(v,gv) \geq 0$ for all $v$; otherwise, it is a nonpositive parabolic (and those are precisely the two parabolic conjugacy classes in $G$). Equivalently, a parabolic $g$ is nonnegative if it can be written as $g = e^{\gamma}$ for some nonzero nilpotent $\gamma \in \frakg_{\geq 0}$, and similarly for the other sign.

A path $g(t) \in G$ is called nonnegative if
\begin{equation} \label{eq:nonnegative-path}
g'(t) g(t)^{-1} \in \frakg_{\geq 0}, \quad \text{or equivalently} \quad g(t)^{-1} g'(t) \in \frakg_{\geq 0}.
\end{equation}
Geometrically, a path in $G$ is nonnegative iff along it, points on $\bR P^1$ never move clockwise. If $g_1(t)$ and $g_2(t)$ are nonnegative paths, so is their pointwise product, since
\begin{equation} \label{eq:make-positive}
g_0(t) = g_1(t)g_2(t) \;\; \Longrightarrow \;\;
g_0'(t) g_0(t)^{-1} = g_1'(t) g_1(t)^{-1} + g_1(t) \big( g_2'(t) g_2(t)^{-1} \big) g_1(t)^{-1}.
\end{equation}
In parallel, we have notions of positive, nonpositive, and negative paths. If $g(t)$ is nonnegative, both $g(t)^{-1}$ and $g(-t)$ are nonpositive, while $g(-t)^{-1}$ is again nonnegative. If $g_1(t)$ is nonnegative, and $g_2(t)$ positive, their product is positive, by \eqref{eq:make-positive}. In particular, any nonnegative path can be perturbed to a positive one, by multiplying it with $e^{t\gamma}$ for some small positive $\gamma$.

Let $\tilde{G} \rightarrow G$ be the universal cover. A preimage in $\tilde{G}$ of a given $g \in G$ is the same as a lift of the action of $g$ to $\bR \rightarrow \bR/\pi \bZ = \bR P^1$. Any $\tilde{g} \in \tilde{G}$ has a rotation number, defined in terms of the action on $\bR$ by
\begin{equation} \label{eq:rotation}
\mathrm{rot}(\tilde{g}) = {\textstyle \lim_{k \rightarrow \infty}}\; \frac{\tilde{g}^k(x) - x}{\pi k}
\quad \text{for any $x \in \bR$.}
\end{equation}
If one composes $\tilde{g}$ with an element in the central subgroup $\mathit{ker}(\tilde{G} \rightarrow G) \iso \bZ$, the rotation number changes by adding that same integer. The other general properties of rotation numbers are: they are continuous; conjugation invariant; homogeneous, meaning that
\begin{equation} \label{eq:homogeneity}
\mathrm{rot}(\tilde{g}^k) = k\, \mathrm{rot}(\tilde{g}), \quad k \in \bZ;
\end{equation}
they can't decrease along nonnegative paths; and they satisfy the quasimorphism property
\begin{equation} \label{eq:quasimorphism}
|\mathrm{rot}(\tilde{g}_1\tilde{g}_2) - \mathrm{rot}(\tilde{g}_1) - \mathrm{rot}(\tilde{g}_2)| \leq 1.
\end{equation}
One can break down the situation into conjugacy types, as follows. Suppose that $\tilde{g}$ is the lift of an elliptic $g \in G$. Then, $g$ is a rotation with angle $\theta$, where $\theta \in (\bR/\pi\bZ) \setminus \{0\}$ (in our terminology, ``rotations'' can be with respect to any oriented basis of $\bR^2$, hence are not necessarily orthogonal). The choice of $\tilde{g}$ corresponds to a lift $\tilde{\theta} \in \bR$, and the rotation number is $\tilde{\theta}/\pi \in \bR \setminus \bZ$. In the other (hyperbolic or parabolic) case, the rotation number is an integer, and can be computed as follows. Fix an eigenvector of $g$. The choice of $\tilde{g}$ determines a homotopy class of paths from the identity to $g$. The image of the eigenvector along such a path yields a loop in $\bR P^1$, whose degree is the rotation number. We would like to state two consequences of this discussion. For the first one, note that the map $\tilde{G} \rightarrow G$ factors through $\mathit{SL}_2(\bR)$. Case-by-case analysis then directly shows that:

\begin{lemma}
If $g \in \mathit{SL}_2(\bR)$ is the image of $\tilde{g} \in \tilde{G}$, then
\begin{equation} \label{eq:parity}
\left\{
\begin{aligned}
& \mathrm{tr}(g) > 0\; \Longleftrightarrow\; \mathrm{rot}(\tilde{g}) \in 2\bZ + (-\half,\half), \\
& \mathrm{tr}(g) < 0\; \Longleftrightarrow\; \mathrm{rot}(\tilde{g}) \in 2\bZ + 1 + (-\half,\half).
\end{aligned}
\right.
\end{equation}
\end{lemma}

For our second observation, let's rewrite \eqref{eq:quasimorphism} more symmetrically as follows:
\begin{equation} \label{eq:triple-trivial}
\tilde{g}_0\tilde{g}_1\tilde{g}_2 = \Id \quad \Longrightarrow \quad
|\mathrm{rot}(\tilde{g}_0) + \mathrm{rot}(\tilde{g}_1) + \mathrm{rot}(\tilde{g}_2)| \leq 1.
\end{equation}

\begin{lemma} \label{th:sharpened-inequality}
(i) In the situation of \eqref{eq:triple-trivial}, suppose that the image $g_k$ of some $\tilde{g}_k$ is elliptic. Then the inequality is strict.

(ii) In the same situation, suppose that $g_k$ is a nonnegative parabolic. Then
\begin{equation}
\mathrm{rot}(\tilde{g}_0) + \mathrm{rot}(\tilde{g}_1) + \mathrm{rot}(\tilde{g}_2) \leq 0.
\end{equation}

(iii) Correspondingly, if $g_k$ is a nonpositive parabolic, one has a lower bound of $0$ for the sum of rotation numbers.
\end{lemma}

\begin{proof}
(i) Without loss of generality, one can assume that $k = 0$ and $\mathrm{rot}(\tilde{g}_0) \in (0,1)$. Then there is a nonnegative path from the identity to $g_0$, which lifts to a path with endpoint $\tilde{g}_0$. Therefore, 
\begin{equation}
\mathrm{rot}(\tilde{g}_0) + \mathrm{rot}(\tilde{g}_1) + \mathrm{rot}(\tilde{g}_2) \leq 
\mathrm{rot}(\tilde{g}_0) + 
\mathrm{rot}(\tilde{g}_0\tilde{g}_1) + \mathrm{rot}(\tilde{g}_2) = \mathrm{rot}(\tilde{g}_0) < 1.
\end{equation}
Passing to inverses yields the corresponing lower bound.

(ii) has the same proof, except that this time, a nonnegative path from the identity to $\tilde{g}_0$ exists for $\mathrm{rot}(\tilde{g}_0) = 0$. Again, one passes to inverses to obtain (iii).
\end{proof}

\begin{example} \label{th:example-path}
Let $g = \exp(\gamma) \in G$, with $\gamma \in \frakg$ sufficiently small. Then, for any integer $r \geq 1$,
\begin{equation}
g(t) = \big( \begin{smallmatrix} \cos(r \pi t) & -\sin(r \pi t) \\ \sin(r \pi t) & \cos(r \pi t) \end{smallmatrix} \big)\exp(t \gamma) 
\end{equation}
is a positive path from the identity to $g$, because of \eqref{eq:make-positive}. The lift $\tilde{g} \in \tilde{G}$ of $g$ associated to that path differs from the obvious lift of $g$ (close to the identity) by $r \in \bZ \subset \tilde{G}$; hence, $\mathrm{rot}(\tilde{g}) \approx r$, with equality unless $g$ is elliptic.
\end{example}

\begin{example} \label{th:three-classes}
Fix hyperbolic conjugacy classes $C_i$ ($i = 0,1,2$). Up to simultaneous conjugation, there is exactly one triple $(g_0,g_1,g_2)$ with $g_i \in C_i$, $g_0 = g_1g_2$, and such that if we take lifts $\tilde{g}_i$ satisfying $\tilde{g}_0 = \tilde{g}_1\tilde{g}_2$, then 
\begin{equation}\label{eq:-1rot}
\mathrm{rot}(\tilde{g}_0) = \mathrm{rot}(\tilde{g}_1) + \mathrm{rot}(\tilde{g}_2) - 1.
\end{equation}
To see that, let's suppose that our conjugacy classes are determined by $\mathrm{tr}(\tilde{g}_i) = \pm( \lambda_i + \lambda_i^{-1})$ for some $\lambda_i > 1$. By \eqref{eq:parity}, if we choose representatives in $\mathit{SL}_2(\bR)$ such that $\mathrm{tr}(g_1) = \lambda_1 + \lambda_1^{-1}$ and $\mathrm{tr}(g_2) = \lambda_2 + \lambda_2^{-1}$, then $\mathrm{tr}(g_0) = -\lambda_0 - \lambda_0^{-1}$. With that in mind, after common conjugation, the only solutions are
\begin{equation} \label{eq:g1g2}
g_1 = \begin{pmatrix} \lambda_1 & 0 \\ 0 &\lambda_1^{-1} \end{pmatrix},
\quad
g_2 = \begin{pmatrix} \lambda_2 - t & \pm (\lambda_2-\lambda_2^{-1} - t) \\ \pm t & \lambda_2^{-1} + t \end{pmatrix}, 
\text{ where }
t = \frac{\lambda_0 + \lambda_0^{-1} + \lambda_1\lambda_2 + \lambda_1^{-1}\lambda_2^{-1}}{\lambda_1 - \lambda_1^{-1}}.
\end{equation}
By looking at the case where $\lambda_1,\lambda_2 \approx 1$ and deforming the matrices $g_1$, $g_2$ to the identity, one sees that the solution of \eqref{eq:g1g2} with the $-$ sign satisfies \eqref{eq:-1rot} (the solution with the $+$ sign satisfies the version of \eqref{eq:-1rot} with a $+1$ instead of $-1$).
\end{example}

\begin{example} \label{th:2-elliptic}
Consider elements $\tilde{g}_i$ of elliptic conjugacy classes $\tilde{C}_i \subset \tilde{G}$ ($i = 1,2$), with 
\begin{equation}
\mathrm{rot}(\tilde{g}_1) \in (0,1), \;\; \mathrm{rot}(\tilde{g}_2) \in (0,1), \;\;
\mathrm{rot}(\tilde{g}_1) + \mathrm{rot}(\tilde{g}_2) > 1.
\end{equation}
For the product $\tilde{g} = \tilde{g}_1\tilde{g}_2$ and its image $g = g_1g_2$ in $G$, the following is a complete list of possibilities:
\begin{equation} \label{eq:3-possi}
\left\{
\begin{aligned}
& \text{$g$ can be elliptic, with any value $\mathrm{rot}(\tilde{g}) \in (1,\mathrm{rot}(\tilde{g}_1)+\mathrm{rot}(\tilde{g}_2)]$;} \\
& \text{$g$ can be nonnegative parabolic, and $\mathrm{rot}(\tilde{g}) = 1$;} \\
& \text{$g$ can lie in any hyperbolic conjugacy class, and $\mathrm{rot}(\tilde{g}) = 1$.}
\end{aligned}
\right.
\end{equation}
Moreover, each possibility happens in a way which is unique up to overall conjugation. To see that, we will carry out computations in the associated conjugacy classes in $\mathit{SL}_2(\bR)$, which correspond to rotations with angles $\theta_i = \pi \mathrm{rot}(\tilde{g}_i) \in (0,\pi)$. The conjugacy class of rotations with angle $\theta_1$ can be identified with the open unit disc in $\bC$, by writing
\begin{equation}
g_1 = \begin{pmatrix} \cos(\theta_1) - \sin(\theta_1) \frac{2\, \mathrm{im}(w)}{1-|w|^2} &
-\sin(\theta_1) \frac{|1+w|^2}{1-|w|^2} \\
\sin(\theta_1) \frac{|1-w|^2}{1-|w|^2} &
\cos(\theta_1) + \sin(\theta_1) \frac{2\, \mathrm{im}(w)}{1-|w|^2} \end{pmatrix}, \;\; |w| < 1.
\end{equation}
Conjugation with standard (Euclidean) rotations acts by rotating $w$. If $g_2$ is the standard rotation with angle $\theta_2$, we have
\begin{equation} \label{eq:g1g2-trace}
\mathrm{tr}(g_1 g_2) = 2 \cos(\theta_1)\cos(\theta_2) - 2 \sin(\theta_1)\sin(\theta_2) {\textstyle \frac{1+|w|^2}{1-|w|^2}}.
\end{equation}
The values of \eqref{eq:g1g2-trace} are $(-\infty, 2\cos(\theta_1+\theta_2)]$, each being reached exactly once. The maximum is reached when $g_1$ is also a standard rotation ($w = 0$), so that $g_1g_2$ has angle $\theta_1+\theta_2>\pi$; from there, as the trace decreases, so does the angle. The case when $\mathrm{tr}(g_1g_2) = -2$ happens as a limit of rotations whose angle approaches $\pi$ from above. Since $g_1g_2$ can't be $-\Id$, it must therefore be a nonnegative parabolic. We know (from the existence of suitable nonnegative paths) that $\mathrm{rot}(\tilde{g}_i) \leq \mathrm{rot}(\tilde{g}_1\tilde{g}_2) \leq \mathrm{rot}(\tilde{g}_i) + 1$. This, together with \eqref{eq:parity}, implies the statements about rotation numbers made in \eqref{eq:3-possi}.
\end{example}

\subsection{Krein theory}
Let's consider the behaviour of nonnegative paths on the level of conjugacy classes. In the elliptic locus, motion can go in one direction only (Krein's Lemma). Namely, for $\gamma \in \frakg_{\geq 0}$ written as in \eqref{eq:abc}, 
\begin{equation} \label{eq:trace-decreases}
g = \big(\begin{smallmatrix} \cos(\theta) & -\!\sin(\theta)\! \\ \sin(\theta) & \cos(\theta) \end{smallmatrix} \big) \text{ with $\theta \in (0,\pi)$} \quad \Longrightarrow \quad
\mathrm{tr}(g \gamma) = -2\alpha \sin(\theta) \leq 0,
\end{equation}
so the angle of rotation $\theta$ can't decrease along a nonnegative path (and its derivative vanishes exactly when that of the path vanishes). There is no comparable constraint in the hyperbolic locus: 
\begin{equation} \label{eq:trace-moves}
g = \big(\begin{smallmatrix} \lambda & 0 \\ 0 & \lambda^{-1}\end{smallmatrix}\big) \text{ with $\lambda>1$} \quad \Longrightarrow \quad
\mathrm{tr}(g \gamma) = (\lambda - \lambda^{-1}) \epsilon,
\end{equation}
which can have either sign. The two parabolic conjugacy classes are each permeable in one direction (from elliptic to hyperbolic or vice versa) only. This is again elementary: 
\begin{equation}
g =  \left(\begin{smallmatrix} 1 & 0 \\ \pm 1 & 1 \end{smallmatrix}\right) \quad \Longrightarrow \quad
\pm \mathrm{tr}(g \gamma) = \delta - \alpha \leq 0
\end{equation}
(in this case, $\mathrm{tr}(g\gamma) = 0$ iff $\gamma = \left(\begin{smallmatrix} 0 & 0 \\ 2\alpha & 0 \end{smallmatrix} \right)$; hence, a nonnegative path intersects the parabolic locus either transversally, or tangentially to the one-parameter subgroup at the intersection point). Finally, if a nonnegative path goes through the identity in $G$, it must leave that element through either nonnegative parabolic elements, or elliptic elements which are small positive rotations (and conversely, it must arrive through nonpositive parabolic elements, or elliptic elements which are small negative rotations). These observations are summarized in Figure \ref{fig:traffic}. 
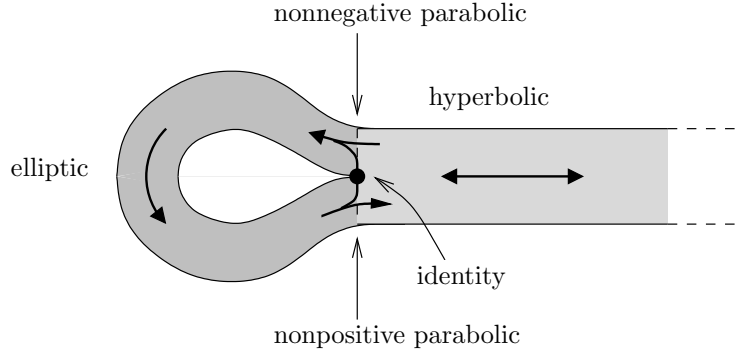
\begin{figure}
\begin{centering}
\hspace{-3em} 
\begin{picture}(0,0)%
\includegraphics{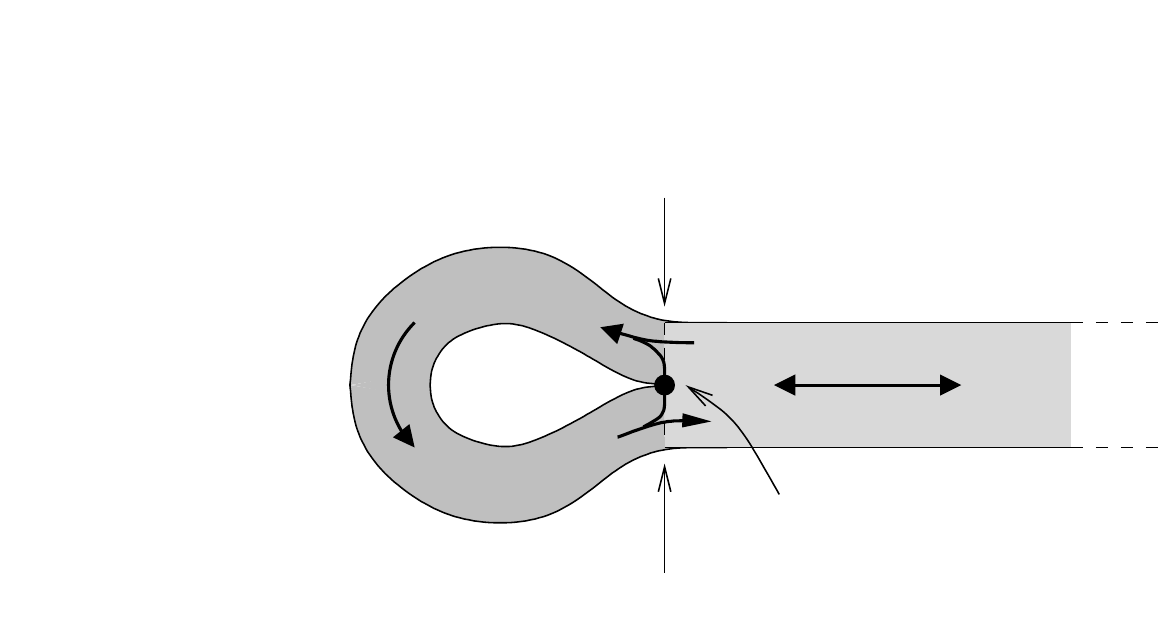}%
\end{picture}%
\setlength{\unitlength}{3947sp}%
\begingroup\makeatletter\ifx\SetFigFont\undefined%
\gdef\SetFigFont#1#2#3#4#5{%
  \reset@font\fontsize{#1}{#2pt}%
  \fontfamily{#3}\fontseries{#4}\fontshape{#5}%
  \selectfont}%
\fi\endgroup%
\begin{picture}(5602,2958)(-39,-1930)
\put(3526,-1486){\makebox(0,0)[lb]{\smash{{\SetFigFont{10}{12.0}{\familydefault}{\mddefault}{\updefault}{identity}%
}}}}
\put(2626,164){\makebox(0,0)[lb]{\smash{{\SetFigFont{10}{12.0}{\familydefault}{\mddefault}{\updefault}{nonnegative parabolic}%
}}}}
\put(976,-811){\makebox(0,0)[lb]{\smash{{\SetFigFont{10}{12.0}{\familydefault}{\mddefault}{\updefault}{elliptic}%
}}}}
\put(3601,-361){\makebox(0,0)[lb]{\smash{{\SetFigFont{10}{12.0}{\familydefault}{\mddefault}{\updefault}{hyperbolic}%
}}}}
\put(2626,-1861){\makebox(0,0)[lb]{\smash{{\SetFigFont{10}{12.0}{\familydefault}{\mddefault}{\updefault}{nonpositive parabolic}%
}}}}
\end{picture}%
\hspace{5em}
\caption{Traffic flow in nonnegative direction on conjugacy classes (in reality, the space of conjugacy classes is one-dimensional and non-Hausdorff; we've drawn it thickened for the sake of legibility).\label{fig:traffic}}
\end{centering}
\end{figure}

Our next topic is the converse question, of constructing nonnegative paths with prescribed motion on conjugacy classes. We will need a technical refinement of the computations from \eqref{eq:trace-decreases} and \eqref{eq:trace-moves}. Writing $g = \left( \begin{smallmatrix} a & c \\ b & d \end{smallmatrix} \right)\in \mathit{SL}_2(\bR)$, the condition $\mathrm{det}(g) = 1$ is equivalent to
\begin{equation} \label{eq:det-1}
(a-d)^2 + (c+b)^2 = \mathrm{tr}(g)^2 - 4 + (c-b)^2.
\end{equation}
As a consequence, if we consider a subset of elements with fixed trace, then $T(g) = c-b$ is a proper function on that subset. Note also that $T(g^2) = T(g)\mathrm{tr}(g)$.

\begin{lemma} \label{th:bound-vector-fields}
(i) Let $C_{\mathit{ell},\theta} \subset \mathit{SL}_2(\bR)$ be the conjugacy class of rotations with angle $\theta \in (0,\pi)$. There is a function
\begin{equation} \label{eq:elliptic-gamma-function}
\left\{
\begin{aligned}
& \gamma_{\theta}: C_{\mathit{ell},\theta} \longrightarrow \frakg_{>0}, \\
& |T(g\gamma_{\theta}(g))| \leq A +B|T(g)|, \text{ for some constants $A,B>0$;} \\
& \mathrm{tr}(g \,\gamma_\theta(g)) \text{ is negative and bounded away from $0$.} \\
\end{aligned}
\right.
\end{equation}

(ii) Let $C_{\mathit{hyp},\lambda} \subset \mathit{SL}_2(\bR)$ be the conjugacy class of hyperbolic elements with eigenvalues $\lambda^{\pm 1}$, for some $\lambda>1$. There are functions
\begin{equation}
\left\{
\begin{aligned}
& \gamma_{\lambda,\pm}: C_{\mathit{hyp},\lambda} \longrightarrow \frakg_{>0}, \\
& |T(g\gamma_{\lambda,\pm}(g))| \leq A +B|T(g)|, \text{ for some constants $A,B>0$;} \\
& \pm\mathrm{tr}(g \,\gamma_{\lambda,\pm}(g)) \text{ is positive and bounded away from $0$.} \\
\end{aligned}
\right.
\end{equation}
 \end{lemma}

\begin{proof}
(i) Take
\begin{equation}
\gamma_{\theta}(g) = \frac{g - \cos(\theta) \Id}{\sin(\theta)},
\end{equation}
which is conjugate to $\big(\begin{smallmatrix} 0 & - 1 \\ 1 & 0 \end{smallmatrix} \big)$, hence lies in $\frakg_{>0}$. Then
\begin{align}
& \mathrm{tr}(g\gamma_{\theta}(g)) = -2\sin(\theta), \\
& T(g \gamma_{\theta}(g)) = \frac{\cos(\theta)}{\sin(\theta)} T(g).
\end{align}

(ii) If $r = r(g)$ is the quantity on either side of \eqref{eq:det-1},
\begin{equation} \label{eq:test-gamma}
\gamma_{\lambda,\pm}(g) = \sqrt{r} \big(\begin{smallmatrix} 0 & -1 \\ 1 & 0 \end{smallmatrix}\big) \pm 
\big(\begin{smallmatrix} 
a-d
&
b+c
\\
b+c
&
-(a-d)
\end{smallmatrix}\big) \in \frakg_{\geq 0}
\end{equation}
satisfies
\begin{align}
& \mathrm{tr}(g\gamma_{\lambda, \pm}(g)) = (c-b)\sqrt{r}\, \pm\, r
=  \textstyle r \Big( \mathrm{sign}(c-b) \sqrt{1 - \frac{\mathrm{tr}(g)^2 - 4}{r}} \pm 1 \Big)
 \\ \notag
& \quad \quad \Longrightarrow \pm \mathrm{tr}(g\gamma_{\lambda,\pm}(g)) \geq \half (\mathrm{tr}(g)^2-4), \\
& T(g \gamma_{\lambda,\pm}(g)) = - \mathrm{tr}(g) \sqrt{r} = - \mathrm{tr}(g) \sqrt{\mathrm{tr}(g)^2 - 4 + T(g)^2}. 
\end{align}
Our choice \eqref{eq:test-gamma} always lies on the boundary of the nonnegative cone. However, one can perturb it into the interior $\frakg_{>0}$, and this perturbation can be chosen sufficiently small so as not to disturb its other properties.
\end{proof}

\begin{convention} \label{th:weak-homotopy}
When we talk about weak homotopy equivalences for spaces of paths in $G$ (and, later on, for spaces of connections  as well), this is understood to refer to smooth maps from finite-dimensional compact manifolds (with boundary or corners) to the relevant space. For instance, saying that some space $\scrP$ as in \eqref{eq:path-spaces} is weakly contractible means that for every smooth map $P \rightarrow \scrP$ from a compact manifold $P$, there is a smooth extension $P \times [0,1] \rightarrow \scrP$ which is constant on the other endpoint of the interval. This technical clarification may not be essential, but it slightly simplifies some proofs, and is anyway the natural context for our applications. 
\end{convention}

\begin{lemma} \label{th:path}
(i) Fix a nondecreasing function $\theta(t) \in (0,\pi)$ ($0 \leq t \leq 1$), and a $g_0 \in C_{\mathit{ell},\theta(0)}$. Then, the space of nonnegative paths $g(t)$, with $g(0) = g_0$ and such that $g(t) \in C_{\mathit{ell},\theta(t)}$, is weakly contractible. A corresponding result holds for (strictly) increasing functions and positive paths.

(ii) Fix a function $\lambda(t) > 1$ ($0 \leq t \leq 1$), and a $g_0 \in C_{\mathit{hyp},\lambda(0)}$. Then, the space of nonnegative paths $g(t)$, with $g(0) = g_0$ and such that $g(t) \in C_{\mathit{hyp},\lambda(t)}$, is weakly contractible. The same holds for positive paths.
\end{lemma}

\begin{proof}
(i) Let's suppose that we have a smooth family of functions 
\begin{equation} \label{eq:elliptic-gamma}
\left\{
\begin{aligned}
& \; \gamma_t: C_{\mathit{ell},\theta(t)} \longrightarrow \frakg_{\geq 0}, \\
& \; |T(g\gamma_t(g))| \leq A + B|T(g)|, \text{ for some constants $A,B>0$;}\\
& \; \mathrm{tr}(g \,\gamma_t(g)) = -2\sin(\theta(t))\theta'(t).
\end{aligned}
\right.
\end{equation}
The boundedness condition ensures that the ODE
\begin{equation} \label{eq:ode-elliptic}
g'(t) = g(t)\gamma_t(g(t))
\end{equation}
can be integrated for any initial value $g(0) \in C_{\mathit{ell},\theta(0)}$. By construction, each solution is a nonnegative path with $g(t) \in C_{\mathit{ell},\theta(t)}$.

Functions \eqref{eq:elliptic-gamma} can be constructed from those in Lemma \ref{th:bound-vector-fields}(i) by multiplying with the bounded nonnegative function 
\begin{equation} \label{eq:bounded-fn}
g \longmapsto -2\sin(\theta(t))\theta'(t) \mathrm{tr}(g \gamma_{\theta(t)}(g))^{-1}. 
\end{equation}
On the other hand, the conditions in \eqref{eq:elliptic-gamma} are convex, allowing us to interpolate between two such functions, and to construct them using partitions of unity. In particular, given a nonnegative path $g(t)$ with the desired behaviour, one can choose $\gamma_t$ so that this path is a solution of \eqref{eq:ode-elliptic}. More precisely, this kind of argument shows that the space of paths is weakly homotopy equivalent to the space of all \eqref{eq:elliptic-gamma}, which implies the desired result. In the positive version of the same result, one takes $\gamma_t$ which have values in $\frakg_{>0}$, and those exist because \eqref{eq:bounded-fn} is strictly positive.

The proof of (ii) is parallel, using the corresponding part of Lemma \ref{th:bound-vector-fields}.
\end{proof}

\subsection{Spaces of nonnegative paths}
Using Lemma \ref{th:path} as our basic ingredient, we will now discuss the topology of spaces of nonnegative paths with constraints on the endpoints. The general notation will be as follows: if $U_0,U_1$ are subsets of $G$, 
\begin{equation} \label{eq:path-spaces}
\begin{aligned}
& \scrP(G,U_0,U_1) = \{ g: [0,1] \rightarrow G \;:\; g(0) \in U_0, \, g(1) \in U_1\}, \\
& \scrP_{\geq 0}(G,U_0,U_1) = \text{subspace of nonnegative paths}, \\
& \scrP_{>0}(G,U_0,U_1) = \text{subspace of positive paths.}
\end{aligned}
\end{equation}
One can introduce similar notation $\scrP_{\leq 0}(G,U_0,U_1)$, $\scrP_{<0}(G,U_0,U_1)$ for nonpositive and negative paths. We will use superscripts $\scrP(G,U_0,U_1)^r$ to restrict to paths along which the rotation number increases by exactly $r$; and similarly for inequalities, such as $\scrP(G,U_0,U_1)^{<r}$. 

\begin{lemma} \label{th:pre-short}
Fix hyperbolic conjugacy classes $C_0$ and $C_1$. Then, evaluation at either endpoint yields a weak homotopy equivalence
\begin{equation}
\scrP_{\geq 0}(G,C_0,C_1)^0 \stackrel{\htp}{\longrightarrow} C_i.
\end{equation}
\end{lemma}

\begin{proof}
By Figure \ref{fig:traffic}, any such path remains within the hyperbolic locus, hence falls into the class considered in Lemma \ref{th:path}(ii). The rest is straightforward.
\end{proof}

\begin{lemma} \label{th:pp}
Fix hyperbolic conjugacy classes $C_i$ ($i = 0,\dots,m$, for some $m \geq 2$). Consider $g_i \in C_i$ together with a nonnegative path $h(t)$ from $h(0) = g_1\cdots g_m$ to $h(1) = g_0$, such that if we take lifts of our elements to $\tilde{g}_i \in \tilde{G}$, compatible with the product and path, then
\begin{equation} \label{eq:lose-one}
\mathrm{rot}(\tilde{g}_0) = \mathrm{rot}(\tilde{g}_1) + \cdots + \mathrm{rot}(\tilde{g}_m) + 1-m.
\end{equation}
Any $g_i$ yields a weak homotopy equivalence between the space of all $(g_0,\dots,g_m,h(t))$ and $C_i$.
\end{lemma}

\begin{proof}
Because of \eqref{eq:quasimorphism} and the fact that rotation numbers don't decrease along nonnegative paths, \eqref{eq:lose-one} requires that, for any $i \geq 1$,
\begin{equation} \label{eq:productminus1}
\mathrm{rot}(\tilde{g}_1 \cdots\tilde{g}_i) = \mathrm{rot}(\tilde{g}_1\cdots \tilde{g}_{i-1}) + \mathrm{rot}(\tilde{g}_i) - 1.
\end{equation}
Suppose that there is a smallest $i \geq 2$ for which the product $g_1 \cdots g_i$ is not hyperbolic. By \eqref{eq:productminus1} and Lemma \ref{th:sharpened-inequality}, it would then have to be a nonnegative parabolic. If $i<m$, one would then apply Lemma \ref{th:sharpened-inequality} again, to get
\begin{equation}
\mathrm{rot}(\tilde{g}_1 \cdots \tilde{g}_{i+1}) \geq \mathrm{rot}(\tilde{g}_1 \cdots \tilde{g}_i) + \mathrm{rot}(\tilde{g}_{i+1}),
\end{equation}
which is a contradiction to \eqref{eq:productminus1} for $i+1$. If $i = m$, the path $h$ would have to go from a nonnegative parabolic to a hyperbolic element without raising the rotation number, which is impossible. We have therefore shown that all products $g_1\cdots g_i$ are hyperbolic.

Consider the space of all $(g_1,\dots,g_m)$ such that $g_i \in C_i$, the product $g = g_1\cdots g_m$ is a fixed hyperbolic element, and $\mathrm{rot}(\tilde{g}_1\cdots \tilde{g}_m) = \mathrm{rot}(\tilde{g}_1) + \cdots + \mathrm{rot}(\tilde{g}_m) + 1-m$. Then, that space is contractible. One can derive this from Example \ref{th:three-classes} by induction on $m$, and the hyperbolicity of the intermediate products. We omit the details, since there is a more conceptual alternative, namely to think of it as a special case of Corollary \ref{th:teichmuller1b} below. On the other hand, the space of nonnegative paths $h(t)$ from our fixed hyperbolic $g$ to $C_0$, along which the rotation number remains constant, is weakly contractible by Lemma \ref{th:pre-short}. By taking the product of those spaces, and then letting $g$ vary, one sees that evaluation at $h(0) = g = g_1 \cdots g_m$ gives a weak homotopy equivalence between the space of all $(g_0,g_1,\dots,g_m,h(t))$ and $G_{\mathit{hyp}}$. Since evaluation at any $h(t)$ lands in the hyperbolic locus, we can evaluate at $h(1) = g_0$ instead. By applying simultaneous conjugation, one extends the result to the other evaluation maps.
\end{proof}

One can interpret Lemma \ref{th:pre-short} as follows. Let $\scrP(G_{\mathit{hyp}}, C_0, C_1)$ be the space of all paths from $C_0$ to $C_1$, but which remain inside the hyperbolic locus. $\scrP_{\geq 0}(G, C_0,C_1)^0$ is a subspace of $\scrP(G_{\mathit{hyp}},C_0,C_1)$, and we have shown that the inclusion yields a weak homotopy equivalence
\begin{equation}
\scrP_{\geq 0}(G,C_0,C_1)^0 \stackrel{\htp}{\longrightarrow} \scrP(G_{\mathit{hyp}},C_0,C_1).
\end{equation}
Alternatively, let's think of the case $C_0 = C_1 = C$. In that case, inclusion of the constant paths yields a weak homotopy equivalence
\begin{equation}
C \stackrel{\htp}{\longrightarrow} \scrP_{\geq 0}(G,C,C)^0.
\end{equation}
This second kind of interpretation also applies to Lemma \ref{th:pp}, which turned out to be weakly homotopy equivalent to the subspace for which the path $h(t)$ is constant. 

From now on, when deriving results about path spaces, we will often only give the principal steps of the proof, omitting ``$\epsilon$--$\delta$ level'' details (readers interested in seeing such kind of arguments in full may want to look at \cite{lalonde-mcduff97}).

\begin{lemma} \label{th:short-path}
Fix an elliptic conjugacy class $C_0$, and a hyperbolic conjugacy class $C_1$. Then, evaluation at the hyperbolic endpoint yields a weak homotopy equivalence
\begin{equation}
\scrP_{\geq 0}(G,C_0,C_1)^{<1} \stackrel{\htp}{\longrightarrow} C_1.
\end{equation}
The same applies to $\scrP_{\geq 0}(G,C_1,C_0)^{<1}$, still taking the evaluation map at the hyperbolic endpoint.
\end{lemma}

\begin{proof}[Sketch of proof]
We begin with a slightly different problem. Let's consider positive paths which start in $C_0$, end in $G_{\mathit{hyp}}$, and along which the rotation number increases by less than $1$. Such a path necessarily has the following structure:
\begin{equation}
\text{for some $T \in (0,1)$,} \; \left\{
\begin{aligned}
& \text{$g(t)$ is elliptic for $t < T$;} \\
& \text{$g(T)$ is a nonpositive parabolic;} \\
& \text{$g(t)$ is hyperbolic for $T < t$.}
\end{aligned}
\right.
\end{equation}
If one fixes $T$ and $g(T)$, then by analyzing $g|[0,T-\epsilon]$ and $g|[T+\epsilon,1]$ for small $\epsilon>0$, using Lemma \ref{th:path}, one sees that the resulting space of paths is weakly contractible. 
Allowing $T$ and $g(T)$ to vary yields a path space that is weakly homotopy equivalent to the nonpositive parabolic stratum, hence also to $G$. Instead of using evaluation at $t = T$, one can use the (homotopic) evaluation at $t = 1$.

We can weaken strict positivity to nonnegativity, without changing the weak homotopy type of the space, since any nonnegative path can be perturbed to a positive one (see \eqref{eq:make-positive} and the subsequent discussion); for this to work, it's important that the endpoint condition $g(1) \in G_{\mathit{hyp}}$ is an open one. Next, again without changing the weak homotopy type, we can require that the hyperbolic endpoint should lie in a fixed conjugacy class, thanks to Lemma \ref{th:path}(ii). This brings us to the desired situation. The other direction (from hyperbolic to elliptic) is proved in the same way.
\end{proof}

\begin{lemma} \label{th:short-path-3}
Fix a hyperbolic conjugacy class $C$. Evaluation at the endpoint yields a weak homotopy equivalence
\begin{equation}
\scrP_{\geq 0}(G,\mathit{id},C)^1 \stackrel{\htp}{\longrightarrow} C.
\end{equation}
\end{lemma}

\begin{proof}[Sketch of proof]
Consider first positive paths which start at the identity, end in $G_{\mathit{hyp}}$, and along which the rotation number increases by $1$. Such a path has the property that $g(t)$ is elliptic for all sufficiently small $t>0$. By looking at $g|[\epsilon,1]$ for small $\epsilon>0$, one reduces their study to Lemma \ref{th:short-path}. The translation back to the original situation also follows that model. 
\end{proof}

\begin{lemma} \label{th:short-path-2}
Fix two hyperbolic conjugacy classes $C_0$ and $C_1$. Then, evaluation at both endpoints yields a weak homotopy equivalence
\begin{equation}
\scrP_{\geq 0}(G,C_0,C_1)^1 \stackrel{\htp}{\longrightarrow} C_0 \times C_1.
\end{equation}
\end{lemma}

\begin{proof}[Sketch of proof]
Once more, we start with a slightly different problem. Namely, consider positive paths, such that both endpoints lie $G_{\mathit{hyp}}$, and along which the rotation number increases by $1$. There is exactly one point $T$ where such a path intersects a fixed elliptic conjugacy class. By breaking up the path into pieces $g|[0,T]$ and $g|[T,1]$, one can think of the whole path space as a fibre product of two spaces of the kind considered in Lemma \ref{th:short-path}. Hence, what we are looking at is the fibre product of two hyperbolic conjugacy classes, formed over a contractible space (the elliptic conjugacy class). As in Lemma \ref{th:short-path}, positivity can be weakened to nonnegativity, and one can use Lemma \ref{th:path}(ii) to replace the condition on the endpoints by one of lying in specified hyperbolic conjugacy classes. 
\end{proof}

One can interpret these results as follows: the inclusion of nonnegative paths into all paths yields weak homotopy equivalences
\begin{equation} \label{eq:into-top}
\begin{aligned}
& \left.\begin{aligned}
& \scrP_{\geq 0}(G,C_0,C_1)^{<1} \stackrel\htp\longrightarrow \scrP(G,C_0,C_1)^{<1} \\
& \scrP_{\geq 0}(G,C_1,C_0)^{<1} \stackrel\htp\longrightarrow \scrP(G,C_1,C_0)^{<1}
\end{aligned}\right\}
&& \text{for Lemma \ref{th:short-path},} 
\\
& \left.\scrP_{\geq 0}(G,\mathit{id},C)^1 \stackrel\htp\longrightarrow \scrP(G,\mathit{id},C)^1\right. && \text{for Lemma \ref{th:short-path-3},} \\
& \left.\scrP_{\geq 0}(G,C_0,C_1)^1 \stackrel{\htp}{\longrightarrow} \scrP(G,C_0,C_1)^1\right. && \text{for Lemma \ref{th:short-path-2}.}
\end{aligned}
\end{equation}
We will see a few more instances of such behaviour in Section \ref{sec:elliptic}. Note that this contrasts with Lemmas \ref{th:pre-short} and \ref{th:pp}, where including nonnegative paths into all paths (with the same rotation behaviour) would not have yielded a homotopy equivalence.

\begin{remark}
One can conjecture that statements similar to \eqref{eq:into-top} should hold more generally for spaces of nonnegative paths which are ``sufficiently long'' (meaning that the rotation number increases by a sufficient amount to force transitions between elliptic and hyperbolic elements). This idea is in tune with the results of \cite{lalonde-mcduff97, slimowitz01}. \end{remark}

The next result can be viewed as a weak version of Lemma \ref{th:short-path-3} under extra constraints. Instead of determining the topology of the path space under consideration, we will argue by an explicit construction.

\begin{lemma} \label{th:unit-path}
Fix two hyperbolic conjugacy classes $C_1$ and $C_2$. Then there is a nonnegative path $g_1(t)$, starting at the identity and ending in $C_1$, along which the rotation number increases by $1$; together with another path $g_2(t) \in C_2$, such that at all times, $g_1(t)g_2(t) \in C_2$. Moreover, one can write 
\begin{equation} \label{eq:conjugating-path}
g_1(t)g_2(t) = k(t) g_2(t) k(t)^{-1}
\end{equation}
for some path $k(t)$ starting at the identity, whose endpoint $k(1)$ is hyperbolic and (if one lifts the path to $\tilde{G}$ in the obvious way) has rotation number $0$.
\end{lemma}

\begin{proof}
Let's set $g_2 = \mathrm{diag}(\lambda,\lambda^{-1})$ for some $\lambda>1$. Consider
\begin{equation} \label{eq:trace-h}
H = \{g \in \mathit{SL}_2(\bR)\;:\; \mathrm{tr} (g g_2) = \lambda+\lambda^{-1} \} \stackrel{\mathrm{tr}}{\longrightarrow} \bR.
\end{equation}
$H$ is a connected hyperboloid in $\bR^3$, and the trace is a Morse function on it, with index $1$ critical points at the identity and $g_2^{-2}$, and corresponding critical levels $2$ and $\lambda^2 +\lambda^{-2} > 2$. The open sublevel set $U = \{\mathrm{tr}(g)<2\} \subset S$ consists of two connected components $U_{\pm}$. If we write $g = \left(\begin{smallmatrix} a & c \\ b & d \end{smallmatrix} \right)$, those components are
\begin{equation}
U_{\pm} = \{ g \in H \;:\; {\pm } b > 0, \text{ or equivalently } {\pm } c < 0 \}.
\end{equation}
We construct a path in $H$ of the following kind. Starting at the identity, the path first moves inside the level set $\{\mathrm{tr}(g) = 2\}$ to $\left( \begin{smallmatrix} 1 & 0 \\ 1 & 1 \end{smallmatrix} \right)$. From there, we strictly decrease the trace, moving into $U_+$, and then further inside that subset until $\{\mathrm{tr}(g) = -2\}$. At that point, the resulting matrix is necessarily conjugate to $\left(\begin{smallmatrix} -1 & 0 \\ 1 & -1 \end{smallmatrix} \right)$, because it still belongs to $U_+$. We continue our path downwards into hyperbolic level sets $\{\mathrm{tr}(g)<-2\}$, until we reach the value of the trace (up to sign) prescribed by our choice of $C_1$. One sees easily that the rotation number increases by $1$ along this path.

Denote the path we have just constructed by $g_1(t)$, and take $g_2(t) = g_2$ to be the constant path. 
Since $g_1(t)g_2$ remains in the conjugacy class of $g_2$, one can write $g_1(t)g_2 = k(t) g_2 k(t)^{-1}$, for some smooth $k(t) = \left( \begin{smallmatrix} p(t) & r(t) \\ q(t) & s(t) \end{smallmatrix} \right)$ with $k(0) = \Id$. Then,
\begin{equation} \label{eq:g2-trace}
\begin{aligned}
\mathrm{tr}(g_1(t)) = \mathrm{tr}( k(t) g_2 k(t)^{-1} g_2^{-1}) & = 
2p(t)s(t) - (\lambda + \lambda^{-2}) q(t)r(t)
\\ & =
(1-p(t)s(t)) (\lambda - \lambda^{-1})^2 + 2.
\end{aligned}
\end{equation} 
By construction $\mathrm{tr}(g_1(1)) < -2$, which with \eqref{eq:g2-trace} yields $p(1) s(1) > 1$, hence
\begin{equation}
\mathrm{tr}(k(1))^2 = (p(1)+s(1))^2 = (p(1)-s(1))^2 + 4 p(1)s(1)> 4. 
\end{equation}
This shows that $k(1)$ must be hyperbolic (one could derive the same result more elegantly from Theorem \ref{th:teichmuller}, applied to a punctured torus). Now, bear in mind that $\mathrm{tr}(g_1(t)) \leq 2$ for all $t$. From this and \eqref{eq:g2-trace}, it follows that $p(t) s(t) > 0$ for all $t$. Hence, both entries $p(t)$ and $s(t)$ remain positive throughout. By tracking how points on the circle move under the action of $k(t)$, one easily concludes that the rotation number of $k(1)$ must vanish. 

At this point, it seems that we have been missing an essential condition, namely, that $g_1(t)$ should be a nonnegative path. However, an inspection of each step of the construction (and the use of Lemma \ref{th:path}) shows that $g_1(t)$ is conjugate to a nonnegative path, by a time-dependent conjugation. One can apply the same conjugation to $g_2(t)$ and $k(t)$, and that yields the desired result.
\end{proof}

\begin{remark}
It would be interesting to consider similar questions for nonnegative paths in the larger group $\mathit{Diff}^+(S^1)$ (this idea was already mooted by Abouzaid and Ganatra). Enlarging the group affords one greater freedom, which may simplify the construction of homotopies. On a more fundamental level, the larger group contains many more conjugacy classes: for instance, the pullbacks of hyperbolic elements by finite covers $S^1 \rightarrow S^1$, which would be relevant to Fukaya categories with multiple fibres as ``stops'' \cite{sylvan16}. 
\end{remark}

\section{$\mathfrak{sl}_2(\bR)$-connections on surfaces\label{sec:connections}}
For our TQFT, we will need to consider connections with nonnegative curvature. There is an obvious relation between such connections and nonnegative paths, and (with one exception, Proposition \ref{th:bochner}) all our results will be derived from the elementary considerations in Section \ref{sec:sl2} in that way. On the other hand, if we restrict attention to flat connections, there is a large body of relevant literature; \cite{goldman88} will be particularly important for us.

\subsection{Basic notions\label{subsec:connections}}
Let $S$ be a connected compact oriented surface with nonempty boundary. We consider $\frakg$-connections on the trivial bundle over $S$, meaning operators 
$\nabla_A = d-A$ for some $A \in \Omega^1(S,\frakg)$ (we will often refer to $A$ as the connection). 
Gauge transformations are maps $\Phi: S \rightarrow G$, and act on connections (covariantly) by
\begin{equation}
\Phi_*A = \Phi A \Phi^{-1} + (d\Phi) \Phi^{-1}.
\end{equation}
With our sign convention for $A$, the curvature is 
\begin{equation} \label{eq:curvature}
F_A = -dA + A \wedge A \in \Omega^2(S,\frakg).
\end{equation}
A connection is said to have nonnegative curvature if it satisfies the (gauge invariant) condition
\begin{equation} \label{eq:nonneg-curvature}
F_A(\xi_1,\xi_2) \in \frakg_{\geq 0} \quad \text{for any oriented basis $(\xi_1,\xi_2) \in TS$.}
\end{equation}
Obviously, there is a parallel notion of nonpositive curvature; one can switch from one to the other by changing the orientation of $S$, or by applying a gauge transformation in $\mathit{PGL}_2^-(\bR)$. Similarly, one could consider connections with (strictly) positive or negative curvature; but we have no use for those notions.

Given a connection, parallel transport along some $c: [0,1] \rightarrow S$ yields a path $g:[0,1] \rightarrow G$, 
\begin{equation} \label{eq:parallel-transport}
g(0) = \Id, \quad g'(t)g(t)^{-1} = A_{c(t)}(c'(t)).
\end{equation}
(If $c$ is a closed loop, $g(1)$ is the holonomy of the connection around $c$.) That path has an obvious lift to $\tilde{G}$, with $\tilde{g}(0)$ the identity; from that, we get a rotation number
\begin{equation}
\mathrm{rot}_c(A) \stackrel{\mathrm{def}}{=} \mathrm{rot}(\tilde{g}(1)) \in \bR. 
\end{equation}
Since $G$ is homotopy equivalent to a circle, the homotopy class of a gauge transformation can be written as $[\Phi] \in H^1(S)$. The effect of gauge transformations on rotation numbers along closed loops $c$ is
\begin{equation}
\label{eq:gauge-c}
\mathrm{rot}_{c}(\Phi_*A) = \mathrm{rot}_{c}(A) + \int_{c} [\Phi].
\end{equation}
Let's denote the boundary circles by $\partial_i S$, where $i \in I = \pi_0(\partial S)$. By \eqref{eq:gauge-c}, the following is a gauge invariant quantity:
\begin{equation} \label{eq:boundary-rotation}
\mathrm{rot}_{\partial S}(A) = \sum_i \mathrm{rot}_{\partial_i S}(A).
\end{equation}

\begin{proposition} \label{th:bochner}
Let $A$ be a connection on $S$, such that the holonomies around the boundary circles are hyperbolic.  If the connection is nonnegatively curved, 
\begin{equation} \label{eq:bochner}
\mathrm{rot}_{\partial S}(A) \leq -\chi(S).
\end{equation}
\end{proposition}

\begin{proof}
We consider a slightly different but equivalent situation. Namely, let $S$ be a Riemann surface with tubular ends, carrying a connection $A$ which, in coordinates $(s,t) \in [0,\infty) \times S^1$ on each end, is of the form $a_{i,t} \mathit{dt}$, and has hyperbolic holonomy around the $S^1$ factor. One can define an analogue of \eqref{eq:boundary-rotation}, which we denote by $\mathrm{rot}_{\mathit{ends}}(A) \in \bZ$. Consider the $\bR$-linear Dolbeault operator associated to the connection,
\begin{equation} \label{eq:dbar-e}
\bar\partial_A = \nabla^{0,1}_A: \smooth_{\mathit{cpt}}(S,\bC) \longrightarrow \Omega_{\mathit{cpt}}^{0,1}(S).
\end{equation}
As stated, this applies to compactly supported smooth functions, but in fact one wants to pass to suitable Sobolev completions, let's say from $W^{1,2}$ to $L^2$. The index theorem (see e.g.\ \cite[Section 3.3]{schwarz95}) says that the completed operator is Fredholm, with
\begin{equation}
\mathrm{index}(\bar\partial_A) = \chi(S) + \mathrm{rot}_{\mathit{ends}}(A).
\end{equation}
The Bochner identity, which applies to any $\xi \in W^{1,2}(S,\bC)$, is
\begin{equation}
\int_S |\bar\partial_A \xi|^2 = \int_S |\nabla \xi|^2  + \int_S \mathrm{det}(\xi, F_A \xi).
\end{equation}
If we assume that our connection has nonnegative curvature, it follows that any $W^{1,2}$ solution of $\bar\partial_A \xi = 0$ must be zero, hence the index is necessarily $\leq 0$.
\end{proof}

By reversing orientation, we get the following:

\begin{corollary} \label{th:mw}
If $A$ is a flat connection with hyperbolic boundary holonomies,
\begin{equation}
|\mathrm{rot}_{\partial S}(A)| \leq -\chi(S).
\end{equation}
\end{corollary}

We will be interested in spaces of connections with fixed boundary behaviour. Concretely, given $S$, choose $a_i \in \Omega^1(\partial_i S,\frakg)$ for each $i \in I$, with hyperbolic holonomy. We will use the following associated notation:
\begin{equation} \label{eq:a-space-notation}
\begin{aligned}
& g_i && \text{holonomy of $a_i$}, \\
& C_i && \text{conjugacy class of $g_i$ in $G$}, \\
& \tilde{g}_i && \text{preferred lift of $g_i$ to $\tilde{G}$}, \\
& \tilde{C}_i && \text{conjugacy class of $\tilde{g}_i$ in $\tilde{G}$.}
\end{aligned}
\end{equation}
The relevant spaces of connections are
\begin{equation} \label{eq:moduli-spaces-2}
\begin{aligned}
\scrA(S,\{a_i\}) & = \{ A \in \Omega^1(S,\frakg)\;:\; A|\partial_iS = a_i\}, \\
\scrA_{\geq 0}(S,\{a_i\}) & \;\; \text{subspace of nonnegatively curved connections}, \\
\scrA_{\mathit{flat}}(S,\{a_i\}) & \;\; \text{subspace of flat connections.}
\end{aligned}
\end{equation}
These spaces are acted on by the group 
\begin{equation} \label{eq:gauge-2}
\scrG(S,\{a_i\}) = \{\Phi: S \rightarrow G \;:\; (\Phi|\partial_iS)_* a_i = a_i\} \htp H^1(S,\partial S).
\end{equation}
The weak homotopy equivalence in \eqref{eq:gauge-2} comes from the fact that the stabilizer of each $a_i$ in $\smooth(\partial_i S,G)$ is a copy of $\bR$, hence contractible. In the case of flat connections, the quotient
\begin{equation} \label{eq:moduli-1}
\scrM_{\mathit{flat}}(S,\{a_i\}) = \scrA_{\mathit{flat}}(S,\{a_i\})/\scrG(S,\{a_i\})
\end{equation}
has a well-known topological interpretation:
\begin{equation}
\label{eq:rep-spaces}
\scrM_{\mathit{flat}}(S,\{a_i\}) \iso \big\{ \sigma: \pi_1(S)^{\mathit{op}} \rightarrow G \;:\; \sigma(\partial_iS) \in C_i, \text{ and } e_{\mathit{rel}}(\sigma) = \textstyle\sum_i \mathrm{rot}(a_i) \big\}/G. 
\end{equation}
Here, $e_{\mathit{rel}}(\sigma)$ is the relative Euler number of the associated $\bR P^1$-bundle, with its section over $\partial S$ given by an eigenvector for the holonomy; and $G$ acts by overall conjugation.

\begin{remark}
In the discussion above, for instance in \eqref{eq:boundary-rotation}, we have always used $\partial_iS$ with its boundary orientation. When looking at specific examples of surfaces, we will sometimes adopt different orientation conventions on some boundary components (and will warn the reader when that happens).
\end{remark}

\subsection{Spaces of flat connections}
For our purpose, what matters is the weak homotopy type of the spaces \eqref{eq:moduli-spaces-2} (in the same sense as in Convention \ref{th:weak-homotopy}). Flat connections are easier to understand in that respect, because of \eqref{eq:rep-spaces}.

Let's begin with a case in which the topological aspect is trivial, namely that of the annulus $S = [0,1] \times S^1$. Our convention is that we identify each boundary circle $\partial_i S = \{i\} \times S^1$ with $S^1$ in the obvious way (which is orientation-reversing for $i = 0$, and orientation-preserving for $i = 1$). In order for $\scrA_{\mathit{flat}}(S,a_0,a_1)$ to be nonempty, the following condition must be satisfied:
\begin{equation} \label{eq:conjugate-holonomies}
\parbox{38em}{$a_0, a_1$ have conjugate holonomies ($C_0 = C_1$), and the same rotation number $r = \mathrm{rot}(a_0) = \mathrm{rot}(a_1)$. Equivalently, the conjugacy classes in $\tilde{G}$ are the same ($\tilde{C}_0 = \tilde{C}_1$).}
\end{equation}

\begin{prop} \label{th:z-homotopy}
Take $S = [0,1] \times S^1$, with boundary conditions \eqref{eq:conjugate-holonomies}. Then there is a weak homotopy equivalence
\begin{equation} \label{eq:z-homotopy}
\scrA_{\mathit{flat}}(S,a_0,a_1) \htp \bZ,
\end{equation}
unique up to adding a constant, and which is compatible with the action of $\scrG(S,a_0,a_1) \htp \bZ$. To state this more concretely, let $C = C_0 = C_1$. By associating to a connection $A$ its family of holonomies around the loops $\{s\} \times S^1$, one defines a canonical weak homotopy equivalence
\begin{equation} \label{eq:z-homotopy-2}
\scrA_{\mathit{flat}}(S,a_0,a_1) \longrightarrow \scrP(C,g_0,g_1).
\end{equation}
\end{prop}

\begin{proof}
Clearly, $\scrG(S,a_0,a_1)$ acts transitively on $\scrA_{\mathit{flat}}(S,a_0,a_1)$, with stabilizer $\bR$ at every point. This implies that each orbit gives a weak homotopy equivalence 
\begin{equation} \label{eq:gauge-orbit}
\scrG(S,a_0,a_1) \stackrel{\htp}{\longrightarrow} \scrA_{\mathit{flat}}(S,a_0,a_1). 
\end{equation}
We know from \eqref{eq:gauge-2} that $\scrG(S,a_0,a_1) \htp \bZ$. The map \eqref{eq:z-homotopy-2} is equivariant with respect to the action of $\Omega G$ (thought of as gauge transformations in the $[0,1]$ variable on the domain, and as acting by pointwise conjugation on the target space). From that and \eqref{eq:gauge-orbit}, it follows that \eqref{eq:z-homotopy-2} is a weak homotopy equivalence.
\end{proof}

\begin{addendum} \label{th:z-homotopy-2}
Suppose that $g_0 = g_1 = g$. Take a path 
\begin{equation} \label{eq:stupid-path}
c: [0,1] \longrightarrow S, \quad c(0) = (0,0), \;\; c(1) = (1,0). 
\end{equation}
For any $A \in \scrA_{\mathit{flat}}(S,a_0,a_1)$, parallel transport along $c$ yields an element of $G$ which commutes with $g$, hence is either trivial or hyperbolic. The rotation number $\mathrm{rot}_c(A) \in \bZ$ then defines an isomorphism $\pi_0(\scrA_{\mathit{flat}}(S,a_0,a_1)) \iso \bZ$.
One can take the path to have winding number zero around the $S^1$ factor, and that gives a preferred choice of \eqref{eq:z-homotopy}, which is the same one gets from identifying $\pi_0(\scrP(C,g,g)) \iso \bZ$ in \eqref{eq:z-homotopy-2}.
\end{addendum}

\begin{addendum} \label{th:z-homotopy-3}
In the situation of Proposition \ref{th:z-homotopy}, let $\tau \in \mathit{Diff}(S,\partial S)$ be the (positive, or right-handed) Dehn twist along $\{\half\} \times S^1$. The action of $\tau^*$ on $\scrA_{\mathit{flat}}(S,a_0,a_1)$ corresponds to subtracting $r$ on the right hand side of \eqref{eq:z-homotopy}. 
\end{addendum}


We omit the proofs, which are straightforward. The interesting point about Addendum \ref{th:z-homotopy-3} is that, even though the topology of the space of flat connections is always the same, the action of the diffeomorphism group depends on the rotation number.

Let's consider more general surfaces $S$. A classical result concerns the case of extremal rotation numbers, by which we mean
\begin{equation} \label{eq:rot-chi}
\sum_i \mathrm{rot}(a_i) = \pm \chi(S).
\end{equation}

\begin{theorem}[Goldman] \label{th:teichmuller}
Suppose that $\chi(S)<0$, and that the $a_i$ satisfy \eqref{eq:rot-chi}. Then, the representation of $\pi_1(S)$ associated to any $A \in \scrM_{\mathit{flat}}(S,\{a_i\})$ is (conjugate to) the holonomy representation of a hyperbolic metric on $S$ with geodesic boundaries of length
$l_i = \cosh^{-1}\big( \half |\mathrm{tr}(g_i)|\big)$.
\end{theorem}

The holonomy representation assumes that we have represented our hyperbolic surface $S$ as a quotient of the hyperbolic disc. The sign in \eqref{eq:rot-chi} depends on whether this representation agrees with the given orientation of $S$ or not. Here are two useful implications:

\begin{corollary} \label{th:all-hyperbolic}
For a flat connection as in Theorem \ref{th:teichmuller}, the holonomies along non-contractible closed curves is hyperbolic.
\end{corollary}

\begin{corollary} \label{th:teichmuller1b}
In the situation of Theorem \ref{th:teichmuller}, $\scrM_{\mathit{flat}}(S,\{a_i\}) \iso \bR^{-3\chi(S)-|\pi_0(\partial S)|}$.
\end{corollary}

For us, what's important is the following immediate consequence of Corollary \ref{th:teichmuller1b} and \eqref{eq:gauge-2}:

\begin{prop} \label{th:connections-rel-boundary}
In the situation of Theorem \ref{th:teichmuller}, there is a weak homotopy equivalence
\begin{equation} \label{eq:htp}
\scrA_{\mathit{flat}}(S,\{a_i\}) \htp H^1(S,\partial S),
\end{equation}
unique up to addition of a constant, and which is compatible with the action of \eqref{eq:gauge-2}.
\end{prop}

The statement is along the same lines as our previous Proposition \ref{th:z-homotopy} for the cylinder. In parallel with Addendum \ref{th:z-homotopy-2}, one can sometimes use rotation numbers along paths which connect different boundary components to reduce the ambiguity in the choice of the map \eqref{eq:htp}. This relies on the following:

\begin{lemma} \label{th:two-agree}
In the situation of Theorem \ref{th:teichmuller}, choose two distinct boundary components, and assume that, when we identify them in an orientation-reversing way, the boundary conditions agree. More precisely, denoting our boundary components by ``left'' and ``right'', what we want is:
\begin{equation} \label{eq:left-right}
\begin{aligned}
&
\left\{ 
\begin{aligned} 
& \epsilon_{\mathit{left}}: S^1 \stackrel\iso\longrightarrow \partial_{\mathit{left}} S && \text{orientation-reversing}, \\
& \epsilon_{\mathit{right}}: S^1 \stackrel\iso\longrightarrow \partial_{\mathit{right}} S && \text{orientation-preserving.} 
\end{aligned}
\right.
\\
& \quad \text{such that} \quad \epsilon_{\mathit{left}}^* a_{\mathit{left}} = \epsilon_{\mathit{right}}^* a_{\mathit{right}} \in \Omega^1(S^1,\frakg).
\end{aligned}
\end{equation}
Then, parallel transport along a path
\begin{equation} \label{eq:connecting-path}
c: [0,1] \longrightarrow S, \quad c(0) = \epsilon_{\mathit{left}}(0), \;\; c(1) = \epsilon_{\mathit{right}}(0),
\end{equation}
associates to $A \in \scrA_{\mathit{flat}}(S,\{a_i\})$ a hyperbolic element of $G$.
\end{lemma}

\begin{proof}
Glue together $\partial_{\mathit{left}} S$ and $\partial_{\mathit{right}} S$ to obtain a surface of genus one higher than $S$. Given any $A \in \scrA_{\mathit{flat}}(S,\{a_i\})$, that surface inherits a flat connection, to which Theorem \ref{th:teichmuller} applies. The holonomy of the glued connection around any non-contractible loop must be hyperbolic, by Corollary \ref{th:all-hyperbolic}; in particular, this holds for the loop obtained by identifying the endpoints of $c$.
\end{proof}

There is a variation of that idea, which involves the following terminology. For any nonzero integer $\mu$, the $\mu$-fold cover $a^\mu$ of some $a \in \Omega^1(S^1,\frakg)$ is defined to be the pullback by the standard $\mu$-fold self-map of the circle:
\begin{equation} \label{eq:a01-cover}
\text{if } a = a_t \mathit{dt}, \quad \text{then } a^\mu = \mu a_{\mu t}\,\mathit{dt}. 
\end{equation}

\begin{lemma} \label{th:two-are-powers}
Suppose that $S$ has at least three boundary components, and that \eqref{eq:rot-chi} holds. Choose parametrizations of two boundary components, as before, but now assume that
\begin{equation}
\left\{
\begin{aligned}
& \epsilon_{\mathit{left}}^*a_{\mathit{left}} = a^{\mu_{\mathit{left}}}, \\
& \epsilon_{\mathit{right}}^*a_{\mathit{right}} = a^{\mu_{\mathit{right}}},
\end{aligned}
\right.
\quad \text{for some $a \in \Omega^1(S^1,\frakg)$ and $\mu_{\mathit{left}},\mu_{\mathit{right}} > 0$.}
\end{equation}
Then, parallel transport along any path \eqref{eq:connecting-path} is again hyperbolic.
\end{lemma}

\begin{proof}
Let $\mu = \mu_{\mathit{left}}\mu_{\mathit{right}}$. There is a $\bZ/\mu$-cover $\tilde{S} \rightarrow S$, such that the preimage of $\partial_{\mathit{left}} S$ has $\mu_{\mathit{left}}$ connected components, and similarly for $\partial_{\mathit{right}} S$ (here, we are using the existence of a third boundary component). Choose one component in each preimage, denoting them by $\partial_{\mathit{left}} \tilde{S}$ and $\partial_{\mathit{right}} \tilde{S}$. The pullback of our boundary conditions to those components can be identified with $a^\mu$, in both cases. Hence, if we have $A \in \scrA_{\mathit{flat}}(S,\{a_i\})$, we can pull it back to $\tilde{S}$, and it will then induce a connection  on surface obtained by gluing together $\partial_{\mathit{left}}\tilde{S}$ and $\partial_{\mathit{right}}\tilde{S}$. One can choose the components and gluing map so that a lift of our path $c$ becomes a loop on the glued surface. With this in mind, the same argument as in Lemma \ref{th:two-agree} goes through.
\end{proof}

Returning to our discussion of \eqref{eq:htp}, these Lemmas see the following use:

\begin{addendum} \label{th:partial-fix}
In the situation of Proposition \ref{th:connections-rel-boundary}, suppose that additionally, Lemma \ref{th:two-agree} or Lemma \ref{th:two-are-powers} applies. Fix a path \eqref{eq:connecting-path}. Then, one can choose the map \eqref{eq:htp} so that its composition with $\int_c: H^1(S,\partial S) \rightarrow \bZ$ equals $\mathrm{rot}_c: \scrA_{\mathit{flat}}(S,\{a_i\}) \rightarrow \bZ$.
\end{addendum}

We now turn to the action of diffeomorphisms $\phi \in \mathit{Diff}(S,\partial S)$. Under \eqref{eq:htp}, the pullback action $\phi^*$ on $\scrA_{\mathit{flat}}(S,\{a_i\})$ corresponds to an affine automorphism, whose linear part is the usual action of $\phi$ on $H^1(S,\partial S)$. In particular, if $\phi$ acts trivially on $H^1(S,\partial S)$, its pullback action on $\scrA_{\mathit{flat}}(S,\{a_i\})$ can be described up to homotopy as a translation by a certain element
\begin{equation} \label{eq:translation-class}
T(\phi) \in H^1(S,\partial S).
\end{equation}
We will need the following computation, which can be reduced to Addendum \ref{th:z-homotopy-3} by looking at a tubular neighbourhood of the twisting curve:

\begin{addendum} \label{th:t-class}
In the situation of Proposition \ref{th:connections-rel-boundary}, let $\tau_d$ be the Dehn twist along a nontrivial simple closed curve $d$ which separates the surface into two parts $S_{\mathit{left}}$ and $S_{\mathit{right}}$; more precisely, we choose an orientation of $d$, and take $S_{\mathit{right}} \subset S$ to be the part for which that orientation agrees with the boundary orientation. Then, the associated class \eqref{eq:translation-class} is
\begin{equation} \label{eq:t-class}
\begin{aligned}
T(\tau_d) 
& = \Big( \pm \chi(S_{\mathit{right}}) - \!\! \sum_{\partial_i S \subset \partial S_{\mathit{right}}} \mathrm{rot}(a_i) \Big)[d]
\\
 & = \Big(\mp \chi(S_{\mathit{left}}) + \!\! \sum_{\partial_i S \subset \partial S_{\mathit{left}}} \mathrm{rot}(a_i) \Big)[d].
\end{aligned}
\end{equation}
\end{addendum}

The coefficient in front of $[d]$ in \eqref{eq:t-class} is the rotation number of our connections along $d$, which is prescribed by \eqref{eq:rot-chi} (the sign there determines those in our formula) and the Milnor-Wood inequality for $S_{\mathit{left}}$, $S_{\mathit{right}}$. 

%
%

\subsection{Spaces of nonnegatively curved connections}
We begin by going back to the basic relationship between curvature and the variation of parallel transport. 

\begin{lemma} \label{th:connections-on-the-square}
Given a nonnegatively curved connection on $[0,1]^2$, let $g(s) \in G$ be obtained by parallel transport along the family of paths $c_s$ from Figure \ref{fig:path}. Then $g(s)$ is a nonpositive path. 
Conversely, fix a nonpositive path $g(s)$, and a connection $a_0 \in \Omega^1([0,1],\frakg)$ whose parallel transport along $[0,1]$ yields $g(0)$. Then there is a nonnegatively curved connection $A$ on $[0,1]^2$, with $A|(\{0\} \times [0,1]) = a_0$, such that parallel transport along each $c_s$ recovers $g(s)$. Moreover, the space of all such $A$ is weakly contractible.
\end{lemma}
\begin{figure}
\begin{center}
\begin{picture}(0,0)%
\includegraphics{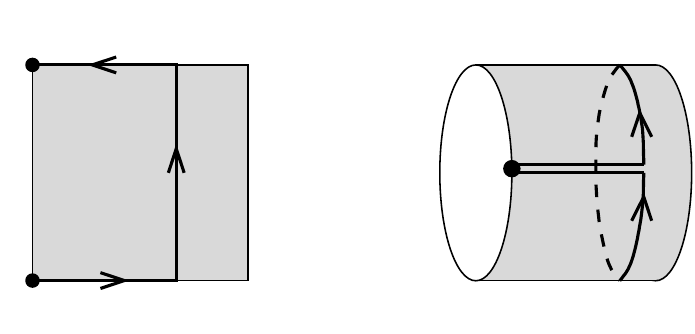}%
\end{picture}%
\setlength{\unitlength}{3947sp}%
\begingroup\makeatletter\ifx\SetFigFont\undefined%
\gdef\SetFigFont#1#2#3#4#5{%
  \reset@font\fontsize{#1}{#2pt}%
  \fontfamily{#3}\fontseries{#4}\fontshape{#5}%
  \selectfont}%
\fi\endgroup%
\begin{picture}(3328,1594)(1036,-1184)
\put(1051,239){\makebox(0,0)[lb]{\smash{{\SetFigFont{10}{12.0}{\familydefault}{\mddefault}{\updefault}{$(0,1)$}%
}}}}
\put(1801,239){\makebox(0,0)[lb]{\smash{{\SetFigFont{10}{12.0}{\familydefault}{\mddefault}{\updefault}{$(s,1)$}%
}}}}
\put(1801,-1111){\makebox(0,0)[lb]{\smash{{\SetFigFont{10}{12.0}{\familydefault}{\mddefault}{\updefault}{$(s,0)$}%
}}}}
\put(1051,-1111){\makebox(0,0)[lb]{\smash{{\SetFigFont{10}{12.0}{\familydefault}{\mddefault}{\updefault}{$(0,0)$}%
}}}}
\end{picture}%
\caption{\label{fig:path}The paths $c_s$ on the square and the cylinder.}
\end{center}
\end{figure}
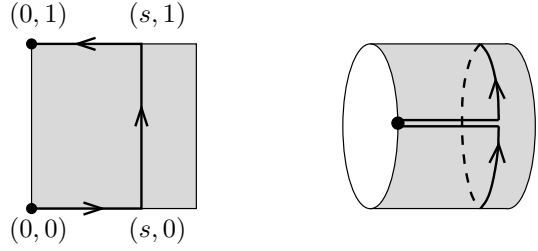 

\begin{proof}
We can work up to gauge transformations which are trivial on $\{0\} \times [0,1]$ (the group of such gauge transformations is contractible). Hence, we may restrict our attention to connections which are trivial in $s$-direction, meaning of the form $A(s,t) \mathit{dt}$; the curvature is then just given by $-\partial_s A$. Let's write $\Pi(s,t) \in G$ for the parallel transport along the straight line from $(s,0)$ to $(s,t)$. As in \eqref{eq:parallel-transport}, this is given by
\begin{equation} \label{eq:a-from-phi}
\Pi(s,0) = \Id, \quad (\partial_t \Pi) \Pi^{-1} = A. 
\end{equation}
Set 
\begin{equation} \label{eq:big-gamma}
\Gamma(s,t) = \partial_t (\Pi^{-1}\, \partial_s \Pi) = \Pi^{-1} (\partial_s A) \Pi \in \frakg.
\end{equation}
Integrating out yields
\begin{equation} \label{eq:phi-equation}
\Pi^{-1} (\partial_s \Pi) = \int_0^t \Gamma(s,\tau) \mathit{d\tau}.
\end{equation}
If the curvature is nonnegative, $\Gamma$ is nonpositive, which by \eqref{eq:phi-equation} implies that $g(s) = \Pi(s,1)$ is a nonpositive path. 

In the other direction, given a nonpositive path, set $\gamma = g^{-1} \,\partial_s g \in \frakg_{\leq 0}$. For $A$ to have nonnegative curvature and the desired parallel transport maps, we need 
\begin{equation} \label{eq:gamma-constraints}
\Gamma(s,t) \in \frakg_{\leq 0} \quad
\text{and} \quad
\int_0^1 \Gamma(s,t) \mathit{dt} = \gamma(s).
\end{equation}
The space of $\Gamma$ satisfying \eqref{eq:gamma-constraints} is clearly contractible. Given $\Gamma$ and the parallel transport maps $\Pi(0,t)$ coming from the given $a_0$, one gets maps $\Pi(s,t)$ by thinking of \eqref{eq:phi-equation} in reverse. This means that we set $\Xi(s,t) = \int_0^t \Gamma(s,\tau) d\tau$, and then define $\Pi(s,t)$ as the solution of the ODE
\begin{equation}
\partial_s \Pi = \Pi\, \Xi.
\end{equation}
Once one has the family $\Pi(s,t)$, there is a unique connection satisfying \eqref{eq:a-from-phi}, which then automatically restricts to $a_0$ on $\{0\} \times S^1$.
\end{proof}

\begin{lemma} \label{th:connections-on-the-cylinder}
Fix a nonpositive path $g(s)$, and a connection $a_0 \in \Omega^1(S^1,\frakg)$ whose holonomy is $g(0)$. Then there is a nonnegatively curved connection $A$ on $S = [0,1] \times S^1$, with $A|(\{0\} \times S^1) = a_0$, such that the holonomy along each $c_s$ (see again Figure \ref{fig:path}) is $g(s)$. Moreover, the space of all such $A$ is weakly contractible.
\end{lemma}

\begin{proof}
We follow the strategy from Lemma \ref{th:connections-on-the-square}, but pull everything back to the universal cover $[0,1] \times \bR \rightarrow S$, and work $\bZ$-periodically in $t$-direction. Given $\gamma = g^{-1} \partial_s g$, consider $\Gamma : [0,1] \times \bR \rightarrow \frakg$ which satisfy \eqref{eq:gamma-constraints} as well as the periodicity condition
\begin{equation}
\Gamma(s,t+1) = g(s)^{-1} \Gamma(s,t) g(s).
\end{equation}
This means that for a solution of \eqref{eq:phi-equation},
\begin{equation} \label{eq:s-t-1}
\begin{aligned}
\Pi(s,t+1)^{-1} \partial_s \Pi(s,t+1) & = \gamma(s) + g(s)^{-1} (\Pi(s,t)^{-1} \partial_s \Pi(s,t)) g(s) 
\\ & \quad = (\Pi(s,t) g(s))^{-1} \partial_s (\Pi(s,t) g(s)).
\end{aligned}
\end{equation}
We start with $a_0 \in \Omega^1(\bR,\frakg)$ which is $\bZ$-periodic, and whose parallel transport along $[0,1]$ is $g(0)$. This means that $\Pi(0,t+1) = \Pi(0,t) g(0)$. Because of \eqref{eq:s-t-1}, we then have $\Pi(s,t+1) = \Pi(s,t) g(s)$ everywhere. As a consequence, the connection determined by $\Pi(s,t)$ through \eqref{eq:a-from-phi} is again $\bZ$-periodic.
\end{proof}

With that in hand, let us return to the spaces of nonnegatively curved connections from \eqref{eq:moduli-spaces-2}. We remind the reader of the standing assumption for such spaces, which is that the boundary conditions $a_i$ are connections with hyperbolic holonomy.

\begin{prop} \label{th:nonnegative-connection}
Take $S = S^1 \times [0,1]$. As boundary conditions, fix $a_0,a_1 \in \Omega^1(S^1,\frakg)$ with the same rotation number $\mathrm{rot}(a_0) = \mathrm{rot}(a_1) = r$. Then there is a weak homotopy equivalence
\begin{equation} \label{eq:first-h}
\scrA_{\geq 0}(S,a_0,a_1) \htp \bZ,
\end{equation}
unique up to adding a constant on the right, and which is compatible with the action of $\scrG(a_0,a_1) \htp \bZ$. More explicitly, by associating to a connection $A$ its holonomies around the loops $d_s: S^1 \rightarrow S$, $d_s(t) = (s,t)$, one gets a weak homotopy equivalence
\begin{equation} \label{eq:connections-to-paths}
\scrA_{\geq 0}(S,a_0,a_1) \longrightarrow \scrP(G_{\mathit{hyp}},g_0,g_1).
\end{equation}
\end{prop}

\begin{proof}
Recall our general notation \eqref{eq:a-space-notation}: $C_1$ is the conjugacy class of the holonomy $g_1$ of $A_1$, and $\tilde{C}_1$ its lift to $\tilde{G}$. Consider the space $\scrA_{\geq 0}(S,a_0,\tilde{C}_1)$ of nonnegatively curved connections which restrict to $a_0$ on $\{0\} \times S^1$, and whose lifted holonomy around $\{1\} \times S^1$ lies in $\tilde{C}_1$ (or equivalently, the holonomy lies in $C_1$ and $\mathrm{rot}_{\partial_1S}(A) = r$). By Lemma \ref{th:connections-on-the-cylinder}, there is a weak homotopy equivalence between that and a suitable space of nonpositive paths:
\begin{equation} \label{eq:a-p}
\scrA_{\geq 0}(S,a_0,\tilde{C}_1) \htp \scrP_{\leq 0}(G,g_0,C_1)^0.
\end{equation}
Lemma \ref{th:pre-short} (applied to the reversed path) shows that the path space in \eqref{eq:a-p} is contractible. On the other hand, we have a weak fibration
\begin{equation} \label{eq:fi}
\scrA_{\geq 0}(S,a_0,a_1) \longrightarrow \scrA_{\geq 0}(S,a_0,\tilde{C}_1) \longrightarrow C_1.
\end{equation}
The last space in \eqref{eq:fi} should really be the space of connections on $S^1$ with (lifted) holonomy in $\tilde{C}_1$. By looking at parallel transport, one sees that this is weakly homotopy equivalent to the path space $\scrP(G,\mathit{id},C_1)^r$. In turn, evaluation at the endpoint yields a weak homotopy equivalence between that path space and $C_1$, which is the form in which \eqref{eq:fi} has been stated. It follows that $\scrA_{\geq 0}(S,a_0,a_1) \htp \Omega C_1 \htp \bZ$. For any $A \in \scrA_{\geq 0}(S,a_0,a_1)$, the holonomy around $\{s\} \times S^1$ is hyperbolic. This follows from Lemma \ref{th:connections-on-the-cylinder} and the behaviour of nonnegative paths, as described in Lemma \ref{th:pre-short}. Hence, \eqref{eq:connections-to-paths} makes sense. The proof that it is a homotopy equivalence involves looking at gauge transformations in the $[0,1]$-variable, as in Proposition \ref{th:z-homotopy}). More precisely, if $PG$ is the space of all paths in $G$ starting $\Id$, the gauge action on a fixed connection yields a commutative diagram extending \eqref{eq:fi},
\begin{equation}
\label{eq:fi2}
\xymatrix{
\scrA_{\geq 0}(S,a_0,a_1) \ar[r] & \scrA_{\geq 0}(S,a_0,\tilde{C}_1) \ar[r] & C_1 \\
\Omega G \ar[u] \ar[r] & PG \ar[u] \ar[r] & G \ar[u]
}
\end{equation}
Since the middle and right hand $\uparrow$ are weak homotopy equivalences, so is the left hand one, and that implies the desired statement concerning \eqref{eq:connections-to-paths}.
\end{proof}

\begin{addendum} \label{th:nonnegative-connection-addendum}
In the situation from Proposition \ref{th:nonnegative-connection}, assume additionally that the boundary holonomies belong to the same one-parameter semigroup, meaning that they are of the form
\begin{equation} \label{eq:semigroup}
g_i = e^{t_i \gamma} \quad \text{with } t_i >0 \text{ and a common } \gamma.
\end{equation}
Then there is a preferred choice of isomorphism $\pi_0(\scrP(G_{\mathit{hyp}},g_0,g_1)) \iso \bZ$, by looking at how the eigenvectors rotate; hence, the same holds for \eqref{eq:first-h}.
\end{addendum}

\begin{addendum} \label{th:dehn-twist-pullback}
In the situation of Proposition \ref{th:nonnegative-connection}, let $\tau \in \mathit{Diff}(S,\partial S)$ be the Dehn twist along $\{\half\} \times S^1$. The action of $\tau^*$ on $\scrA_{\geq 0}(S,a_0,a_1)$ corresponds to subtracting $r$ on the right hand side of \eqref{eq:first-h}.
\end{addendum}

The first of these Addenda is obvious. The second one follows from the description \eqref{eq:connections-to-paths} (alternatively, one can deform $a_1$ so that \eqref{eq:conjugate-holonomies} holds, and reduce the computation to Addendum \ref{th:z-homotopy-3}; this works because the inclusion of the flat connections into the nonnegatively curved ones is a weak homotopy equivalence, by Propositions \ref{th:z-homotopy} and  \ref{th:nonnegative-connection}).


\begin{prop} \label{th:nonnegative-connection-2}
Consider the situation from Proposition \ref{th:nonnegative-connection}, except that now,
$\mathrm{rot}(a_0) = \mathrm{rot}(a_1) + 1$. Then $\scrA_{\geq 0}(S,a_0,a_1)$ is weakly contractible.
\end{prop}

\begin{proof}
The proof is similar to that of Proposition \ref{th:nonnegative-connection}. By Lemmas \ref{th:short-path-2} and \ref{th:connections-on-the-cylinder}, 
\begin{equation}
\scrA_{\geq 0}(S,a_0,\tilde{C}_1) \htp \scrP_{\leq 0}(G,g_0,C_1)^{-1} \htp C_1.
\end{equation}
One has the same fibration \eqref{eq:fi}, and this time it follows that the fibre is weakly contractible.
\end{proof}

\begin{prop} \label{th:nonnegative-connection-3}
Take $S$ to be a disc, with the boundary parametrized clockwise. As boundary condition, fix $a \in \Omega^1(S^1,\frakg)$ with $\mathrm{rot}(a) = 1$. Then $\scrA_{\geq 0}(S,a)$ is weakly contractible.
\end{prop}

\begin{proof}
Once more, consider first the space $\scrA_{\geq 0}(S, \tilde{C})$ where we only prescribe the (lifted) conjugacy class of the boundary holonomy. Fix $\ast \in S \setminus \partial S$, and a family of maps $\phi_t: S \rightarrow S$ with $\phi_0 = \mathit{id}$, $\phi_t(\ast) = \ast$, $\phi_t|\partial S = \mathit{id}$, $\mathrm{det}(D\phi_t) \geq 0$ everywhere; and such that $\phi_1$ contracts a neighbourhood of $\ast$ to that point. By pulling back via those maps, one sees that $\scrA_{\geq 0}(S,\tilde{C})$ is weakly homotopy equivalent to the subspace of connections which are trivial near the chosen point. Using this and Lemma \ref{th:connections-on-the-cylinder}, one shows that the holonomy around concentric loops yields a weak homotopy equivalence 
\begin{equation}
\scrA_{\geq 0}(S,\tilde{C}) \htp \scrP_{\leq 0}(G,\mathit{id},C^{-1})^{-1} = \scrP_{\geq 0}(G,\mathit{id},C)^1.
\end{equation}
By combining this with Lemma \ref{th:short-path-3}, one finds that the holonomy around $\partial S$ yields a weak homotopy equivalence $\scrA_{\geq 0}(S,\tilde{C}) \htp C$. Our space $\scrA_{\geq 0}(S,a)$ is homotopically the fibre of that map.
\end{proof}

\begin{prop} \label{th:nonnegative-connection-4}
Let $S$ be a genus zero surface with $m+1 \geq 3$ boundary components, labeled by $I = \{0,\dots,m\}$. We choose identifications $\partial_iS \iso S^1$ which are orientation-reversing for $i = 0$, and orientation-preserving for $i > 0$. As boundary conditions, take $a_i \in \Omega^1(S^1,\frakg)$ which satisfy 
\begin{equation} \label{eq:m-minus-1}
\mathrm{rot}(a_0) = \mathrm{rot}(a_1) + \cdots + \mathrm{rot}(a_m) + 1-m. 
\end{equation}
Then there is a weak homotopy equivalence, unique up to adding a constant, and compatible with the action of \eqref{eq:gauge-2},
\begin{equation} \label{eq:pants-connections}
\scrA_{\geq 0}(S,\{a_i\}) \htp H^1(S,\partial S).
\end{equation}
\end{prop}
\begin{figure}
\begin{centering}
\begin{picture}(0,0)%
\includegraphics{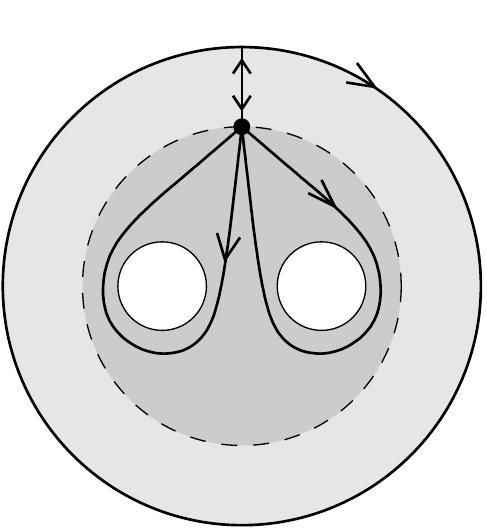}%
\end{picture}%
\setlength{\unitlength}{3355sp}%
\begingroup\makeatletter\ifx\SetFigFont\undefined%
\gdef\SetFigFont#1#2#3#4#5{%
  \reset@font\fontsize{#1}{#2pt}%
  \fontfamily{#3}\fontseries{#4}\fontshape{#5}%
  \selectfont}%
\fi\endgroup%
\begin{picture}(2730,2972)(2086,-2025)
\put(3276,-1336){\makebox(0,0)[lb]{\smash{{\SetFigFont{10}{12.0}{\rmdefault}{\mddefault}{\updefault}{$S_{\mathit{right}}$}%
}}}}
\put(3826,-1186){\makebox(0,0)[lb]{\smash{{\SetFigFont{10}{12.0}{\rmdefault}{\mddefault}{\updefault}{$c_2$}%
}}}}
\put(2926,-1186){\makebox(0,0)[lb]{\smash{{\SetFigFont{10}{12.0}{\rmdefault}{\mddefault}{\updefault}{$c_1$}%
}}}}
\put(3751,764){\makebox(0,0)[lb]{\smash{{\SetFigFont{10}{12.0}{\rmdefault}{\mddefault}{\updefault}{$c_0$}%
}}}}
\put(3276,-1861){\makebox(0,0)[lb]{\smash{{\SetFigFont{10}{12.0}{\rmdefault}{\mddefault}{\updefault}{$S_{\mathit{left}}$}%
}}}}
\end{picture}%
\caption{\label{fig:holes}The loops $c_i$ from the proof of Proposition \ref{th:nonnegative-connection-4}, for $m = 2$.}
\end{centering}
\end{figure}
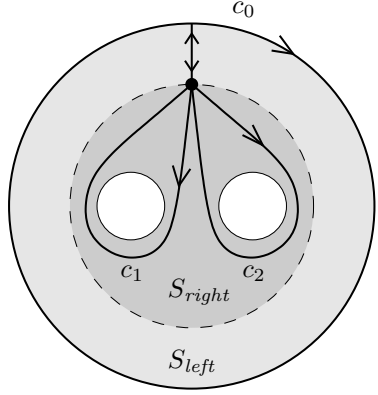%

\begin{proof}
Let's start with the space $\scrA_{\geq 0}(S,\{\tilde{C}_i\})$ where only the conjugacy classes of the (lifted) boundary holonomies are fixed. As in \eqref{eq:fi}, this fits into a weak fibration
\begin{equation} \label{eq:fi3}
\scrA_{\geq 0}(S,\{a_i\}) \longrightarrow \scrA_{\geq 0}(S,\{\tilde{C}_i\}) \longrightarrow \prod_{i=0}^m C_i.
\end{equation}
Let's think of $S$ as being glued together, rather trivially, from two parts: $S_{\mathit{left}}$ is a collar neighborhood of $\partial_0 S$, so that the other part $S_{\mathit{right}} \subset S$ is a deformation retract, containing $\partial_i S$ for $i>0$ (see Figure \ref{fig:holes}). One can find a family of self-maps $(\phi_t)$ of $S$, starting at the identity, which are all the identity on $\partial S$, such that $\mathrm{det}(D\phi_t) \geq 0$ everywhere, and where $\phi_1$ retracts $S_{\mathit{right}}$ to a one-dimensional skeleton. From that, it follows (as in the proof of Proposition \ref{th:nonnegative-connection-3}) that $\scrA_{\geq 0}(S,\{\tilde{C}_i\})$ is weakly homotopy equivalent to the subspace of those connections which are flat on $S_{\mathit{right}}$.

The space of flat connections on $S_{\mathit{right}}$, with holonomy around $\partial_i S$ ($i>0$) lying in $\tilde{C}_i$, is weakly homotopy equivalent to $\prod_{i=1}^m C_i$, where we can think of the homotopy equivalence as being given by holonomies around loops $c_i$ with a common base point (see again Figure \ref{fig:holes}). In a slight break with our usual notation, we will use $g_i$ to denote those holonomies, and $\tilde{g}_i$ for their lifts to $\tilde{C}_i$. Extending such a connection on $S_{\mathit{right}}$ to a nonnegatively curved one on the whole of $S$, with the desired behaviour over $\partial_0 S$, amounts (up to homotopy, thanks to Lemma \ref{th:connections-on-the-cylinder}) to choosing a nonnegative path from  $\tilde{g}_1\cdots \tilde{g}_m$ to some point $\tilde{g}_0 \in \tilde{C}_0$ (the holonomy around $c_0$). To conclude, we have now shown that $\scrA_{\geq 0}(S,\{\tilde{C}_i\})$ is weakly homotopy equivalent to the space considered in Lemma \ref{th:pp}, hence to $C_0$. By looking at \eqref{eq:fi3}, one sees that $\scrA_{\geq 0}(S,\{a_i\}) \htp \prod_{i>0} \Omega C_i = \bZ^m$. Inspection of the argument shows that each gauge orbit yields a weak homotopy equivalence $\scrG(S,\{a_i\}) \rightarrow \scrA_{\geq 0}(S,\{a_i\})$, and from that we get the statement in its desired form.
\end{proof}

We have arranged the results above roughly in order of complexity of their proofs, which was convenient since they all follow the same strategy. However, this might not be the best way of understanding their meaning. As the reader may have noticed, Proposition \ref{th:nonnegative-connection} was a close cousin of the previous discussion concerning flat connections, in Proposition \ref{th:z-homotopy}; similarly, one can compare Propositions \ref{th:connections-rel-boundary} (specialized to genus zero) and \ref{th:nonnegative-connection-4}. The outcome is that we have weak homotopy equivalences
\begin{equation} \label{eq:compare-flat}
\begin{aligned}
& \scrA_{\mathit{flat}}(S,a_0,a_1) \stackrel{\htp}{\longrightarrow} \scrA_{\geq 0}(S,a_0,a_1) && \text{ for Proposition \ref{th:z-homotopy} (as mentioned before)},
\\ &
\scrA_{\mathit{flat}}(S,\{a_i\}) \stackrel{\htp}{\longrightarrow} \scrA_{\geq 0}(S,\{a_i\}) && \text{ for Proposition \ref{th:nonnegative-connection-4}}.
\end{aligned}
\end{equation}
The other two spaces of nonnegatively curved connections that we have analyzed show a different behaviour: they are contractible, which one could interpret as a comparison with the spaces of all connections,
\begin{equation} \label{eq:compare-all}
\begin{aligned}
& \scrA_{\geq 0}(S,a_0,a_1) \stackrel{\htp}{\longrightarrow} \scrA(S,a_0,a_1) && \text{ for Proposition \ref{th:nonnegative-connection-2},}
\\ &
\scrA_{\geq 0}(S,a) \stackrel{\htp}{\longrightarrow} \scrA(S,a) && \text{ for Proposition \ref{th:nonnegative-connection-3}.}
\end{aligned}
\end{equation}
This reflects our previous distinction between spaces of ``short paths'' (whose topology could be related to constant paths) and ``long paths'' (related to spaces of all paths).

\begin{remark}
It is tempting to think that a comparison result such as \eqref{eq:compare-flat} should hold for higher genus surfaces, as long as the rotation numbers are chosen extremal, meaning that they saturate the bound in Proposition \ref{th:bochner}; and similarly for closed surfaces, with rotation numbers replaced by the Euler class.
\end{remark}

\section{TQFT considerations\label{sec:tqft}} 
We now introduce our toy model formalism, and explore its implications. In fact, we will go over similar terrain three times: first in a way that only covers TQFT operations in a standard sense; then (in Section \ref{sec:diff-axiom}) with additional features that lead to the construction of connections; and finally, in a mild generalization which allows elliptic holonomies as well (Section \ref{sec:elliptic}). Overall, the discussion roughly follows \cite[Sections 5--6]{seidel16}.

\subsection{Disc configurations}
Our formalism is built as a version of the classical framed little disc operad \cite{getzler94b, salvatore-wahl03}, enriched with connections (but not in the gauge theory sense: we keep a fixed trivialization of the bundle).
\begin{itemize}
\itemsep.5em
\item
Write $D \subset \bC$ for the closed unit disc. An $m$-disc configuration, for $m \geq 0$, is an ordered collection of (closed round) discs $D_1,\dots,D_m$, which are contained in the interior of $D$, and are pairwise disjoint. To such a configuration, we can associate the (genus zero oriented) surface
\begin{equation} \label{eq:surface}
S = D \setminus \bigcup_{i=1}^m (D_i \setminus \partial D_i),
\end{equation}
whose boundary circles we label by $\partial_0 S = \partial D$, and $\partial_i S = \partial D_i$ for $i=1,\dots,m$. We also want to allow one exceptional case, called the identity configuration, which consists of a single disc $D_1 = D$. For that configuration, \eqref{eq:surface} becomes a circle, thought of as an annulus of thickness zero.

\item
A framing of an $m$-disc configuration consists of boundary parametrizations
\begin{equation} \label{eq:boundary-parametrizations}
\epsilon_i: S^1 \longrightarrow \partial_i S, \quad
\epsilon_i(t) = \rho_i e^{-2\pi i t} + \zeta_i, \quad i = 1,\dots,m.
\end{equation}
Here, $\zeta_i$ is the center of $D_i$, and $|\rho_i| > 0$ is its radius, so only $\mathrm{arg}(\rho_i)$ is a free parameter; for the identity configuration, we specify that the parameter must be $\rho_1 = 1$. To the \eqref{eq:boundary-parametrizations} we add a fixed parametrization of the boundary of $D$,
\begin{equation} \label{eq:boundary-parametrizations-0}
\epsilon_0: S^1 \longrightarrow \partial_0 S, \quad
\epsilon_0(t) = e^{-2\pi it}.
\end{equation}
The clockwise parametrization convention means that $\epsilon_0$ is inverse to the boundary orientation, while the other $\epsilon_i$ are compatible with it. 

\item
Suppose that we have a framed disc configuration, together with $a_i \in \Omega^1(S^1,\frakg)$ for $i = 0,\dots,m$, which have hyperbolic holonomy (as before, we will write $g_i$ for the holonomies, and $C_i$ for their conjugacy classes). A compatible connection is given by an $A \in \Omega^1(S,\frakg)$ with nonnegative curvature, and whose pullback by \eqref{eq:boundary-parametrizations}, \eqref{eq:boundary-parametrizations-0} equals $a_i$. We also impose the following technical condition: for every boundary circle $\partial_iS$, let $r_i$ be the radial retraction of a neighbourhood onto that circle; then,  $A = r_i^*(A|\partial_i S)$ near $\partial_iS$. This means that the behaviour of $A$ close to $\partial S$ is entirely determined by the $a_i$. Again, we have to mention the exceptional case of the identity configuration, where necessarily $a_0 = a_1$, and there is no additional freedom in choosing $A$.
\end{itemize}
A decorated disc configuration is a disc configuration together with a choice of framing and compatible connection. We write $Y = (\epsilon_1,\dots,\epsilon_m,A)$ for the decoration, and usually represent the decorated disc configuration by the pair $(S,Y)$.

There is a gluing process for decorated disc configurations. Start with an $m_1$-configuration and an $m_2$-configuration. Choose some $i \in \{1,\dots,m_1\}$, and insert the second configuration into the first one in place of the $i$-th disc, rescaling and rotating it as prescribed by $i$-th boundary parametrization. In terms of the associated surfaces, this means that we glue them together along one of their boundary circles, using the given parametrizations. Assuming that the boundary conditions for the connections agree, those will give rise to a connection on the glued surface. The outcome is a decorated configuration consisting of $m = m_1+m_2-1$ discs. The gluing construction is strictly associative. One can formally include the identity configuration in the gluing process, for which (as the name suggests) it is the neutral element.

We will use families of disc configurations, parametrized by smooth manifolds (possibly with boundary or corners) $P$, and equipped with the same kind of decorations. Aside from one exception, this notion is straightforward:
\begin{itemize} \itemsep.5em
\item The disc configurations vary smoothly depending on the parameter, giving rise to a fibration with fibres \eqref{eq:surface}, which we denote by $\underline{S} \rightarrow P$.
\item Similarly, the framings vary smoothly with $P$. The connection is given by a one-form in fibre direction on $\underline{S}$. The boundary conditions $(a_0,\dots,a_m)$ are always constant (independent of where we are on $P$). We usually denote such a decoration by $\underline{Y} = (\underline{\epsilon}_1,\dots,\underline{\epsilon}_m,\underline{A})$.
\end{itemize}
The exceptional case occurs when we have a family of annuli with the same boundary conditions ($m = 1$ and $a_0 = a_1$). That could include instances of the identity configuration, which means that $\underline{S}$ will be singular in general. This is potentially a problem for the definition of connection on the fibres: to avoid that, we assume that in a neighbourhood of the subset of $P$ where the identity configuration occurs, the framing should be trivial (no rotation), and  the connection should be pulled back from radial projection to $S^1$. The gluing construction extends to families (one starts with a family over $P_1$ and another over $P_2$, and ends up with one over $P_1 \times P_2$). 

\subsection{Spaces of decorations}
Fix an $m$-disc configuration, which is not the identity configuration, and boundary values $(a_0,\dots,a_m)$. Write
\begin{equation} \label{eq:y-space}
\begin{aligned}
\scrY_{\geq 0}(S,\{a_i\}) & \;\; \text{space of all decorations $Y = (\epsilon_1,\dots,\epsilon_m,A)$ of $S$,} \\
\scrY_{\mathit{flat}}(S,\{a_i\}) &\;\; \text{subspace where the connection is flat.}
\end{aligned}
\end{equation}
With respect to the previously considered spaces $\scrA_{\geq 0}(S,\{a_i\})$ and $\scrA_{\mathit{flat}}(S,\{a_i\})$, the only difference (apart from certain conventions concerning boundary orientations) is the additional datum given by the framings. Nevertheless, we find it convenient to explicitly translate the outcome of the discussion from Section \ref{sec:connections} into the language of \eqref{eq:y-space}, since that will be relevant for our applications. The most obvious case is that of the empty disc configuration ($m = 0$), where there is no framing, and Proposition \ref{th:nonnegative-connection-3} says the following:

\begin{proposition} \label{th:the-disc}
Let $S$ be a disc, with $\mathrm{rot}(a_0) = 1$. Then $\scrY_{\geq 0}(S,a_0)$ is contractible.
\end{proposition}

Next, consider the case of the annulus ($m = 1$). Let's introduce the following abelian topological group (depending on $r \in \bZ$):
\begin{equation}
\Gamma_r = \begin{cases} \bZ/r & r \neq 0, \\ \bZ \times S^1 & r  = 0. \end{cases}
\end{equation}

\begin{prop} \label{th:x-annulus}
(i) Let $S$ be an annulus, with boundary conditions which satisfy \eqref{eq:conjugate-holonomies}. Then there is a weak homotopy equivalence
\begin{equation} \label{eq:x-annulus}
\scrY_{\mathit{flat}}(S,a_0,a_1) \htp \Gamma_r 
\end{equation}
which is unique up to (homotopy and) adding an element of $\bZ/r$. To state this more precisely, choose an identification $[0,1] \times S^1 \iso S$ which restricts to $\epsilon_i$ along $\{i\} \times S^1$. Then, the holonomy around the circles $\{s\} \times S^1$ gives a map
\begin{equation} \label{eq:explicit-mod-r}
\scrY_{\mathit{flat}}(S,a_0,a_1) \longrightarrow \pi_0(\scrP(C,g_0,g_1)) \times_{\bZ} \bZ/r,
\end{equation}
where $C$ is the common conjugacy class of the $g_i$, and $\bZ = \pi_1(C)$ acts by composition of paths with loops. The map \eqref{eq:explicit-mod-r} is canonical (independent of the choice of identification), and yields the discrete part of \eqref{eq:x-annulus}; for $r = 0$, $\mathrm{arg}(\rho_1)$ yields the continuous part.

(ii) Take boundary conditions $a_0,a_1$ which have the same rotation number $r$. Then there is a homotopy equivalence
\begin{equation} \label{eq:x-annulus-2}
\scrY_{\geq 0}(S,a_0,a_1) \htp \Gamma_r.
\end{equation}
This has the same uniqueness property as \eqref{eq:x-annulus}, and can be described as in \eqref{eq:explicit-mod-r}, except that the map now takes values in $\pi_0(\scrP(G_{\mathit{hyp}},g_0,g_1)) \times_{\bZ} \bZ/r$. In the more restricted case \eqref{eq:conjugate-holonomies}, the inclusion $\scrY_{\mathit{flat}}(S,a_0,a_1) \subset \scrY_{\geq 0}(S,a_0,a_1)$ is a weak homotopy equivalence, and our construction reduces to that in (i).

(iii) Take boundary conditions $a_0,a_1$ such that $\mathrm{rot}(a_0) = \mathrm{rot}(a_1) + 1$. Then $\mathrm{arg}(\rho_1)$ yields a weak homotopy equivalence
\begin{equation}
\scrY_{\geq 0}(S,a_0,a_1) \htp S^1.
\end{equation}
\end{prop}

The proofs are straightforward, given our previous discussion. Each of them is based on considering the fibration
\begin{equation} \label{eq:rho-fibration-1}
\scrA_{\mathit{flat}}(S,a_0,a_1) \longrightarrow \scrY_{\mathit{flat}}(S,a_0,a_1) \xrightarrow{\mathrm{arg}(\rho_1)} S^1,
\end{equation}
or its counterpart for $\scrY_{\geq 0}(S,a_0,a_1)$. For (i), Proposition \ref{th:z-homotopy} tells us that the fibre is homotopy equivalent to $\bZ$, and Addendum \ref{th:z-homotopy-3} shows that going around the $S^1$ base corresponds to subtracting $r$. The same goes for (ii) and (iii), using Proposition \ref{th:nonnegative-connection} with Addendum \ref{th:dehn-twist-pullback}, respectively Proposition \ref{th:nonnegative-connection-2}. 

\begin{addendum} \label{th:x-annulus-a}
(i) In the situation of Proposition \ref{th:x-annulus}(i), suppose that $g_0 = g_1 = g$. Then, following Addendum \ref{th:z-homotopy-2}, there is a preferred choice of \eqref{eq:x-annulus}. One can see that from \eqref{eq:explicit-mod-r}; or equivalently, it can be described as the map 
\begin{equation} \label{eq:rot-mod-r}
(\epsilon_1,A) \longmapsto \mathrm{rot}_c(A) \in \bZ/r, 
\end{equation}
where $c$ is any path from $\epsilon_0(0)$ to $\epsilon_1(0)$ (the rotation number mod $r$ is independent of the choice of path).

(ii) In the situation of Proposition \ref{th:x-annulus}(ii), suppose that $g_0$ and $g_1$ lie in the same one-parameter semigroup. Then, there is again a preferred choice of \eqref{eq:x-annulus-2}, see Addendum \ref{th:nonnegative-connection-addendum} (but it is unclear whether the alternative description \eqref{eq:rot-mod-r} would still work).
\end{addendum}

\begin{example} \label{th:fractional-rotation}
Let $r \neq 0$. Consider the configuration consisting of a disc centered at zero, with parametrization $\epsilon_1(t) = \rho_1 e^{-2\pi i (t+k/r)}$, for some $k \in \bZ/r$. Suppose that $a \in \Omega^1(S^1,\frakg)$ is an $r$-fold cover, in the sense of \eqref{eq:a01-cover}. Radial projection from the disc to the circle equips $S$ with a flat connection $A$, whose boundary values are $a_0 = a_1 = a$. By looking at \eqref{eq:rot-mod-r}, one sees that this belongs to the $k$-th component of $\scrY_{\mathit{flat}}(S,a,a)$.
\end{example}

\begin{prop} \label{th:decorations}
(i) Consider a configuration of $m \geq 2$ discs, with boundary conditions such that \eqref{eq:m-minus-1} holds, and write $r_i = \mathrm{rot}(a_i)$. Then, there is a weak homotopy equivalence
\begin{equation} \label{eq:x-space}
\scrY_{\mathit{flat}}(S,\{a_i\}) \htp \prod_{i=1}^m \Gamma_{r_i},
\end{equation}
unique up to adding an element of $\prod_{i=1}^m \bZ/r_i$. 

(ii) In the same situation, the inclusion $\scrY_{\mathit{flat}}(S,\{a_i\}) \subset \scrY_{\geq 0}(S,\{a_i\})$ is a weak homotopy equivalence. Hence, $\scrY_{\geq 0}(S,\{a_i\})$ admits the same description \eqref{eq:x-space}.
\end{prop}

To prove this (let's say, for flat connections), one considers the fibration
\begin{equation} \label{eq:multiple-rho}
\scrA_{\mathit{flat}}(S,\{a_i\}) \longrightarrow \scrY_{\mathit{flat}}(S,\{a_i\}) \xrightarrow{(\mathrm{arg}(\rho_1),\dots,\mathrm{arg}(\rho_m))} (S^1)^m.
\end{equation}
By Proposition \ref{th:connections-rel-boundary}, the fibre is homotopy equivalent to $H^1(S,\partial S)$, which we identify with $\bZ^m$ by taking generators dual to paths $c_i$ ($1 \leq i \leq m$) from $\partial_0 S$ to $\partial_i S$. The monodromy of \eqref{eq:multiple-rho} around the $i$-th $S^1$ factor agrees with the pullback action $(\tau_{d_i}^{-1})^*$ of an inverse (or left-handed) Dehn twist along a loop $d_i$ which is a parallel copy of $\partial_i S$. Let's orient this loop opposite to the boundary orientation of $\partial_iS$; it then satisfies $d_i \cdot c_i = -1$, hence its Poincar{\'e} dual is minus the $i$-th unit element in our basis. Addendum \ref{th:t-class} tells us that the action of $\tau_{d_i}^{-1}$ corresponds to adding the element
\begin{equation} \label{eq:rot-ai}
-T(\tau_{d_i}) = -r_i [d_i] = (0,\dots,r_i,\dots,0) \in H^1(S,\partial S) \iso \bZ^m,
\end{equation}
which implies (i). For (ii), one compares Propositions \ref{th:connections-rel-boundary} and \ref{th:nonnegative-connection-4}, as already mentioned in \eqref{eq:compare-flat}.

\begin{addendum} \label{th:two-agree-3}
In the situation of Proposition \ref{th:decorations}(i), suppose that $a_0 = a_m = a$ has rotation number $r$, and that $\mathrm{rot}(a_i) = 1$ for $i = 1,\dots,m-1$. In analogy with Addendum \ref{th:x-annulus-a}(i), an explicit choice of (the discrete part of) \eqref{eq:x-space} is then given by 
\begin{equation} \label{eq:rot-mod-2}
(\epsilon_1,\dots,\epsilon_m,A) \longmapsto \mathrm{rot}_c(A) \in \bZ/r,
\end{equation}
where $c$ is any path from $\epsilon_0(0)$ to $\epsilon_m(0)$. 
\end{addendum}

Most of the necessary argument is carried over from Addendum \ref{th:partial-fix}. What remains to be shown is that \eqref{eq:rot-mod-2} is independent of the choice of path, or equivalently, is invariant under $\mathit{Diff}(S,\partial S)$. To see that, note that $\pi_0(\mathit{Diff}(S,\partial S))$ is generated by Dehn twists along (necessarily separating) curves $d$. If $d$ does not separate $\partial_0 S$ and $\partial_1 S$, its intersection number with $c$ must be zero, which in terms of Addendum \ref{th:t-class} means that $T(\tau_d) \cdot [c] = 0$. In the remaining case, $d \cdot c = \pm 1$ and $T(\tau_d)$ is a multiple of $r[d]$, so using Addendum \ref{th:t-class} again, we get $T(\tau_d) \cdot [c] \in r\bZ$. In either case, \eqref{eq:rot-mod-2} is unchanged under $\tau_d^*$.

One can deepen this discussion by allowing the disc configuration to vary. Let $\mathit{Conf}_m$ (for some $m \geq 2$) be the space of all $m$-disc configurations (this is homotopy equivalent to the standard ordered configuration space). Over it, we have a universal family of surfaces $\underline{S}$; and on the fibres of that family, we have a space of all possible choices of decorations (let's say, flat ones), denoted by $\scrY_{\mathit{flat}}(\underline{S},\{a_i\})$. This all fits into a diagram, of which the columns and rows are fibrations:
\begin{equation} \label{eq:monodromy-diagram}
\xymatrix{
{\bZ^m} \ar[d]_-{\prod_i r_i} \ar[r] 
&
{\mathit{point}} \ar[d] \ar[r] 
&
(S^1)^m \ar[d]^-{\text{inclusion}}
\\
{\bZ^m} \ar[r] \ar[d]
&
{\scrY_{\mathit{flat}}(\underline{S},\{a_i\})}
\ar[r] \ar@{=}[d]
&
(S^1)^m \times \mathit{Conf}_m \ar[d]^-{\text{projection}}
\\
\prod_{i=1}^m \Gamma_{r_i} \ar[r]
& 
\scrY_{\mathit{flat}}(\underline{S},\{a_i\}) \ar[r]
& \mathit{Conf}_m
}
\end{equation}
Here, $(S^1)^m \times \mathit{Conf}_m$ is the space of framed configurations. The middle row of \eqref{eq:monodromy-diagram} restricts to \eqref{eq:multiple-rho} on each configuration, and the description of the map $\bZ^m \rightarrow \bZ^m$ reflects our previous discussion of \eqref{eq:rot-ai}. Here is one important additional piece of information concerning the middle row:

\begin{addendum} \label{th:twist-loop}
In the situation of Proposition \ref{th:decorations}, consider a loop in $\mathit{Conf}_m$ formed by moving the $i$-th and $j$-th disc anticlockwise around each other, for some $1 \leq i < j \leq m$. Embed that into $(S^1)^m \times \mathit{Conf}_m$ by taking the $(S^1)^m$ component to be constant. Then, the monodromy of $\scrY_{\mathit{flat}}(\underline{S},\{a_i\}) \rightarrow (S^1)^m \times \mathit{Conf}_m$ along that loop is given by adding the element 
\begin{equation} \label{eq:ij-vector}
(0,\dots,\overbrace{1-\mathrm{rot}(a_j)}^{\text{$i$-th position}},\dots,\overbrace{1-\mathrm{rot}(a_i)}^{\text{$j$-th position}},\dots,0) \in \bZ^m.
\end{equation}
\end{addendum}

To see this, take the loop $d_{ij}$ which encloses exactly the $i$-th and $j$-th discs. In other words, it divides the surface into pieces $S_{\mathit{left}}$ and $S_{\mathit{right}}$, of which $S_{\mathit{right}}$ has boundary consisting exactly of $d_{ij}$, $\partial_i S$ and $\partial_j S$. We orient $d_{ij}$ compatibly with $\partial S_{\mathit{right}}$, which means that its Poincar{\'e} dual in $H^1(S,\partial S)$ is minus the $i$-th and $j$-th unit vectors, analogously to \eqref{eq:rot-ai}. By Addendum \ref{th:t-class}, the pullback action of $\tau_{d_ij}^{-1}$ is given by adding
\begin{equation} \label{eq:rot-aij}
-T(\tau_{d_{ij}}) = (1 - \mathrm{rot}(a_i) - \mathrm{rot}(a_j))[d_{ij}].
\end{equation}
The monodromy we are looking for (corresponding to moving the discs, but not changing the angles $\rho_i$ or $\rho_j$) is actually the pullback action of $\tau_{d_{ij}}^{-1}\tau_{d_i}\tau_{d_j}$, where $d_i$ and $d_j$ are as in \eqref{eq:rot-ai}. Combining that equation with \eqref{eq:rot-aij} yields the desired result. Of course, by virtue of Proposition \ref{th:decorations}(ii), all this carries over to spaces $\scrY_{\geq 0}(\underline{S},\{a_i\})$.

\subsection{The TQFT axioms}
Fix a commutative coefficient ring $R$. We suppose that the following data are given. For every $a \in \Omega^1(S^1,\frakg)$ with hyperbolic holonomy (our standing assumption), one has a chain complex of $R$-modules, denoted by $C^*(a)$; its cohomology will be written as $H^*(a)$. For any family of decorated disc configurations whose base is an oriented compact manifold with corners $P$, and whose boundary values are $a_0,\dots,a_m$, one has a map
\begin{equation} \label{eq:operation}
C^*(a_1) \otimes \cdots \otimes C^*(a_m) \longrightarrow C^{*-\mathrm{dim}(P)}(a_0).
\end{equation}
We will denote \eqref{eq:operation} by $\phi_P$, which is concise if somewhat incomplete.
The behaviour of these maps is governed by three (groups of) axioms. These are very similar to those in \cite{seidel16}, hence will be only stated briefly:
\begin{itemize} \itemsep.5em

\item Exchanging the labeling of the $m$ discs permutes the inputs of the associated operation \eqref{eq:operation}, with the expected Koszul signs. Reversing the orientation of $P$ flips the sign of \eqref{eq:operation}. Taking the disjoint union of two parameter spaces yields the sum of the associated operations. Finally, pulling back a family by a finite covering $\tilde{P} \rightarrow P$ of degree $\mu$ yields 
\begin{equation} \label{eq:covering-axiom}
\phi_{\tilde{P}} = \mu \phi_P
\end{equation}
(the last part, concerning coverings, will be needed only once, in the construction underlying Proposition \ref{th:bigraded-lie}).

\item The coboundary (defined using the differentials on the chain complexes) of \eqref{eq:operation} equals the sum of the operations associated to the codimension one boundary faces of $P$. If those boundary faces are written as $\partial_j P$, then the sign convention is that 
\begin{equation} \label{eq:boundary-sign}
\begin{aligned}
& (-1)^{\mathrm{dim}(P)} d\phi_P(x_1,\dots,x_m) - \phi_P(dx_1,\dots,x_m) - 
\cdots 
\\ & \qquad \qquad \cdots - (-1)^{|x_1|+\cdots+|x_{m-1}|} \phi_P(x_1,\dots,dx_m) + \sum_j \phi_{\partial_j P}(x_1,\dots,x_m) = 0.
\end{aligned}
\end{equation}

\item Gluing of families corresponds to composition of operations. More precisely, suppose that we have families of $m_k$-disc configurations, with parameter spaces $P_k$ ($k = 1,2$), and glue the second one into the $i$-th place of the first one. The composition axiom says that 
\begin{equation} \label{eq:tqft}
\begin{aligned}
\phi_{P_1 \times P_2}(x_1,\dots, x_{m_1+m_2-1}) & = 
(-1)^{(\mathrm{dim}(P_1)+|x_1|+\cdots+|x_{i-1}|)\mathrm{dim}(P_2)}
\\ & \phi_{P_1}(x_1, \dots, x_{i-1},
\phi_{P_2}(x_i,\dots,x_{i+m_2-1}),\dots,x_{m_1+m_2-1}).
\end{aligned}
\end{equation}
Moreover, the identity configuration induces the identity map $C^*(a) \rightarrow C^*(a)$.
\end{itemize}

We will now explore the consequences of these axioms, remaining on the level of the cohomology groups $H^*(a)$. Take Proposition \ref{th:x-annulus}(i) in the case $a_0 = a_1 = a$. This implies that $H^*(a)$ carries an action of $H_{-*}(\Gamma_r)$, for $r = \mathrm{rot}(a)$, or in a more concrete formulation:

\begin{proposition} \label{th:auto}
$H^*(a)$ carries the following endomorphisms:
\begin{equation} \label{eq:endomorphisms}
\left\{
\begin{aligned}
& \text{if $r = \mathrm{rot}(a) \neq 0$, an action of $\bZ/r$;} \\
& \text{if $r = 0$, a $\bZ$-action; and additionally, an endomorphism of degree $-1$.}
\end{aligned}
\right.
\end{equation}
\end{proposition}

We denote the generator of the $\bZ/r$-action by $\Sigma$. 
More generally, Proposition \ref{th:x-annulus}(ii), together with Addendum \ref{th:x-annulus-a}, implies the following:

\begin{prop} \label{th:well-defined}
(i) Given $a_0,a_1$ with the same rotation number $r$, there is an isomorphism 
\begin{equation} \label{eq:01-iso}
H^*(a_1) \iso H^*(a_0),
\end{equation}
which commutes with the action of \eqref{eq:endomorphisms}, and is unique up to composition with powers of $\Sigma$.

(ii) If the holonomies of $a_0$ and $a_1$ lie in the same one-parameter semigroup, meaning that \eqref{eq:semigroup} holds, there is a distinguished choice of \eqref{eq:01-iso}, and these choices are compatible with composition.
\end{prop}

The situation for unequal rotation numbers is quite different. Applying Proposition \ref{th:x-annulus}(iii) yields the following:

\begin{prop} \label{th:canonical-continuation}
Suppose that $\mathrm{rot}(a_0) = \mathrm{rot}(a_1) + 1$. Then, there are canonical maps
\begin{align}
\label{eq:a-continuation}
& C: H^*(a_1) \longrightarrow H^*(a_0), \\
\label{eq:pseudo-bv}
& D: H^*(a_1) \longrightarrow H^{*-1}(a_0).
\end{align}
The ``continuation map'' \eqref{eq:a-continuation} is invariant under composition with $\Sigma$ on the left or right. It is also compatible with (any choice of) the isomorphisms relating different choices of $a_0$ or $a_1$, with the same rotation number. The map \eqref{eq:pseudo-bv} has the same properties, and additionally, it anticommutes with \eqref{eq:a-continuation}.
\end{prop}

In a slightly flamboyant formulation, one could say that the continuation maps are more robustly independent of the choice of $a$ than the groups $H^*(a)$ themselves. From Proposition \ref{th:the-disc}, we get:

\begin{prop} \label{th:unit}
Suppose that $\mathrm{rot}(a) = 1$. Then there is a distinguished element
\begin{equation} \label{eq:e}
e \in H^*(a).
\end{equation}
For different choices of $a$, these elements are related by \eqref{eq:01-iso} (which is canonical for this particular rotation number).
\end{prop}

We now turn to multilinear operations:

\begin{prop} \label{th:gerstenhaber-algebra}
For $\mathrm{rot}(a) = 1$, $H^*(a)$ carries a canonical structure of a Gerstenhaber algebra, for which \eqref{eq:e} is the unit.
\end{prop}

\begin{proof}
Take the universal family with boundary conditions $\{a,\dots,a\}$. Consider the $\bZ^m$-covering which is the middle row of  \eqref{eq:monodromy-diagram}. Addendum \ref{th:twist-loop} and \eqref{eq:rot-ai} determine this covering, which turns out to be the universal cover of the $(S^1)^m$ factor. This means that the map
\begin{equation} \label{eq:f-map}
\scrY_{\mathit{flat}}(\underline{S},a,\dots,a) \longrightarrow (S^1)^m \times \mathit{Conf}_m
\end{equation}
has the following properties: its first component is homotopic to a constant, and its second component is a weak homotopy equivalence. Given any map $P \rightarrow \mathit{Conf}_m$, one therefore has a homotopically unique lift to $\scrY_{\mathit{flat}}(\underline{S},a,\dots,a)$. More concretely, a single configuration of two discs ($P = \mathit{point}$) induces the product; two-disc configurations moving anticlockwise around each other ($P = S^1$) induce the bracket; and all the necessary relations between these operations can be implemented by explicit cobordisms, which means that they come from relations in the groups $\mathit{MSO}_d(\mathit{Conf}_m)$ (for $d \leq 2$). We omit the details, since the relation between the little disc operad and Gerstenhaber algebras is classical.

In the construction of the Gerstenhaber structure, we used flat connections, but of course, the same argument goes through with a more general choice of nonnegatively curved ones, thanks to Proposition \ref{th:decorations}(ii). Nonnegatively curved connections appear naturally when discussing unitality, where they arise from the definition of \eqref{eq:e}. The proof of unitality goes as follows: take two copies of the disc which defines \eqref{eq:e}, and glue them into the pair-of-pants which defines the product. After appealing to Proposition \ref{th:the-disc}, we conclude that the outcome again defines \eqref{eq:e}. In other words, $e$ is an idempotent. Now take the same pair-of-pants as before, and glue in just one copy of the disc. The outcome is an annulus, whose associated operation, multiplication with $e$, is necessarily some power of $\Sigma$. Because $e$ is idempotent and multiplication is associative, it follows that $e$ is a unit.
\end{proof}

\begin{addendum} \label{th:all-a}
Another application of nonnegatively curved connections is a mild generalization of Proposition \ref{th:gerstenhaber-algebra}, where one considers ``all boundary conditions with rotation number $1$ simultaneously''. What this means is that for any such $a_k$, there are maps
\begin{align}
\label{eq:012-product}
& \cdot: H^*(a_1) \otimes H^*(a_2) \longrightarrow H^*(a_0), \\
\label{eq:012-bracket}
& [\cdot,\cdot]: H^*(a_1) \otimes H^*(a_2) \longrightarrow H^{*-1}(a_0).
\end{align}
These satisfy the Gerstenhaber algebra axioms, for all situations where the compositions make sense. For instance, the associativity equation
\begin{equation}
(x_1x_2)x_3 = x_1(x_2x_3) \in H^*(a_0), \quad \text{where } x_k \in H^*(a_k)
\end{equation}
always holds, where the intermediate terms can be taken to lie in arbitrary groups, $x_1x_2 \in H^*(a_4)$ and $x_2x_3 \in H^*(a_5)$. Similarly, inserting a unit element on either the left or right into \eqref{eq:012-product} always results in the isomorphism \eqref{eq:01-iso}. From this, it follows that \eqref{eq:01-iso} is an isomorphism of Gerstenhaber algebras.
\end{addendum}


Now fix $a_1,a_2$, with $\mathrm{rot}(a_1) = 1$, and $\mathrm{rot}(a_2) = r$ arbitrary. Consider the universal family of $m$-disc configurations, and its decorations, where the connections are flat and have boundary values $(a_2,a_1,\dots,a_1,a_2)$. We can use Addendum \ref{th:two-agree-3} to pick a preferred subspace
\begin{equation} \label{eq:map-to-gamma}
\scrY_{\mathit{flat}}^0(\underline{S},a_2,a_1,\dots,a_1,a_2) \subset \scrY_{\mathit{flat}}(\underline{S},a_2,a_1,\dots,a_1,a_2), 
\end{equation}
namely that on which the map \eqref{eq:rot-mod-2} vanishes (and in the case where $r = 0$, we require additionally that the angle $\arg(\rho_m)$ should be trivial).  In other words, we use \eqref{eq:rot-mod-2} to provide a trivialization of the fibration which forms the bottom row of \eqref{eq:monodromy-diagram}. The resulting homotopy equivalence $\scrY_{\mathit{flat}}^0(\underline{S},a_2,a_1,\dots,a_1,a_2) \rightarrow \mathit{Conf}_m$ gives a distinguished choice of decorations (up to homotopy) for each disc configuration. This choice is compatible with the two gluing processes that will be relevant for us:
\begin{itemize}
\item \itemsep.5em
Suppose that we take two configurations, with boundary values $(a_2,a_1,\dots,a_1,a_2)$ and $(a_1,\dots,a_1)$, respectively, and glue the second one into the first one (in one of the possible ways). The rotation number \eqref{eq:rot-mod-2} of the first configuration remains unchanged under gluing. Hence, this gluing process preserves the subspace $\scrY_{\mathit{flat}}^0(\underline{S},a_2,a_1,\dots,a_1,a_2)$.

\item 
Suppose that we take two decorated disc configurations, both having boundary values $(a_2,a_1,\dots,a_1,a_2)$, and glue them (in the unique possible way). This gluing process adds up the rotation numbers \eqref{eq:rot-mod-2}. Hence, again, if the original configurations lie in $\scrY_{\mathit{flat}}^0(\underline{S},a_2,a_1,\dots,a_1,a_2)$, then so does the glued one.
\end{itemize} 
Let's pass to concrete implications. Using a two-disc configuration, respectively an $S^1$-family of such configurations, and their lifts to $\scrY_{\mathit{flat}}^0(\underline{S},a_2,a_1,a_2)$, one defines a product and bracket
\begin{align}
\label{eq:r-product}
& \cdot: H^*(a_1) \otimes H^*(a_2) \longrightarrow H^*(a_2), \\
\label{eq:r-bracket}
& [\cdot,\cdot]: H^*(a_1) \otimes H^*(a_2) \longrightarrow H^{*-1}(a_2).
\end{align}

\begin{lemma} \label{th:rot-bracket}
The family of decorated disc configurations underlying \eqref{eq:r-bracket} has the following geometry: two discs move around each other; during that rotation, the framing of the first disc rotates $(1-r)$ times, while that of the second one remains constant (up to homotopy).
\end{lemma}

\begin{proof}
Consider the middle row of \eqref{eq:monodromy-diagram}, specialized to our case:
\begin{equation} \label{eq:m-2}
\bZ^2 \longrightarrow \scrY_{\mathit{flat}}(\underline{S},a_2,a_1,a_2) \longrightarrow (S^1)^2 \times \mathit{Conf}_2.
\end{equation}
By construction, the holonomy around the two $S^1$ factors is $(-\mathrm{rot}(a_1),0)$ and $(0,-\mathrm{rot}(a_2))$, respectively. The operation \eqref{eq:r-bracket} is defined using a loop in $\scrY_{\mathit{flat}}^0(\underline{S},a_2,a_1,a_2)$ which consists of the two discs moving around each other once, and where the framings rotate by some amounts $d_1,d_2 \in \bZ$. The image of that loop in $(S^1)^2 \times \mathit{Conf}_2$ must be a loop along which \eqref{eq:m-2} has trivial monodromy. On the other hand, by Addendum \ref{th:twist-loop}, that monodromy is
\begin{equation}
(1-\mathrm{rot}(a_2), 1-\mathrm{rot}(a_1)) + d_1 (-\mathrm{rot}(a_1),0) + d_2 (0,-\mathrm{rot}(a_2)) \\
 = (1-r-d_1,-rd_2).
\end{equation}
If $r \neq 0$, it follows that
\begin{equation} \label{eq:rot12}
d_1 = 1-r, \;\; d_2 = 0.
\end{equation}
If $r = 0$, we know a priori that $d_2 = 0$, by definition of $\scrY_{\mathit{flat}}^0(\underline{S},a_2,a_1,a_2)$, which leads to the same conclusion \eqref{eq:rot12}. 
\end{proof}

\begin{prop} \label{th:gerstenhaber-module}
The operations \eqref{eq:r-product} and \eqref{eq:r-bracket} make $H^*(a_2)$ into a (unital) Gerstenhaber module over the Gerstenhaber algebra $H^*(a_1)$. For $a_2 = a_1$, this is the diagonal module.
\end{prop}

\begin{proof}
Let's consider the desired relations (the defining properties of a Gerstenhaber module; see e.g.\ \cite[Definition 1.1]{kowalzig-kraehmer14}, which however uses different sign conventions), for $x_1,x_2 \in H^*(a_1)$ and $x_3 \in H^*(a_2)$:
\begin{align}
\label{eq:r-relation-1}
& (x_1 x_2) x_3 = x_1 (x_2 x_3), \\
\label{eq:r-relation-2}
& 
[x_1,x_2x_3] =  [x_1,x_2] x_3 + (-1)^{(|x_1|-1)|x_2|} x_2[x_1,x_3], \\
\label{eq:r-relation-4}
& [[x_1,x_2],x_3] + (-1)^{|x_1|} [x_1,[x_2,x_3]] + (-1)^{(|x_1|-1)|x_2|} [x_2,[x_1,x_3]] = 0.
\end{align}
Because the choices of decorations, taken to lie in $\scrY_{\mathit{flat}}^0(\underline{S},a_2,a_1,a_1,a_2)$, are homotopically unique, all we need to show is that the underlying relations hold up to bordism in $\mathit{Conf}_3$. 

The only nontrivial case is \eqref{eq:r-relation-4}, since there the framings, as analyzed in Lemma \ref{th:rot-bracket}, affect the families obtained from the gluing process. To spell this out in the simplest possible terms, let's temporarily ignore the sizes of discs, and just think of $\mathit{Conf}_3$ as the ordered three-point configuration space in $\bC$. The only thing that depends on $r$ is the first term in \eqref{eq:r-relation-4}, which (up to homotopy) comes from the family
\begin{equation}
S^1 \times S^1 \longrightarrow \mathit{Conf}_3, \quad
(p_1,p_2) \longmapsto (0, \half e^{2\pi i (1-r) p_1 + 2\pi i p_2}, e^{2\pi i p_1}).
\end{equation}
Even there, the $r$-dependence can be removed by a change of parameters $p_2 \mapsto p_2 + (r-1)p_1$. The remaining argument is therefore the same as in the case $r = 1$, which is standard (and indeed, already entered into Proposition \ref{th:gerstenhaber-algebra}).

Even though our definition of \eqref{eq:map-to-gamma} was restricted to flat connections, there is an alternative formulation which works more generally for nonnegatively curved ones. Namely, given a decorated $m$-disc configuration with boundary values $(a_2,a_1,\dots,a_1,a_2)$, glue in $(m-1)$ discs with boundary value $a_1$. The outcome is an annulus, with boundary values $(a_2,a_2)$, to which one can associate an element in $\bZ/r$ as in Proposition \ref{th:x-annulus-a}(ii). The subspace $\scrY_{\geq 0}^0(\underline{S},a_2,a_1,\dots,a_1,a_2)$ is defined by asking that to be zero. One can use this approach to prove unitality (as in Proposition \ref{th:gerstenhaber-algebra}), and also to prove that the Gerstenhaber module structure is independent of the choice of $a_1$ and $a_2$ (as in Addendum \ref{th:all-a}). More precisely, it is compatible with the canonical isomorphisms between different choices of $a_1$; the isomorphisms between different choices of $a_2$ are not canonical, but any such choice is a map of Gerstenhaber modules.
\end{proof}

\begin{addendum} \label{th:cop-out}
The indirect algebraic argument for unitality (based on first showing $e^2 = e$) used in Propositions \ref{th:gerstenhaber-algebra} and \ref{th:gerstenhaber-module} may strike the reader as a cop-out. We will therefore outline a geometric alternative, which uses the paths constructed in Lemma \ref{th:unit-path}. Take the pair-of-pants surface $S_1$, and equip it with a flat connection whose holonomies around the boundary components, parametrized as in \eqref{eq:boundary-parametrizations}, are $(g_2(1),g_1(1),g_2(1))$. More precisely, this connection should have boundary values $(a_2,a_1,a_2)$, where $\mathrm{rot}(a_1) = 1$; and additionally, parallel transport along a path from $\epsilon_0(0)$ to $\epsilon_1(0)$ is $k(1)$, whereas parallel transport along a path from $\epsilon_1(0)$ to $\epsilon_2(0)$ is the identity $\Id$ (we refer to Figure \ref{fig:unit}, with $t$ set to $1$, for the precise picture). The TQFT operation associated to $S_2$ defines the product $H^*(a_1) \otimes H^*(a_2) \rightarrow H^*(a_2)$. Using the nonnegative path $g_1(t)$, one can equip the disc $S_2$ with a nonnegatively curved connection with boundary value $a_1$. Gluing together $S_1$ and $S_2$ yields the geometric description of the action of \eqref{eq:e} on $H^*(a_2)$ (to check this, one uses the fact that $k(1)$ has rotation number $0$, which is part of Lemma \ref{th:unit-path}).

There is a one-parameter family of deformations  of this construction, where the boundary holo\-no\-mies on $S_1$ are $(g_2(t),g_1(t),g_2(t))$, and the parallel transport maps are $k(t)$ resp.\ the identity (see again Figure \ref{fig:unit}). Correspondingly, on $S_2$ one has a connection with boundary holonomy $g_1(t)$. After gluing, one gets a one-parameter family of connections on the cylinder, whose boundary values can still be assumed to be $(a_2,a_2)$ (since the conjugacy class of the holonomies remains the same). For $t = 0$, the outcome is a flat connection on the cylinder, with the additional property that parallel transport along a path joining the two boundary circles is $k(0) = \Id$. Hence, the associated operation is the identity map on $H^*(a_2)$.
\end{addendum}
\begin{figure}
\begin{centering}
\begin{picture}(0,0)%
\includegraphics{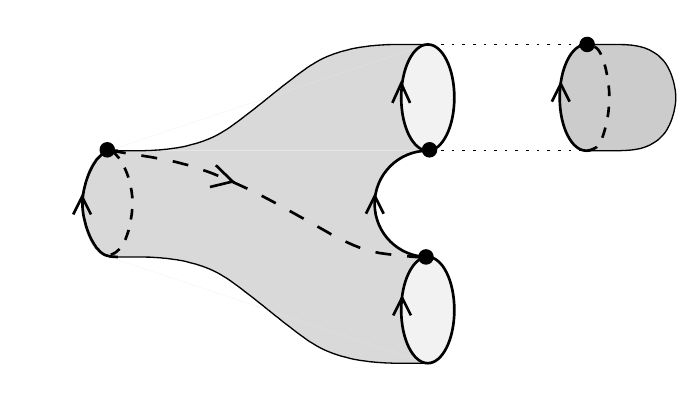}%
\end{picture}%
\setlength{\unitlength}{3355sp}%
\begingroup\makeatletter\ifx\SetFigFont\undefined%
\gdef\SetFigFont#1#2#3#4#5{%
  \reset@font\fontsize{#1}{#2pt}%
  \fontfamily{#3}\fontseries{#4}\fontshape{#5}%
  \selectfont}%
\fi\endgroup%
\begin{picture}(3827,2274)(1186,-2084)
\put(2551,-1186){\makebox(0,0)[lb]{\smash{{\SetFigFont{9}{10.8}{\rmdefault}{\mddefault}{\updefault}{$k(t)$}%
}}}}
\put(3506,-2011){\makebox(0,0)[lb]{\smash{{\SetFigFont{9}{10.8}{\rmdefault}{\mddefault}{\updefault}{$g_2(t)$}%
}}}}
\put(3511, 19){\makebox(0,0)[lb]{\smash{{\SetFigFont{9}{10.8}{\rmdefault}{\mddefault}{\updefault}{$g_1(t)$}%
}}}}
\put(1201,-1186){\makebox(0,0)[lb]{\smash{{\SetFigFont{9}{10.8}{\rmdefault}{\mddefault}{\updefault}{$g_2(t)$}%
}}}}
\put(4351,-836){\makebox(0,0)[lb]{\smash{{\SetFigFont{9}{10.8}{\rmdefault}{\mddefault}{\updefault}{$g_1(t)$}%
}}}}
\put(3401,-1011){\makebox(0,0)[lb]{\smash{{\SetFigFont{9}{10.8}{\rmdefault}{\mddefault}{\updefault}{$\Id$}%
}}}}
\end{picture}%
\caption{\label{fig:unit}The geometric version of the proof that $e$ is a unit (Remark \ref{th:cop-out}). The flat connection on the pair-of-pants with the given parallel transport maps exists thanks to \eqref{eq:conjugating-path}.}
\end{centering}
\end{figure}
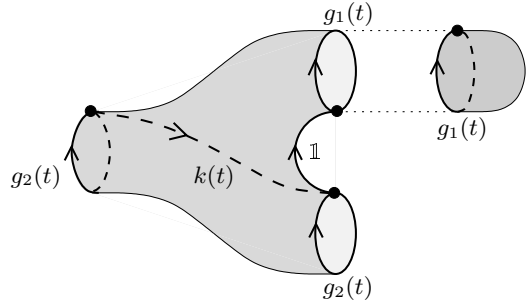%

There are more restrictive notions of Gerstenhaber module than that used above, which are obtained by imposing a relation on $[x_1 x_2,x_3]$. The relation that holds in our context depends on the rotation number:

\begin{prop} \label{th:strange-relation}
The Gerstenhaber module structure from Proposition \ref{th:gerstenhaber-module} satisfies
\begin{equation} \label{eq:strange-relation}
[x_1x_2,x_3] = (1-r)[x_1,x_2]x_3 + (-1)^{(|x_1|-1)|x_2|} x_2[x_1,x_3] + (-1)^{|x_1|} x_1[x_2,x_3].
\end{equation}
\end{prop}

\begin{proof}
The idea is the same as for \eqref{eq:r-relation-4}, with simpler underlying geometry. For $1 \leq i < j \leq 3$, let $\lambda_{ij} \in H_1(\mathit{Conf}_3)$ be the class of a loop where the $i$-th and $j$-th point rotate anticlockwise once around each other (while the other point is far away and remains constant). Because of Lemma \ref{th:rot-bracket}, the geometric object underlying $[x_1x_2,x_3]$ is a loop of configurations whose homology class is $(1-r) \lambda_{12} + \lambda_{13} + \lambda_{23}$. The equality \eqref{eq:strange-relation} then amounts to giving each of those three summands a separate geometric representation.
\end{proof}

\begin{remark}
As an analogue of Proposition \ref{th:strange-relation} for $r = 0$, consider the Gerstenhaber algebra $A^* = \Lambda^*(TX)$ of polyvector fields on a manifold $X$, and the Gerstenhaber module structure of the space $M^* = \Omega^{-*}(X)$ of differential forms. The first nontrivial instances of this structure (omitting signs) are as follows:
\begin{equation} \label{eq:schouten}
\begin{array}{l|ll}
& \text{algebra action on $\theta \in M^*$} & \text{Lie action on $\theta \in M^*$} \\
\hline 
\text{functions $f \in A^0$} & \theta \mapsto f\theta & \theta \mapsto df \wedge \theta \\
\text{vector fields $Z \in A^1$} & \theta \mapsto i_Z\theta & \theta \mapsto L_Z\theta
\end{array}
\end{equation}
Then, \eqref{eq:strange-relation} for $x_1 = f$, $x_2 = Z$, $x_3 = \theta$ (and $r = 0$) becomes
\begin{equation}
L_{fZ} \theta = (Z.f) \theta - i_Z (df \wedge \theta) + f (L_Z \theta),
\end{equation}
which is indeed an easy consequence of the standard Cartan formalism. More generally, 
suppose that we look at a version of the Gerstenhaber module which consists of differential forms twisted by the $(-r)$-th power of canonical bundle (the bundle of top degree forms). With respect to \eqref{eq:schouten}, the difference is that the Lie action of $Z \in A^1$ is now written as (with respect to a local volume form $\eta$) 
\begin{equation}
\theta \otimes \eta^{-r} \longmapsto \Big( L_Z \theta -  r \theta \frac{L_Z\eta}{\eta} \Big) \otimes \eta^{-r}.
\end{equation}
This satisfies an instance of \eqref{eq:strange-relation} for the given value of $r$. If, instead of differential topology, one considers the same formulae in algebraic geometry, this would be (roughly speaking) the $B$-model mirror of our symplectic geometry constructions.
\end{remark}

For the next construction of operations, we switch to a more constrained setup. Fix some $a$ with $\mathrm{rot}(a) = 1$. All our boundary conditions will be covers $a^{r_i}$, where the degrees $r_i>0$ satisfy $r_1 + \dots + r_m - r_0 = m-1$. Moreover, the disc configurations will be such that all discs have centers on the imaginary axis, ordered (say) from bottom to top. We also restrict to the trivial choice of framings, meaning that $\rho_i > 0$. For each $i = 1,\dots,m$, fix a path $c_i$, going from $\epsilon_0(0) = 1 \in \partial_0 S$ to $\epsilon_i(1)$, and which lies in $S \cap \{\mathrm{re}(z) \geq 0\}$; this assumption fixes the homotopy class of the path (rel endpoints). As explained in Addendum \ref{th:partial-fix}, we can use such paths to define a map
\begin{equation} \label{eq:m-rotations}
(\mathrm{rot}_{c_1}, \dots, \mathrm{rot}_{c_m}): \scrA_{\mathit{flat}}(S,\{a_i\}) \longrightarrow \bZ^m
\end{equation}
which, by Proposition \ref{th:connections-rel-boundary}, is a weak homotopy equivalence. For any framed disc configuration in our class, there is a choice of $A$ for which \eqref{eq:m-rotations} vanishes. Such choices are unique up to homotopy, and moreover, they are compatible with gluing discs.  This leads to the following:

\begin{prop} \label{th:bigraded-algebra}
Fix some $a$ with $\mathrm{rot}(a) = 1$. Then, $\bigoplus_{s \geq 0} H^*(a^{s+1})$ has the structure of a bigraded associative algebra, for which \eqref{eq:e} is the unit.
\end{prop}

\begin{proof}
Because of our restrictions, the spaces of choices involved are weakly homotopy equivalent to ordered configuration spaces of points on the real line, which are of course contractible. A single configuration of two discs, with $(r_0 = s_1+s_2+1,\, r_1 = s_1+1,\, r_2 = s_2+1)$, defines the desired product structure. Associativity follows by gluing and considering the resulting three-disc configurations. Since the product $H^*(a) \otimes H^*(a^{s+1}) \rightarrow H^*(a^{s+1})$ agrees with the one previously defined in Proposition \ref{th:gerstenhaber-module}, we already know that $e$ is a left unit; the proof for right unitality is parallel.
\end{proof}

\begin{remark}
Even though we will not discuss that explicitly, a refined version of our arguments coud be used to construct algebraic structures on the cochain level: for instance, by looking at families parametrized by associahedra, one can get a bigraded $A_\infty$-algebra structure underlying that from Proposition \ref{th:bigraded-algebra}.
\end{remark}

\subsection{Multivalued framings\label{subsec:multivalued}}
For the final piece of algebraic structure we want to extract from our TQFT, it is convenient to think of disc configurations in a slightly different way. 
\begin{itemize} \itemsep.5em
\item
When choosing boundary values $a \in \Omega^1(S^1,\frakg)$, we allow only ones with $r = \mathrm{rot}(a) \geq 1$. Moreover, each such $a$ must be invariant under $1/r$-rotation of the circle. Equivalently, it must be an $r$-fold cover, in the sense of \eqref{eq:a01-cover} (of some underlying connection, which necessarily has rotation number $1$).

\item 
Our disc configurations will come with multivalued framings, which are maps like \eqref{eq:boundary-parametrizations}, but for which the $\rho_i$ are determined only up to multiplication with $r_i$-th roots of unity; here, $r_i$ is the rotation number of the associated boundary value $a_i$.
\end{itemize}
Given a disc configuration, let $\scrY_{\mathit{multi}, \geq 0}(S,\{a_i\})$ be the space of multivalued decorations, and $\scrY_{\mathit{multi, flat}}(S,\{a_i\})$ the subspace where the connection is flat. The analogue of Proposition \ref{th:decorations} (with a similar proof, which we omit) is the following:

\begin{prop} \label{th:multi-contractible}
If \eqref{eq:m-minus-1} holds, both $\scrY_{\mathit{multi,flat}}(S,\{a_i\})$ and $\scrY_{\mathit{multi},\geq 0}(S,\{a_i\})$ are weakly contractible.
\end{prop}

The relevant version of the gluing process goes as follows. Start with two families of configurations with multivalued decorations, parametrized by $P_1$ and $P_2$, respectively, and whose associated rotation numbers are $(r_{1,0},\dots,r_{1,m_1})$ and $(r_{2,0},\dots,r_{2,m_2})$. Assume additionally that $P_2$ comes with an (orientation-preserving) action of $\bZ/r_{2,0}$, which is such that acting by the generator corresponds to rotating the disc configuration (including multivalued decorations) by $1/r_{2,0}$. Suppose that we have an $1 \leq i \leq m_1$ such that $r_{1,i} = r_{2,0}$, and such that the relevant boundary values agree. Gluing produces a family over a twisted product, which is a fibration 
\begin{equation} \label{eq:twisted-product}
P_2 \longrightarrow P_1 \tilde\times_i P_2 \longrightarrow P_1.
\end{equation}
More precisely, there is a $\bZ/r_{1,i}$-covering $\tilde{P}_{1,i} \rightarrow P_1$, which corresponds to choosing lifts of the multivalued framing of the $i$-th disc to an actual framing. One then takes the quotient
\begin{equation} \label{eq:twisted-product-2}
\tilde{P}_{1,i} \times P_2 \longrightarrow P_1 \tilde\times_i P_2 \stackrel{\text{def}}{=} \tilde{P}_{1,i} \times_{\bZ/r_{2,0}} P_2.
\end{equation}
It is straightforward to see that the gluing process, defined a priori on $\tilde{P}_{1,i} \times P_2$, indeed descends to the quotient. Finally, note that if $P_1$ also carries an action of $\bZ/r_{1,0}$ which rotates disc configurations, then the glued family inherits that action.

\begin{lemma} \label{th:fractional-loop}
Consider two discs moving around each other, parametrized by $p \in S^1$. Using Proposition \ref{th:multi-contractible}, we can equip these discs with multivalued decorations, which vary smoothly along our loop, for any  rotation numbers $(r_0, r_1, r_2)$ such that $r_0 = r_1+r_2-1$. The associated multivalued framings are, up to homotopy,
\begin{equation} \label{eq:exact-framings}
\left\{
\begin{aligned}
& \epsilon_1(t) = \rho_1 e^{2\pi i (1-r_2)r_1^{-1} p - 2\pi i t} = 
\rho_1 e^{2\pi i (1-r_0r_1^{-1}) p - 2\pi i t}, \\
& \epsilon_2(t) = e^{2\pi i p} \zeta_2 + \rho_2 e^{2\pi i (1-r_1)r_2^{-1} p - 2\pi i t} \rho_2 z = 
e^{2\pi i p} \zeta_2 + \rho_2 e^{2\pi i (1-r_0r_2^{-1}) p - 2\pi i t}.
\end{aligned}
\right.
\end{equation}
(For the formula to make sense, $|\rho_1|, |\rho_2|$ need to be sufficiently small, and $\zeta_2 \neq 0$). In words, as the discs move, the first one rotates by a fractional amount $1-r_0/r_1$, and the second one by $1-r_0/r_2$.
\end{lemma}

This is the analogue of Lemma \ref{th:rot-bracket}, and is proved in the same way, by appealing to Lemma \ref{th:twist-loop}; we omit the argument. 

\begin{addendum} \label{th:fractional-lie-family}
We may in fact choose the framings to be exactly given by \eqref{eq:exact-framings}. Then, rotating such a configuration by $1/r_0$ yields
\begin{equation} \label{eq:rotated-conf}
\left\{
\begin{aligned}
& e^{2\pi i r_0^{-1}} \epsilon_1(t) = \rho_1 e^{2\pi i r_1^{-1}} e^{2\pi i (1-r_0r_1^{-1}) (p+r_0^{-1}) - 2\pi i t}, \\
& e^{2\pi i r_0^{-1}} \epsilon_2(t) = e^{2\pi t}(p+r_0^{-1}) \zeta_2 + \rho_2 e^{2\pi i r_2^{-1}} e^{2\pi i (1-r_0r_2^{-1}) (p+ r_0^{-1}) - 2\pi i t}.
\end{aligned}
\right.
\end{equation}
By definition, the $e^{2\pi i r_k^{-1}}$ factors on the right hand side of \eqref{eq:rotated-conf} are irrelevant for multivalued framings. It follows that, after a suitable choice of connections, our family admits a $\bZ/r_0$-action of the kind needed in the gluing process, namely $p \mapsto p+r_0^{-1}$.
\end{addendum}

Let's assume that the coefficient ring of our TQFT satisfies $\bQ \subset R$. Consider families of disc configurations with multivalued decorations. Given such a family, with parameter space $P$, there is a cover $\tilde{P} \rightarrow P$ with structure group $\bZ/r_1 \times \cdots \times \bZ/r_m$, $r_i = \mathrm{rot}(a_i)$, on which the framings become single-valued. We use those covers and our usual TQFT operation \eqref{eq:operation} to define
\begin{equation} \label{eq:divide-operation}
\phi_{\mathit{multi},P} = r_1^{-1}\cdots r_m^{-1} \phi_{\tilde{P}}.
\end{equation}
The induced cohomology level map
\begin{equation} \label{eq:multivalued-operation}
H^*(\phi_{\mathit{multi},P}): H^*(a_1) \otimes \cdots \otimes H^*(a_m) \longrightarrow H^{*-\mathrm{dim}(P)}(a_0)
\end{equation}
is invariant under the action of $\bZ/r_1 \times \cdots \times \bZ/r_m$ on its source. If we assume in addition that our family carries an action of $\bZ/r_0$ rotating the discs, then \eqref{eq:multivalued-operation} invariant under the action of that group on its target. To understand the behaviour of these invariants under gluing, we consider the operation associated to \eqref{eq:twisted-product-2}. If we take the families of decorated disc configurations over $\tilde{P}_1$ and $\tilde{P}_2$ and glue them together, then the outcome descends to $\tilde{P}_1 \times_{\bZ/r_{2,0}} \tilde{P}_2$. As a consequence of the axiom \eqref{eq:covering-axiom} concerning coverings, and the composition axiom, we get, in abbreviated notation,
\begin{equation}
\phi_{\tilde{P}_1 \times_{\bZ/r_{2,0}} \tilde{P}_2} = r_{2,0}^{-1} \phi_{\tilde{P}_1 \times \tilde{P}_2} = r_{2,0}^{-1} \phi_{\tilde{P}_1} \phi_{\tilde{P}_2}, 
\end{equation}
and hence
\begin{equation}
\phi_{\mathit{multi},P_1 \tilde{\times}_i P_2} = \phi_{\mathit{multi},P_1} \phi_{\mathit{multi},P_2}.
\end{equation}
Here, ``abbreviated notation'' means that these composition formulae are shorthand for appropriate analogues of \eqref{eq:tqft} (with the same signs). We have been heading for the following application:

\begin{prop} \label{th:bigraded-lie}
Suppose that the coefficient ring $R$ is a $\bQ$-algebra. For each $r > 0$, fix some $a^r \in \Omega^1(S^1,\frakg)$ with rotation number $r$, and which is an $r$-fold cover (not necessarily of the same underlying connection). Then, $\bigoplus_{s \geq 0} H^{*+1}(a^{s+1})^{\bZ/(s+1)}$ has the structure of a bigraded Lie algebra.
\end{prop}

\begin{proof}
Consider a loop of two-disc configurations as in Lemma \ref{th:fractional-loop}. One gets a map
\begin{equation}
H^*(\phi_{\mathit{multi},P}): H^*(a^{r_1}) \otimes H^*(a^{r_2}) \longrightarrow H^{*-1}(a^{r_0}),
\end{equation}
where $r_0 = r_1+r_2-1$. By definition, this map is $\bZ/r_1$-invariant, which means that it remains the same if we compose it with projection (defined by averaging over orbits) to the invariant part of $H^*(a^{r_1})$. In other words, no information is lost if we restrict it to that invariant part. The same applies to $\bZ/r_2$ acting on $H^*(a^{r_2})$. Finally, the map takes value in the $\bZ/r_0$-invariant part of $H^*(a^{r_0})$. We use it to define our Lie bracket. 

Each term in the Jacobi identity corresponds to a family of disc configurations with multivalued decorations, parametrized by $T^2$. Thanks to Proposition \ref{th:multi-contractible}, we only have to check that these families satisfy the desired relation in the configuration space $\mathit{Conf}_3$. 
Let's look at the expression $[x_1,[x_2,x_3]]$. Geometrically, this is given by gluing together two copies of the family from Lemma \ref{th:fractional-loop}, with rotation numbers $(r_{1,0} = r_0, r_{1,1} = r_1, r_{1,2} = r_2+r_3-1)$ and $(r_{2,0} = r_2+r_3-1, r_{2,1} = r_2, r_{2,2} = r_3)$ and parameters $p_1,p_2$. Again using Lemma \ref{th:fractional-loop}, the relevant covering is
\begin{equation}
\tilde{P}_{1,2} = \{ (\tilde{p}_1, p_1) \in S^1 \times S^1 \;:\; (r_2+r_3-1)\tilde{p}_1 = (1-r_1) p_1 \}.
\end{equation}
With this at hand, the families over $\tilde{P}_{1,2}$ and $P_2$ are:
\begin{equation} \label{eq:p12}
\begin{aligned} 
&
\left\{
\begin{aligned}
& \epsilon_{1,1}(t) = \rho_{1,1} e^{2\pi i (2-r_2+r_3) r_1^{-1} p_1 - 2\pi i t}, \\
& \epsilon_{1,2}(t) = e^{2\pi i p_1} \zeta_{1,2} + \rho_{1,2} e^{2\pi i \tilde{p}_1 - 2\pi i t}.
\end{aligned}
\right.
\\ &
\left\{
\begin{aligned}
& \epsilon_{2,1}(t) = \rho_{2,1} e^{2\pi i (1-r_3)r_2^{-1} p_2 - 2\pi i t} , \\
& \epsilon_{2,2}(t) = e^{2\pi i p_2} \zeta_{2,2} + \rho_{2,2} e^{2\pi i (1-r_2)r_3^{-1} p_2 - 2\pi i t}.
\end{aligned}
\right.
\end{aligned}
\end{equation}
The common action of $\bZ/(r_2+r_3-1)$ by which we have to divide in \eqref{eq:twisted-product} is generated by $(\tilde{p}_1,p_1, p_2) \mapsto (\tilde{p}_1 + (r_2+r_3-1)^{-1}, p_1, p_2 - (r_2+r_3-1)^{-1})$. Hence, the quotient is parametrized by $(p_1, \tilde{p}_2 = \tilde{p}_1+p_2)$. The glued family is then of the form
\begin{equation} \label{eq:glue-3}
\left\{
\begin{aligned}
& \epsilon_1(t) = \rho_1 e^{2\pi i (2-r_2+r_3) r_1^{-1} p_1 - 2\pi i t}, \\
& \epsilon_2(t) = e^{2\pi i p_1} \zeta_2 + \rho_2 e^{2\pi i ((1-r_3) \tilde{p}_2 + (1-r_1) p_1)r_2^{-1} - 2\pi i t}, \\
& \epsilon_3(t) = e^{2\pi i p_1} \zeta_2 + e^{2\pi i \tilde{p}_2} \zeta_3 + \rho_3 e^{2\pi i ((1-r_1) p_2 + (1-r_2) \tilde{p}_2)r_3^{-1} - 2\pi i t}.
\end{aligned}
\right.
\end{equation}
We have not spelled out the relation between the constants appearing here and those in \eqref{eq:p12}; the only relevant point is that $|\zeta_3| < |\zeta_2|$, so that the three discs in our configuration (in the given order) form a sun-earth-moon system. In particular, if we ignore framings, and just focus on the position of the celestial bodies, then those are independent of the rotation number. The same holds for the other terms of the Jacobi identity, which are just permuted versions of \eqref{eq:glue-3}; hence, the proof of that identity reduces to a standard argument.
\end{proof}

\section{The differentiation axiom\label{sec:diff-axiom}}

Following the model of \cite{seidel16}, we now build an additional differentiation operator into the TQFT setup. Because of the intended application, there is also a shift in emphasis, from the cohomology level to the chain level (still remaining fairly ad hoc, rather than giving a description of all the resulting chain level operations).

\subsection{The formalism}
Assume that our coefficient ring $R$ comes with a derivation $\partial: R \rightarrow R$; and that all $C^*(a)$ come with additive graded endomorphisms
\begin{equation} \label{eq:qabla}
\begin{aligned}
& \qabla: C^*(a) \longrightarrow C^*(a), \\
& \qabla(fx) = f (\qabla x) + (\partial f) x \quad \text{for $f \in R$, $x \in C^*(a)$.}
\end{aligned}
\end{equation}
Note that we do not ask that $\qabla$ should commute with the differential. These endomorphisms interact with the geometry as follows:
\begin{itemize}
\itemsep.5em
\item The surface \eqref{eq:surface} associated to a decorated disc configuration is allowed to optionally carry a marked point $Z \in S$ (in the exceptional case of the identity configuration, this is a point on the circle.) This additional datum behaves in the obvious way with respect to gluing (where of course only one of the two surfaces involved may carry a marked point). For a family of configurations over $P$, we similarly allow an optional smooth section $\underline{Z}: P \rightarrow \underline{S}$. If our family includes instances of the identity configuration, we ask that in some neighbourhood of the subset of $P$ where that configuration occurs, the section $\underline{Z}$ should take values in the boundary component $\partial_0\underline{S}$. For the associated TQFT operations, we keep the same axioms as before, treating the additional marked point as formally having degree $2$. This means that a family over a compact oriented $P$ with a $\underline{Z}$ section gives rise to an operation
\begin{equation} \label{eq:phi-map-with-z-point}
\phi_P: C^*(a_1) \otimes \cdots \otimes C^*(a_m) \longrightarrow C^{*-\mathrm{dim}(P)+2}(a_0).
\end{equation}

\item
Consider the identity configuration, for some $a$, and turn this into a family over $P = S^1$ by adding the point $Z_p = e^{2\pi i p}$. We denote the associated operation \eqref{eq:phi-map-with-z-point} by
\begin{equation} \label{eq:r-map}
R: C^*(a) \longrightarrow C^{*+1}(a), 
\end{equation}
and require that it should measure the failure of \eqref{eq:qabla} to be a chain map:
\begin{equation} \label{eq:differentiation-1}
R= \qabla d - d \qabla.
\end{equation}

\item
Consider any family $\underline{S} \rightarrow P$, where none of the fibres is the identity configuration, and which does not come with a marked point (section). We can associate to this a new family, denoted by $\underline{S}_\bullet \rightarrow P_\bullet$, as follows. The new parameter space is the total space of the original family, $P_\bullet = \underline{S}$, and $\underline{S}_\bullet$ is the pullback of $\underline{S}$ to that space. This comes with a tautological section $\underline{Z}$. We call this process ``adding a marked point in all possible places'', and require that the associated operation should be related to the original one by
\begin{equation} \label{eq:differentiation-2}
\phi_{P_\bullet}(x_1,\dots,x_m) = \qabla \phi_P(x_1,\dots,x_m) - \phi_P(\qabla x_1,\dots,x_m) - \cdots - \phi_P(x_1,\dots,\qabla x_m).
\end{equation}
\end{itemize}

The second and third part of the requirements above form the ``differentiation axiom''. We will now explore some immediate consequences. Take the empty configuration ($m = 0$, which means that $S = D$), with the point $Z = 0 \in S$. After equipping this with a nonnegatively curved connection whose boundary value $a$ has rotation number $1$ (Proposition \ref{th:the-disc}), we get a cocycle
\begin{equation} \label{eq:ks}
k \in C^2(a).
\end{equation}
We call this the Kodaira-Spencer cocycle, and its class in $H^2(a)$ the Kodaira-Spencer class. 

\begin{lemma} \label{th:bracket-with-k}
Fix $a_1,a_2$ as in Proposition \ref{th:gerstenhaber-module}. Consider a chain level version of the bracket constructed there,
\begin{equation} \label{eq:chain-bracket}
[\cdot,\cdot]: C^*(a_1) \otimes C^*(a_2) \longrightarrow C^{*-1}(a_2).
\end{equation}
Let $k$ be the Kodaira-Spencer cocycle in $C^2(a_1)$. Then there is a chain homotopy
\begin{equation} \label{eq:rho-homotopy}
\begin{aligned}
& \rho: C^*(a_2) \longrightarrow C^*(a_2), \\
& d\rho(x) - \rho(dx) - R(x) + [k,x] = 0.
\end{aligned}
\end{equation}
\end{lemma}

\begin{proof}
Take \eqref{eq:ks} for $a = a_1$, and the family underlying \eqref{eq:chain-bracket}; glue them together; and then (temporarily) forget the marked point. The outcome is a family, still parametrized by $p \in P = S^1$, in which the surfaces $S_p$ are all annuli, carrying boundary parametrizations and nonnegatively curved connections, with boundary values $(a_2,a_2)$. Proposition \ref{th:x-annulus}(ii) and Addendum \ref{th:x-annulus-a}(ii) say that up to homotopy, our family is classified by a map
\begin{equation} \label{eq:nullhomotopic-s1}
P \longrightarrow \Gamma_r. 
\end{equation}
Assume first that $r \neq 0$, so that $\Gamma_r$ is discrete. Let's look at a single value of the parameter. Then, the gluing construction precisely describes the module action of \eqref{eq:e} from Proposition \ref{th:gerstenhaber-module}. In Remark \ref{th:cop-out}, we have given a direct geometric proof of the unitality of that action; that argument implies that gluing yields the trivial element of $\Gamma_r$, hence that \eqref{eq:nullhomotopic-s1} is trivial. In the remaining case $r = 0$, what we have just said shows that \eqref{eq:nullhomotopic-s1} takes values in the trivial connected component of $\Gamma_0 = \bZ \times S^1$; one can then use Lemma \ref{th:rot-bracket} to show that the map $\pi_1(P) \rightarrow \pi_1(\Gamma_0)$, which describes the $p$-dependence of the boundary parametrization of our annuli, is also trivial. 

We now know that the glued family can be deformed to a one that is constant in $p$ except for the marked point, which moves once around the annulus. For concreteness, let's give a specific formula for what we can deform our family to (as decorated disc configurations):
\begin{equation} \label{eq:124}
\left\{
\begin{aligned}
& D_1 = \quarter D, \\
& \epsilon_1(t) = \quarter e^{-2\pi i t}, \\
& A = \text{the pullback of $a_2$ by radial projection}, \\
& Z_p = \half e^{2\pi i p}.
\end{aligned}
\right.
\end{equation}
By deforming the constants $\quarter$ and $\half$ to $1$, \eqref{eq:124} can be connected to the family underlying $R$. Combining this with the previous deformation yields a family parametrized by an annulus, whose associated TQFT operation $\rho$ precisely satisfies \eqref{eq:rho-homotopy}, as an instance of \eqref{eq:boundary-sign}.
\end{proof}


In principle, these constructions should be supplemented by suitable uniqueness statements (since we are on the cochain level, this always means uniqueness up to chain homotopy). As a particularly simple instance, \eqref{eq:ks} is homotopically unique, by which we mean the following: if $k$ and $\tilde{k}$ are the cochains arising from two choices of decorations, then
\begin{equation} \label{eq:secondary-k}
\tilde{k}-k = d\kappa;
\end{equation}
and $\kappa \in C^1(a)$, which comes from a family of decorations parametrized by $[0,1]$, is itself unique up to adding coboundaries of elements in $C^0(a)$ (one could continue this with higher homotopies). This follows immediately from the contractibility statement in Proposition \ref{th:the-disc}.
Similarly (but with a slightly more complicated argument), $\rho$ is homotopically unique, and there is a relation between the maps $\rho$ obtained from different choices of cochain representative $k$. Finally, $k$ and $\rho$ are compatible, in a suitable sense, with the quasi-isomorphisms relating the complexes $C^*(a)$ for different choices of $a$. We omit the details.

\subsection{The connection}
From now on, our discussion no longer applies to the TQFT formalism in general, but requires a special additional assumption:

\begin{assumption} \label{th:ks-vanishes}
Fix some $a \in \Omega^1(S^1,\frakg)$ with rotation number $1$. Assume that the class of $k$ in $H^2(a)$ is nullhomologous. Given that, we choose a bounding cochain $\theta \in C^1(a)$, $d\theta = k$. Two such cochains will be considered equivalent if their difference is a nullhomologous cocycle.
\end{assumption}

By adding arbitrary cocycles $t \in C^1(a)$ to $\theta$, the set of equivalence classes becomes an affine space over $H^1(a)$. Note that equivalence classes of $\theta$ can be transferred from one representative of $k$ to another: given \eqref{eq:secondary-k} and a choice of $\theta$ such that $d\theta = k$, one sets $\tilde\theta = \theta + \kappa$. Similarly, one can pass between different $a$.

\begin{proposition} \label{th:hyperbolic-connection}
Suppose that Assumption \ref{th:ks-vanishes} holds. Then, any group $H^*(a_2)$ carries a differentiation operator
\begin{equation} \label{eq:nabla-h}
\begin{aligned}
& \nabla: H^*(a_2) \longrightarrow H^*(a_2), \\
& \nabla(fx) = f \nabla x + (\partial f) x.
\end{aligned}
\end{equation}
This operator depends on the equivalence class of the bounding cochain $\theta$; changing that by some $t \in H^1(a)$ has the effect of adding the Lie action $x \mapsto -[t,x]$ to \eqref{eq:nabla-h}.
\end{proposition}

\begin{proof}
The differentiation axiom and Lemma \ref{th:bracket-with-k} show that the Kodaira-Spencer class encodes the essential failure of $\qabla$ to be a chain map. Correspondingly, under Assumption \ref{th:ks-vanishes}, one can remedy that failure:
\begin{equation} \label{eq:nabla}
\begin{aligned}
& \nabla: C^*(a_2) \longrightarrow C^*(a_2), \\
& \nabla x = \qabla x + \rho(x) - [\theta,x]
\end{aligned}
\end{equation}
is a chain map. We define \eqref{eq:nabla-h} to be the induced operation on cohomology.
\end{proof}

There is a uniqueness statement for \eqref{eq:nabla-h}, proving the independence of $\nabla$ from all choices made its construction, and its compatibility with the isomorphisms between the groups $H^*(a_2)$ for different choices of $a_2$. The argument is relatively straighforward, and we omit it. Going further along the lines of \cite{seidel16}, one could consider the compatibility of $\nabla$ with each of the previously defined algebraic structures; since we have no immediate need or applications for such compatibility results, we prefer not to pursue that line of inquiry here.

\section{Elliptic holonomy\label{sec:elliptic}}
We now enlarge our TQFT by additionally allowing, as boundary values, connections with elliptic (which by definition still means nontrivial) holonomy. This allows us to incorporate the parallel theory from \cite{seidel16}. There is little new in the strategies of proof; we will therefore focus on a few basic technical statements, and then on the main result, which is the comparison theorem for connections (Proposition \ref{th:comparison-of-connections}).

\subsection{Spaces of nonnegative paths, revisited\label{subsec:additional}}
We temporarily return to the subject of Section \ref{sec:sl2}. Our first observation concerns the issue of perturbing nonnegative paths to positive ones, while keeping the conjugacy class of the endpoints the same. We have encountered this implicitly before, for instance, in the proof of Lemma \ref{th:short-path}: there, our path ended in the hyperbolic locus, and we used the fact that positive motion on hyperbolic conjugacy classes is unconstrained. Here is an alternative approach, which uses the elliptic locus instead:

\begin{lemma} \label{th:make-positive}
Let $g(t)$ be a nonconstant nonnegative path, such that $g(1)$ is elliptic. Then there is a deformation $g_s(t)$, with $g_0(t) = g(t)$, such that:
\begin{equation} \label{eq:per-pos}
\left\{
\begin{aligned}
& g_s(0) = g(0), \\
& \text{$g_s(1)$ lies in the same conjugacy class as $g(1)$,} \\
& \text{$g_s(t)$ is a positive path, for all $s>0$.}
\end{aligned}
\right.
\end{equation}
The same statement applies to families of paths.
\end{lemma}

\begin{proof}
This is an adaptation of the argument from Lemma \ref{th:path}. Let's first assume, for simplicity, that our path remains entirely inside the elliptic locus, so that $g(t)$ is a rotation with angle $\theta(t) \in (0,\pi)$. Because the path is nonnegative and not constant, we know from \eqref{eq:trace-decreases} that $\theta(1) > \theta(0)$. Consider time-dependent functions \eqref{eq:elliptic-gamma}. There is such a $\gamma_t$ which additionally satisfies $g'(t) = g(t)\gamma_t(g(t))$. There is also one more choice $\tilde\gamma_t$, obtained from Lemma \ref{th:bound-vector-fields}(i) by rescaling as in \eqref{eq:bounded-fn}, which has the additional property that $\tilde\gamma_t$ takes values in $\frakg_{> 0}$ whenever $\theta'(t) > 0$. Take the paths $\tilde{g}_s(t)$ obtained by solving the ODE
\begin{equation}
\tilde{g}_s(0) = g(0), \quad \tilde{g}_s'(t) = \tilde{g}_s(t)\big((1-s)\gamma_t(g_s(t)) + s\tilde{\gamma}_t(g_s(t))\big).
\end{equation}
By construction, these move through the same conjugacy classes as the original path, to which they specialize for $s = 0$. For $s>0$, the path $\tilde{g}_s$ is positive at all times $t$ where $\theta'(t)>0$. Now we ``redistribute the positivity'':
\begin{equation}
\begin{aligned} \label{eq:redistribute}
& g_s(t) =\begin{pmatrix} \cos(\alpha_s(t)) & -\sin(\alpha_s(t)) \\ \sin(\alpha_s(t)) & \cos(\alpha_s(t))
\end{pmatrix} \tilde{g}_s(t), \\
&
\text{where} \quad
\left\{ 
\begin{aligned}
& \alpha_0(t) = 0, \\
& \alpha_s(0) = \alpha_s(1) = 0, \\
& \text{for $s>0$, $\alpha_s'(t)>0$ at all $t$ where $\tilde{g}_s(t)$ is not (strictly) positive.}
\end{aligned} 
\right.
\end{aligned}
\end{equation}
If we choose $\alpha_s$ sufficiently small, \eqref{eq:redistribute} will be a positive path for all $s>0$, satisfying \eqref{eq:per-pos}. 

In the general case, fix some $T$ such that $g|[T,1]$ is elliptic; one can also arrange that there is a point in $[T,1]$ where the angle of rotation has positive derivative. Now, apply the first part of the argument to $g|[T,1]$, but multiplying $\tilde{\gamma}_t$ by a cutoff function which vanishes on $[0,T]$. The outcome is a family of paths $\tilde{g}_s$ which agree with the original $g$ on $[0,T]$, and which are positive somewhere in $[T,1]$. We then apply the same ``redistribution'' argument to those paths. The parametrized situation is similar.
\end{proof}

\begin{lemma} \label{th:elliptic-to-elliptic}
Fix elliptic conjugacy classes $C_0,C_1 \subset G$. Let $r_0,r_1 \in \bR \setminus \bZ$ be rotation numbers of some lifts of those classes to $\tilde{G}$, chosen so that $(r_0,r_1)$ contains exactly one integer, and set $\rho = r_1-r_0$. Then, $\scrP_{\geq 0}(G,C_0,C_1)^\rho$ is weakly contractible.
\end{lemma}

\begin{proof}[Sketch of proof]
Lemma \ref{th:make-positive} shows that $\scrP_{\geq 0}(G,C_0,C_1)^\rho$ is weakly homotopy equivalent to its subspace $\scrP_{>0}(G,C_0,C_1)^\rho$ of positive paths, so we'll consider that instead. From Figure \ref{fig:traffic}, one sees that paths in $\scrP_{>0}(G,C_0,C_1)^\rho$ can meet the trivial element $\Id$ at most once. Let's write
\begin{equation} \label{eq:pcc}
\scrP_{>0}(G,C_0,C_1)^\rho = \scrS_{\mathit{open}} \cup \scrS_{\mathit{id}},
\end{equation}
where $\scrS_{\mathit{id}}$, the subspace of paths which go through $\Id$, is a codimension two submanifold; and $\scrS_{\mathit{open}}$ is its complement. 

{\bf Claim (i):} {\em $\scrS_{\mathit{id}}$ is weakly contractible.}

Except for the time $T \in (0,1)$ where $g(T) = \Id$, a path $g(t)$ in $\scrS_{\mathit{id}}$ must remain in the elliptic locus. To first order near $T$, the path is determined by $g'(T) \in \frakg_{>0}$, which is a contractible choice. On the other hand, Lemma \ref{th:path}(i) applies to $g|[0,T-\epsilon]$ and $g|[T+\epsilon,1]$ for any $\epsilon>0$, which are therefore determined uniquely up to homotopy by their endpoints $g(T \pm \epsilon)$. Combining these two ideas yields the claimed fact.

{\bf Claim (ii):} {\em $\scrS_{\mathit{open}}$ is weakly homotopy equivalent to a circle.}

Any path $g(t)$ in $\scrS_{\mathit{open}}$ has the following structure:
\begin{equation}
\text{for some $0<T_0<T_1< 1$,}\; 
\left\{
\begin{aligned}
& \text{$g(t)$ is elliptic for $t<T_0$;} \\
& \text{$g(T_0)$ is a nonpositive parabolic;} \\
& \text{$g(t)$ is hyperbolic for $T_0<t<T_1$;} \\
& \text{$g(T_1)$ is a nonnegative parabolic;} \\
& \text{$g(t)$ is elliptic for $T_1<t$.}
\end{aligned}
\right.
\end{equation}
A more precise form of our claim says that, by evaluating such paths at a point in $(T_0,T_1)$ (smoothly depending on the path), we get a weak homotopy equivalence between $\scrS_{\mathit{open}}$ and the hyperbolic locus $G_{\mathit{hyp}}$. Indeed, by breaking up the path at such a point, we can reduce our claim to (a fibre product of) two copies of Lemma \ref{th:short-path}, in a version for positive rather than nonnegative paths.

With that settled, we will now show that $\scrP_{>0}(G,C_0,C_1)^\rho$ is weakly contractible. Take a path in $\scrS_{\mathit{id}}$, and a two-parameter family of perturbations of that path, which moves it in all directions transverse to $\scrS_{\mathit{id}}$. As a local model near the identity element, if the original path $g(t)$ is a standard rotation by $t$, the family could be 
\begin{equation}
g_{\delta,\epsilon}(t) = \exp \left( \begin{smallmatrix} \epsilon & \delta-t \\ \delta+t & -\epsilon \end{smallmatrix} \right). 
\end{equation}
For $(\delta,\epsilon)$ nontrivial, $g_{\delta,\epsilon}(0)$ is a hyperbolic element, whose eigenvector for the larger eigenvalue $\exp(\sqrt{\delta^2+\epsilon^2})$ is 
\begin{equation}
\begin{pmatrix} (\sqrt{\delta^2+\epsilon^2} + \epsilon)^{1/2} \\ \mathrm{sign}(\delta) (\sqrt{\delta^2+\epsilon^2} - \epsilon)^{1/2}\end{pmatrix}.
\end{equation}
As $(\delta,\epsilon)$ moves in a small circle around $0$, that eigenvector goes once (up to sign, since there's no preferred orientation of the $(\delta,\epsilon)$ plane) around $\bR P^1$. The same holds for general paths and perturbations (by a deformation argument). This shows that the link of the stratum $\scrS_{\mathit{id}}$ yields a weak equivalence from the circle to $\scrS_{\mathit{open}}$. Weak contractibility of $\scrP_{>0}(G,C_0,C_1)^\rho$ follows from that and the previous claims.
\end{proof}

\begin{lemma} \label{th:ehe}
Fix conjugacy classes $C_0,C_1,C_2$ which are, respectively, elliptic, hyperbolic, and elliptic. For the elliptic ones, we require that they correspond to rotations with angles $\theta_0 > \theta_2$ in $(0,\pi)$. Consider $g_i \in C_i$ together with a nonnegative path $h(t)$ from $h(0) = g_1g_2$ to $h(1) = g_0$, such that if we take lifts of our elements to $\tilde{g}_i \in \tilde{G}$, compatible with the product and path, the numbers $\mathrm{rot}(\tilde{g}_0)$ and $\mathrm{rot}(\tilde{g}_1)+\mathrm{rot}(\tilde{g}_2)$ lie in the same component of $\bR \setminus \bZ$. Then, the space of all such $(g_0,g_1,g_2,h(t))$ is weakly homotopy equivalent to $C_1$, by projecting to $g_1$.
\end{lemma}

\begin{proof}[Sketch of proof]
We will consider positive paths instead of nonnegative ones. This does not affect the weak homotopy type, by the same argument as in Lemma \ref{th:short-path}, applied to the path $h(t)g_2^{-1}$ (which goes from the hyperbolic element $g_1$ to $g_0g_2^{-1}$). Suppose (without loss of generality) that $\mathrm{rot}(\tilde{g}_0) \in (0,1)$, $\mathrm{rot}(\tilde{g}_2) \in (-1,0)$, which means that  $\mathrm{rot}(\tilde{g}_1) = 1$. One can apply the elementary considerations from Example \ref{th:2-elliptic} to the pair $(\tilde{g}_0,\tilde{g}_2^{-1})$, and that describes all the possible options for $\tilde{g}_0\tilde{g}_2^{-1}$. In particular, we see that the function $(\tilde{g}_0,\tilde{g}_2) \mapsto \mathrm{tr}(g_0g_2^{-1})$ is a submersion everywhere except at its maximum value (which is achieved when $g_0,g_2$ commute). Let's decompose the space $\scrS$ of all $(g_0,g_1,g_2,h(t))$ into three parts $\scrS_{\mathit{ell}}$, $\scrS_{\mathit{para}}$, $\scrS_{\mathit{hyp}}$, depending on the conjugacy type of $g = g_0g_2^{-1}$. The two and last subsets are open, and $\scrS_{\mathit{para}}$ is a hypersurface separating them.

{\bf Claim (i):} {\em The projection $\scrS_{\mathit{hyp}} \rightarrow C_1$ is a weak homotopy equivalence.}

Given any hyperbolic element $g$, there is a contractible space of positive paths from $C_1$ to $g$ along which the rotation number is constant, by (the positive path version of) Lemma \ref{th:pre-short}. On the other hand, by Example \ref{th:2-elliptic}, there is an $\bR$ worth of choices of $(g_0,g_2)$ such that $g_0g_2^{-1} = g$. Hence, the map $g_0g_2^{-1}: \scrS_{\mathit{hyp}} \rightarrow G_{\mathit{hyp}}$ is a weak homotopy equivalence; but it is also homotopic to the map appearing in our claim.

{\bf Claim (ii):} {\em The projection $\scrS_{\mathit{para}} \rightarrow C_1$ is a weak homotopy equivalence.}

This is similar to the previous case. It requires a statement about positive paths from hyperbolic to nonnegative parabolic elements, but that can be obtained from Lemma \ref{th:pre-short} and a separate consideration of positive paths near the parabolic endpoint, as in Lemma \ref{th:short-path-3}.

{\bf Claim (iii):} {\em The projection $\scrS_{\mathit{ell}} \rightarrow C_1$ is a weak homotopy equivalence.} 

Let's work with a fixed hyperbolic $g_1$. Given an elliptic $g$, the space of positive paths connecting $g_1$ to $g$, along which the rotation number increases by less than $1$, is weakly contractible, by (the positive path version of) Lemma \ref{th:short-path}. Hence, the homotopy fibre of $\scrS_{\mathit{ell}} \rightarrow C_1$ is weakly homotopy equivalent to the space of $(g_0,g_2)$ such that $g_0g_2^{-1}$ is elliptic. By the computation in Example \ref{th:2-elliptic}, this space is an open disc.

Putting together the three claims above yields the desired result.
\end{proof}

\subsection{Spaces of connections, revisited}
As in Section \ref{sec:connections}, statements about nonnegative paths translate into ones about nonnegatively curved connections. We use the same notation $\scrA_{\geq 0}(S,\{a_i\})$ as in \eqref{eq:moduli-spaces-2}, but now allowing boundary values $a_i$ with elliptic as well as hyperbolic holonomy.

\begin{prop} \label{th:grow-elliptic}
Let $S = [0,1] \times S^1$. As boundary conditions, fix $a_0,a_1 \in \Omega^1(S^1,\frakg)$ with elliptic holonomy, such that $\mathrm{rot}(a_0) \geq \mathrm{rot}(a_1)$, and where the interval $(\mathrm{rot}(a_1), \mathrm{rot}(a_0))$ contains at most one integer. Then $\scrA_{\geq 0}(S,a_0,a_1)$ is contractible.
\end{prop}

\begin{prop} \label{th:mixed-continuation}
(i) Let $S = [0,1] \times S^1$. As boundary conditions, fix $a_0 \in \Omega^1(S^1,\frakg)$ with hyperbolic holonomy, $a_1 \in \Omega^1(S^1,\frakg)$ with elliptic holonomy, and such that $\mathrm{rot}(a_1) \in (\mathrm{rot}(a_0)-1,\mathrm{rot}(a_0))$. Then $\scrA_{\geq 0}(S,a_0,a_1)$ is weakly contractible.

(ii) The same is true if $a_0$ has elliptic holonomy, $a_1$ has hyperbolic holonomy, and $\mathrm{rot}(a_0) \in (\mathrm{rot}(a_1),\mathrm{rot}(a_1)+1)$.
\end{prop}

Both statements are proved by the same strategy as Proposition \ref{th:nonnegative-connection}. The relevant results about path spaces are: for Proposition \ref{th:grow-elliptic}, Lemma \ref{th:path}(i) or Lemma \ref{th:elliptic-to-elliptic} (depending on whether the relevant interval contains an integer or not); and for Proposition \ref{th:mixed-continuation}, Lemma \ref{th:short-path}.

\begin{prop} \label{th:elliptic-disc}
Let $S$ be a disc, with the boundary parametrized clockwise. As boundary condition, fix $a \in \Omega^1(S^1,\frakg)$ with elliptic holonomy and $\mathrm{rot}(a) \in (0,1) \cup (1,2)$. Then $\scrA_{\geq 0}(S,a)$ is weakly contractible.
\end{prop}

One can reduce Proposition \ref{th:elliptic-disc} to Lemma \ref{th:path}(i) or Lemma \ref{th:elliptic-to-elliptic} (depending on which interval $\mathrm{rot}(a)$ lies in), just like Proposition \ref{th:the-disc} was derived (via Lemma \ref{th:short-path-3}) from Lemma \ref{th:short-path}. Our final result follows from Lemma \ref{th:ehe}, in the same way as in Proposition \ref{th:nonnegative-connection-4}:

\begin{prop} \label{th:ehe-pants}
Let $S$ be a pair-of-pants, with boundary parametrizations as in Proposition \ref{th:nonnegative-connection-4} (specialized to $m = 2$). As boundary conditions, take $a_0, a_2 \in \Omega^1(S^1,\frakg)$ with elliptic holonomy, and $a_1 \in \Omega^1(S^1, \frakg)$ with hyperbolic holonomy, such that for some $r \in \bZ$,
\begin{equation}
\mathrm{rot}(a_0) \in (\mathrm{rot}(a_2)+1,r+2), \;\; \mathrm{rot}(a_1) = 1, \;\; \mathrm{rot}(a_2) \in (r,r+1).
\end{equation}
Then $\scrA_{\geq 0}(S, a_0,a_1,a_2)$ is weakly contractible.
\end{prop}

\subsection{The formal setup}
We continue in the same axiomatic framework as in Section \ref{sec:tqft}, except that now, the boundary conditions are allowed to have holonomies that are elliptic or hyperbolic. This gives rise to additional cohomology groups $H^*(a)$, whose basic properties can be summarized as follows:

\begin{proposition} \label{th:only-elliptic}
(i) Consider connections $a \in \Omega^1(S^1,\frakg)$ with elliptic holonomy. Up to canonical isomorphism, $H^*(a)$ depends only on $\mathrm{rot}(a) \in \bR \setminus \bZ$. 

(ii) Each such group carries a degree $-1$ endomorphism, the BV operator $\Delta$.

(iii) For $a_0, a_1$ as in Proposition \ref{th:grow-elliptic}, there is a canonical continuation map
$C: H^*(a_1) \rightarrow H^*(a_0)$.

(iv) If $\mathrm{rot}(a) \in (0,1) \cup (1,2)$, there is a canonical element $e \in H^0(a)$.
\end{proposition}

Parts (i) and (ii) can be constructed using flat connections on an annulus. The third part, obviously, uses Proposition \ref{th:grow-elliptic}; and (iv) follows from Proposition \ref{th:elliptic-disc}. 
Our next topic is the comparison between the elliptic and hyperbolic groups $H^*(a)$.

\begin{proposition} \label{th:mixed-continuation-2}
(i) Take $a_0,a_1$ as in Proposition \ref{th:mixed-continuation}(i). Then there is a canonical continuation map $C: H^*(a_1) \rightarrow H^*(a_0)$. In the case where $\mathrm{rot}(a_0) = 1$ and $\mathrm{rot}(a_1) \in (0,1)$, this takes the element $e \in H^0(a_1)$ from Proposition \ref{th:only-elliptic}(iv) to the previously constructed element \eqref{eq:e} of the same name in $H^0(a_0)$.

(ii) For $a_0,a_1$ as in Proposition \ref{th:mixed-continuation}(ii), there is a canonical continuation map as before. If $\mathrm{rot}(a_0) \in (1,2)$ and $\mathrm{rot}(a_1) = 1$, this takes the element \eqref{eq:e} in $H^0(a_1)$ to its counterpart from Proposition \ref{th:only-elliptic}(iv).

(iii) Composition of the maps from (i) and (ii), in either order, yields the continuation maps from \eqref{eq:a-continuation},  respectively from Proposition \ref{th:only-elliptic}(iii). Hence, in the limit we get an isomorphism
\begin{equation} \label{eq:lim-lim}
\underrightarrow{\lim}_{r \in \bZ}\, H^*(a^r) \iso \underrightarrow{\lim}_{r \in \bR \setminus \bZ}\, H^*(a^r).
\end{equation}
Here, on the left hand side of \eqref{eq:lim-lim}, we choose for each $r \in \bZ$ a connection $a^r$ with hyperbolic holonomy and rotation number $r$; the direct limit is independent of all choices up to canonical isomorphism, thanks to the properties of continuation maps (Proposition \ref{th:canonical-continuation}). On the right hand side, we do the same for elliptic holonomies.

(iv) The map \eqref{eq:pseudo-bv} factors through the BV operator. More precisely, given $a_0,a_1,a_2$ which, respectively, have hyperbolic, elliptic, and hyperbolic holonomy, and such that $\mathrm{rot}(a_2) < \mathrm{rot}(a_1) < \mathrm{rot}(a_0) = \mathrm{rot}(a_2)+1$, we have a commutative diagram
\begin{equation}
\xymatrix{
\ar@/_1pc/[rrr]_-{D} 
H^*(a_2) \ar[r]^-{C} & H^*(a_1) 
\ar[r]^-{\Delta}
& \ar[r]^-{C} H^{*-1}(a_1) & 
H^{*-1}(a_0).
}
\end{equation}
\end{proposition}

Existence and uniqueneness of these ``mixed'' continuation maps follows from Proposition \ref{th:mixed-continuation}. The statements about the unit elements are consequences of Propositions \ref{th:nonnegative-connection-3} and \ref{th:elliptic-disc}. In part (iii), the composition of continuation maps is given geometrically by connections on the annulus to which Propositions \ref{th:nonnegative-connection-2} or \ref{th:grow-elliptic} apply. The same applies to (iv), except that in that case, we are dealing with families of connections in which the parametrization of one of the boundary components rotates. 

Let's return for a moment to considering just connections with elliptic holonomy. The simplest way to approach the construction of further operations on the associated groups $H^*(a)$ is to significantly constrain the class of connections involved. Concretely, let's consider only those $a = a^r \in \Omega^1(S^1,\frakg)$ of the form 
\begin{equation} \label{eq:standard-connection}
a^r = r \alpha\, \mathit{dt}, \quad r \in \bR \setminus \bZ, \;\; \alpha = \left( \begin{smallmatrix} 0 & -\pi \\ \pi & 0 \end{smallmatrix} \right) \in \frakg_{>0}.
\end{equation}
These obviously satisfy $\mathrm{rot}(a^r) = r$. Similarly, on any surface \eqref{eq:surface}, we consider
\begin{equation} \label{eq:d-beta}
A = \beta \otimes \alpha \in \Omega^1(S,\frakg), \quad \text{for some } \beta \in \Omega^1(S) \text{ with } d\beta \leq 0.
\end{equation}
Additionally, the restriction of $\beta$ to each boundary circle must lie in the class \eqref{eq:standard-connection}, which means that it has constant coefficients (with respect to the given parametrization) and integral in $\bR \setminus \bZ$. The inequality in \eqref{eq:d-beta} expresses the nonnegative curvature condition (see \eqref{eq:curvature} for the sign). The advantage of this framework is that \eqref{eq:d-beta} is a convex condition, hence the choice of $\beta$ always lies in a contractible space (this is what one usually encounters in discussions of the operations on symplectic cohomology, see e.g.\ \cite[Section 8]{seidel07} or \cite{ritter10}). Moreover, the choice of boundary parametrizations \eqref{eq:boundary-parametrizations} now becomes independent of that of the connection. By applying this idea to the disc, annulus, and pair-of-pants, one gets operations
\begin{align}
\label{eq:standard-unit}
& e \in H^0(a^r), && r>0, \\
\label{eq:standard-continuation}
& C: H^*(a^{r_1}) \longrightarrow H^*(a^{r_0}), && r_0 \geq r_1, \\
\label{eq:standard-product}
& \cdot: H^*(a^{r_1}) \otimes H^*(a^{r_2}) \longrightarrow H^*(a^{r_0}), && r_0 \geq r_1+r_2.
\end{align}
The product \eqref{eq:standard-product} is associative and commutative in an appropriate sense, and inserting \eqref{eq:standard-unit} into one of its inputs recovers \eqref{eq:standard-continuation}. Of course, \eqref{eq:standard-unit} and \eqref{eq:standard-continuation} agree with the previously defined structures of the same name, in their domains of definition (here, we are using special choices of nonnegatively curved connections, whereas before we allowed any choice and used more sophisticated methods to prove well-definedness). We will be particularly interested in a further family of operations, the twisted brackets
\begin{equation} \label{eq:twisted-bracket}
[\cdot,\cdot]_c : H^*(a^{r_1}) \otimes H^*(a^{r_2}) \longrightarrow H^{*-1}(a^{r_0}), \quad r_0 \geq r_1+r_2, \;\; c \in \bZ.
\end{equation}
These are obtained from a one-parameter family of two-disc configurations, in which: the two discs move once around each other; the boundary parametrization of the first disc rotates $-c$ times, while that of the second disc remains constant. Concretely, one could take the dependence on the parameter $p \in P = S^1$ to be 
\begin{equation}
\left\{
\begin{aligned}
& \epsilon_1(t) = \zeta_1 e^{2\pi i p} + \rho_1 e^{-2\pi i (t + cp)}, \\
& \epsilon_2(t) = \rho_2 e^{-2\pi i t}.
\end{aligned}
\right.
\end{equation}
In the case $c = 0$, we write $[\cdot,\cdot]$, without the superscript. The bracket for general $c$ is related to that special case as in \eqref{eq:two-brackets}.

\subsection{A family of connections}
Into the enlarged framework including elliptic holonomy, we now re-introduce the differentiation axiom, as in Section \ref{sec:diff-axiom}. This yields Kodaira-Spencer classes for connections with both hyperbolic and elliptic holonomy, related to each other in the same way as the unit elements in Proposition \ref{th:mixed-continuation-2}. 

\begin{lemma} \label{th:bracket-with-elliptic-k}
Fix $r_0, r_1, r_2 \in \bR \setminus \bZ$, such that $r_1>0$ and $r_0 \geq r_1+r_2$, and take $a^{r_i}$ as in \eqref{eq:standard-connection}. Consider the following structures on the associated complexes $C^*(a^{r_i})$:
\begin{equation}
\left\{
\begin{aligned}
& \text{the Kodaira-Spencer cocycle $k \in C^2(a^{r_1})$,} \\
& \text{the continuation map $C: C^*(a^{r_2}) \rightarrow C^*(a^{r_0})$,} \\
& \text{the bracket $[\cdot,\cdot]_c: C^*(a^{r_1}) \otimes C^*(a^{r_2}) \rightarrow C^{*-1}(a^{r_0})$, for some $c \in \bZ$,} \\
& \text{the map $R: C^*(a^{r_2}) \rightarrow C^{*+1}(a^{r_2})$.}
\end{aligned}
\right.
\end{equation}
Then there is a chain homotopy
\begin{equation} \label{eq:rho-homotopy-elliptic}
\begin{aligned}
& \rho_c: C^*(a^{r_2}) \longrightarrow C^*(a^{r_0}), \\ 
& d\rho_c(x) - \rho_c(dx) - C(R(x)) + [k,x]_c  \\ & \quad = d\rho_c(x) - \rho_c(dx) - C \qabla (dx) + d C\qabla(x)  + [k,x]_c= 0.
\end{aligned}
\end{equation}
\end{lemma}

This is the analogue of Lemma \ref{th:bracket-with-k}. We omit the proof, which uses the same geometric construction as before; the details are actually simpler this time, because they just amount to choosing suitable one-forms \eqref{eq:d-beta}.

\begin{assumption} \label{th:ks-vanishes-2}
For some $r_1 \in \bR^{>0} \setminus \bN$, the Kodaira-Spencer cocycle $k \in C^2(a^{r_1})$ is nullhomologous. Given that, we choose a bounding cochain $\theta \in C^1(a^{r_1})$, $d\theta = k$.
\end{assumption}

\begin{proposition} \label{th:elliptic-connection-2}
Suppose that Assumption \ref{th:ks-vanishes-2} holds. Fix $r_0, r_2 \in \bR \setminus \bZ$ such that $r_0 \geq r_1+r_2$. Then, there is a family of ``connections'' with respect to the continuation map $C: H^*(a^{r_2}) \rightarrow H^*(a^{r_0})$, indexed by $c \in \bZ$. These are maps
\begin{equation} \label{eq:nabla-h-elliptic}
\begin{aligned}
& \nabla_c: H^*(a^{r_2}) \longrightarrow H^*(a^{r_0}), \\
& \nabla_c(fx) = f\, \nabla_c x + (\partial f) C(x).
\end{aligned}
\end{equation}
\end{proposition}

This is the counterpart of Proposition \ref{th:hyperbolic-connection}, using the bracket \eqref{eq:twisted-bracket} and the associated homotopy \eqref{eq:rho-homotopy-elliptic}: the chain map underlying \eqref{eq:nabla-h-elliptic} is
\begin{equation} \label{eq:elliptic-nabla}
\nabla_c x = C(\qabla x) + \rho_c(x) - [\theta,x]_c.
\end{equation}
We will not discuss how $\nabla_c$ depends on $c$ (but see \cite[Section 6]{seidel16}, which considers the relation between $c = -1$ and $c = 0$; the general case would work in the same way).

\begin{proposition} \label{th:comparison-of-connections}
Suppose that Assumption \ref{th:ks-vanishes} holds (and hence, so does Assumption \ref{th:ks-vanishes-2} for any $r_1>1$). Take $r_0, r_2 \in \bR \setminus \bZ$ such that $r_1 = r_0 - r_2 \in (1,2)$, and where the interval $(r_0,r_2)$ contains exactly one integer $r$. Fix some $a \in \Omega^1(S^1,\frakg)$ whose holonomy is hyperbolic, and which has rotation number $r$. Then, for appropriately coordinated choices of bounding cochains, \eqref{eq:nabla-h-elliptic} for $c = r-1$ factors through \eqref{eq:nabla-h}, as follows:
\begin{equation} \label{eq:comparison-of-connections}
\xymatrix{
\ar@/_1pc/[rrr]_-{\nabla_c} 
H^*(a^{r_2}) \ar[r]^-{C} & H^*(a) 
\ar[r]^-{\nabla}
& \ar[r]^-{C} H^*(a) & 
H^*(a^{r_0}).
}
\end{equation}
\end{proposition}

\begin{proof}
Each step in the proof is a fairly straightforward deformation argument, but there are a number of them. Part of the expository challenge is keeping the notation unambiguous, which we will attempt to do by adding rotation numbers as superscripts everywhere. In particular, for the given $r_i \in \bR \setminus \bZ$ ($i = 0,1,2$), we keep the notation $a^{r_i}$ for the connections (with elliptic holonomy) defined in \eqref{eq:standard-connection}; but now also write $a^r = a$ for the connection with (hyperbolic holonomy and) rotation number $r \in \bZ$, chosen arbitrarily, which appears in \eqref{eq:comparison-of-connections}. Additionally, we choose another connection $a^1$ with (hyperbolic holonomy and) rotation number $1$. 
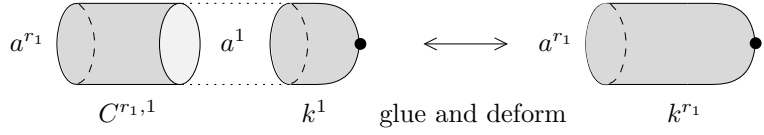
\begin{figure}
\begin{centering}
\begin{picture}(0,0)%
\includegraphics{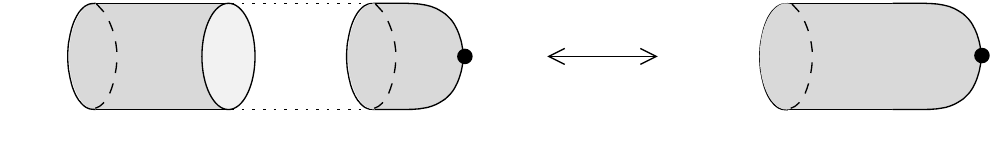}%
\end{picture}%
\setlength{\unitlength}{3355sp}%
\begingroup\makeatletter\ifx\SetFigFont\undefined%
\gdef\SetFigFont#1#2#3#4#5{%
  \reset@font\fontsize{#1}{#2pt}%
  \fontfamily{#3}\fontseries{#4}\fontshape{#5}%
  \selectfont}%
\fi\endgroup%
\begin{picture}(5593,915)(2011,-3214)
\put(4776,-3186){\makebox(0,0)[lb]{\smash{{\SetFigFont{10}{12.0}{\rmdefault}{\mddefault}{\updefault}{glue and deform}%
}}}}
\put(4201,-3186){\makebox(0,0)[lb]{\smash{{\SetFigFont{10}{12.0}{\rmdefault}{\mddefault}{\updefault}{$k^1$}%
}}}}
\put(2701,-3186){\makebox(0,0)[lb]{\smash{{\SetFigFont{10}{12.0}{\rmdefault}{\mddefault}{\updefault}{$C^{r_1,1}$}%
}}}}
\put(2046,-2661){\makebox(0,0)[lb]{\smash{{\SetFigFont{10}{12.0}{\rmdefault}{\mddefault}{\updefault}{$a^{r_1}$}%
}}}}
\put(3601,-2661){\makebox(0,0)[lb]{\smash{{\SetFigFont{10}{12.0}{\rmdefault}{\mddefault}{\updefault}{$a^1$}%
}}}}
\put(6901,-3186){\makebox(0,0)[lb]{\smash{{\SetFigFont{10}{12.0}{\rmdefault}{\mddefault}{\updefault}{$k^{r_1}$}%
}}}}
\put(5951,-2661){\makebox(0,0)[lb]{\smash{{\SetFigFont{10}{12.0}{\rmdefault}{\mddefault}{\updefault}{$a^{r_1}$}%
}}}}
\end{picture}%
\caption{\label{fig:relate-k}The family of surfaces underlying \eqref{eq:kkk2}.}
\end{centering}
\end{figure}

As a preliminary step, we need to explain what ``appropriately coordinated choices of bounding cochains'' means. Two version of the Kodaira-Spencer cocycle are relevant:
\begin{align} \label{eq:kkk}
& k^1 \in C^2(a^1) \quad \;\;\text{(hyperbolic case),} \\
\label{eq:kkk1}
& k^{r_1} \in C^2(a^{r_1}) \quad \text{(elliptic case).} 
\end{align}
These are related by $C^{r_1,1}: C^*(a^1) \rightarrow C^*(a^{r_1})$. More precisely, there is a cochain
\begin{equation} \label{eq:kkk2}
\begin{aligned}
& \kappa^{r_1} \in C^1(a^{r_1}), \\
& k^{r_1} = C^{r_1,1}(k^1) + d\kappa^{r_1},
\end{aligned}
\end{equation}
whose geometric construction is shown schematically in Figure \ref{fig:relate-k}. It relies on Proposition \ref{th:elliptic-disc}, which says that if we glue together the surfaces which define $C^{r_1,1}$ and $k^1$, the outcome is a surface which (together with its nonnegatively curved connection) can be deformed to that underlying $k^{r_1}$. We assume that a choice of bounding cochain $\theta^1 \in C^1(a^1)$ for \eqref{eq:kkk} is given, and then take that for \eqref{eq:kkk1} to be
\begin{equation} \label{eq:kkk3}
\theta^{r_1} = C^{r_1,1}(\theta^1) + \kappa^{r_1} \in C^1(a^{r_1}).
\end{equation}

In our enhanced notation, the definition \eqref{eq:elliptic-nabla} of the ``elliptic'' connection reads as follows:
\begin{equation} \label{eq:elliptic-nabla-2}
\nabla_c^{r_0,r_2} x = C^{r_0,r_2} (\qabla^{r_2} x) + \rho_c^{r_0,r_2}(x) - [\theta^{r_1},x]^{r_0,r_1,r_2}_c.
\end{equation}
Our first task is to break up the continuation map in \eqref{eq:elliptic-nabla-2} into two pieces. As a consequence of Proposition \ref{th:grow-elliptic}, the following diagram of continuation maps is homotopy commutative:
\begin{equation} \label{eq:c-is-cc}
\xymatrix{
\ar@/_1pc/[rrrr]_-{C^{r_0,r_2}}
C^*(a^{r_2}) \ar[rr]^-{C^{r,r_2}} && C^*(a^r) \ar[rr]^-{C^{r_0,r}} && C^*(a^{r_0}).
}
\end{equation}
Using Proposition \ref{th:grow-elliptic} in the same way, one can modify $\rho_c^{r_0,r_2}$ so that it satisfies a version of \eqref{eq:rho-homotopy-elliptic} where the continuation map involved is $C^{r_0,r} C^{r,r_2}$ instead of $C^{r_0,r_2}$. Denote this new version by $\rho_{\mathit{broken},c}^{r_0,r_2}$. On the cohomology level, \eqref{eq:elliptic-nabla-2} can be rewritten as
\begin{equation} \label{eq:elliptic-nabla-3}
\left[\nabla_{c}^{r_0,r_2} x\right] = \left[C^{r_0,r} C^{r,r_2} \qabla^{r_2} x + \rho_{\mathit{broken},c}^{r_0,r_2}(x) - [\theta^{r_1},x]_{c}^{r_0,r_1,r_2}\right] \in H^*(a^{r_0}).
\end{equation}
More precisely, let $\gamma^{r_0,r_2}$ be the chain homotopy which appears in \eqref{eq:c-is-cc}. The relation between $\rho_c^{r_0,r_2}$ and $\rho_{\mathit{broken},c}^{r_0,r_2}$ can be encoded schematically as follows:
\begin{equation} \label{eq:magic-triangle}
\xymatrix{
C^{r_0,r_2} R^{r_2} \ar@{-}[dd]_-{\gamma^{r_0,r_2} R^{r_2}} \ar@{-}[drr]^-{\rho_c^{r_0,r_2}} \\
&& [k^{r_1},\cdot]_c^{r_0,r_1,r_2} \\
C^{r_0,r} C^{r,r_2} R^{r_2} \ar@{-}[urr]_-{\rho_{\mathit{broken},c}^{r_0,r_2}}
}
\end{equation}
Each side of \eqref{eq:magic-triangle} is a homotopy, induced by a family of decorated disc configuration whose parameter space is an annulus $[0,1] \times S^1$ (and over $\{0,1\} \times S^1$, the families can be identified in pairs). There is a cobordism between the union of these three annuli and a family over $T^2$ (essentially, just rounding off the corners where we have identified the boundary circles of the annuli). One then uses Proposition \ref{th:grow-elliptic} to show that that family can be extended over a solid torus, which provides the (degree $-1$) operation that ``fills in the interior of the triangle'' in \eqref{eq:magic-triangle}. The cochain relating the two sides of \eqref{eq:elliptic-nabla} is given by that operation together with $\gamma^{r_0,r_2} \qabla^{r_2}$. We denote the (chain level) right hand side of \eqref{eq:elliptic-nabla-3} by $\nabla_{\mathit{broken},c}^{r_0,r_2}x$.

Take the decorated disc configuration underlying $C^{r,r_2}$, and ``add a marked point in all possible places''. The result defines an operation $C^{r,r_2}_\bullet$, and using \eqref{eq:differentiation-2}, we can rewrite \eqref{eq:elliptic-nabla-3} as
\begin{equation} \label{eq:elliptic-nabla-4}
\nabla_{\mathit{broken},c}^{r_0,r_2}x = C^{r_0,r} \qabla^r C^{r,r_2}(x) + C^{r_0,r} C^{r,r_2}_\bullet(x) + \rho_{\mathit{broken},c}^{r_0,r_2}(x) - [\theta^{r_1},x]_{c}^{r_0,r_1,r_2}.
\end{equation}
Additionally using \eqref{eq:kkk3}, we get
\begin{equation} \label{eq:elliptic-nabla-5}
\begin{aligned}
\nabla_{\mathit{broken},c}^{r_0,r_2}x = &\; C^{r_0,r} \qabla^r C^{r,r_2}(x) + C^{r_0,r} C^{r,r_2}_\bullet(x) + \rho_{\mathit{broken},c}^{r_0,r_2}(x)\\ & \quad - [\kappa^{r_1},x]_c^{r_0,r_1,r_2} - [C^{r_1,1}(\theta^1),x]_{c}^{r_0,r_1,r_2}.
\end{aligned}
\end{equation}
We now want to apply a further chain homotopy, which relates two different versions of the Lie bracket:
\begin{equation}
\begin{aligned}
& \beta^{r_0,1,r_2}_c: C^*(a^1) \otimes C^*(a^{r_2}) \longrightarrow C^{*-2}(a^{r_0}), \\
& [C^{r_1,1}(x_1),x_2]_c^{r_0,r_1,r_2} - C^{r_0,r} [x_1,C^{r_2,r}(x_2)]^{r,1,r} 
\\ & \quad =
d\beta_c^{r_0,1,r_2}(x_1,x_2) + \beta_c^{r_0,1,r_2}(dx_1,x_2) + (-1)^{|x_1|} \beta_c^{r_0,1,r_2}(x_1,dx_2).
\end{aligned}
\end{equation}
This follows by comparing Lemma \ref{th:rot-bracket} with the definition of \eqref{eq:twisted-bracket}, where the equality $c = r-1$ is key, and then using Proposition \ref{th:ehe-pants} to construct the family of connections that interpolates between the two sides. From \eqref{eq:elliptic-nabla-5} we can therefore pass to (the chain homotopic expression)
\begin{equation} \label{eq:elliptic-nabla-6}
\begin{aligned}
\nabla^{\mathit{broken},c}_{r_0,r_2}x \htp & \; C^{r_0,r} \qabla^r C^{r,r_2}(x) + C^{r_0,r} C^{r,r_2}_\bullet(x) + \rho_{\mathit{broken},c}^{r_0,r_2}(x) \\ & \quad - [\kappa^{r_1},x]_c^{r_0,r_1,r_2} - C^{r_0,r} [\theta^1,C^{r,r_2}(x)]^{r,1,r} - \beta_c^{r_0,1,r_2}(k^1,x).
\end{aligned}
\end{equation}

Comparing \eqref{eq:elliptic-nabla-6} with \eqref{eq:nabla} shows that the desired statement concerning \eqref{eq:comparison-of-connections} reduces to the following nullhomotopy of maps $C^*(a^{r_2}) \rightarrow C^*(a^{r_0})$:
\begin{equation} \label{eq:null-last}
C^{r_0,r} C^{r,r_2}_\bullet(x) + \rho_{\mathit{broken},c}^{r_0,r_2}(x) - C^{r_0,r} \rho^{r,r} C^{r,r_2}(x) 
- [\kappa^{r_1},x]_c^{r_0,r_1,r_2} - \beta^{r_0,1,r_2}_c(k^1,x) \htp 0.
\end{equation}
Importantly, this no longer contains $\qabla$ or the bounding cochain $\theta^1$: it is purely a relation between TQFT operations. Each term in \eqref{eq:null-last} comes from a family parametrized by an annulus, and for the nullhomotopy, the same reasoning as in \eqref{eq:elliptic-nabla-3} applies. The analogue of \eqref{eq:magic-triangle} is this:
\begin{equation} \label{eq:magic-pentagon}
\xymatrix{
& \ar@{-}[dl]_-{C^{r_0,r} C^{r,r_2}_\bullet} 
C^{r_0,r} C^{r,r_2} R^{r_2} 
\ar@{-}[dr]^-{\rho^{r_0,r_2}_{\mathit{broken},c}} & \\
\ar@{-}[d]_-{C^{r_0,r}\rho^{r,r}C^{r,r_2}}
C^{r_0,r} R^r C^{r,r_2} && 
[k^{r_1},\cdot]^{r_0,r_1,r_2}_c 
\ar@{-}[d]^-{[\kappa^{r_1},\cdot]^{r_0,r_1,r_2}_c} \\
  C^{r_0,r} [k^1,C^{r,r_2}(\cdot)]^{r,1,r} \ar@{-}@/_2pc/[rr]_-{\beta^{r_0,1,r_2}_c(k^1,\cdot)}
&& [C^{r_1,1}(k^1), \cdot]_c^{r_0,r_1,r_2} 
}
\end{equation}
In order to better visualize the situation, we also include a version of \eqref{eq:magic-pentagon} where the operations are described geometrically, in Figure \ref{fig:pentagon}.
\end{proof}
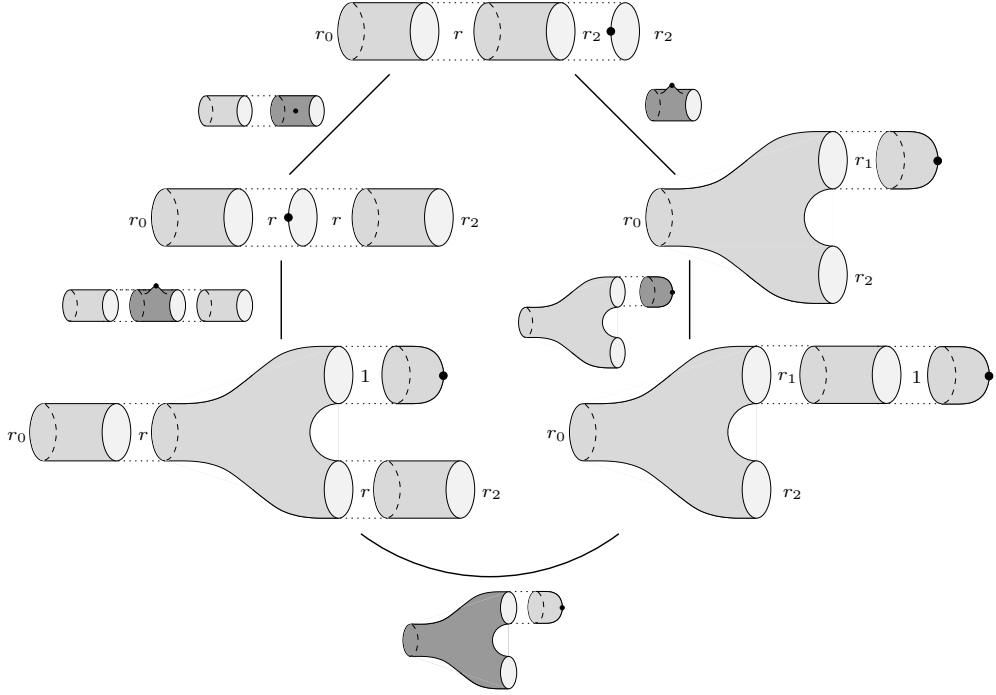
\begin{figure}
\begin{centering}
\begin{picture}(0,0)%
\includegraphics{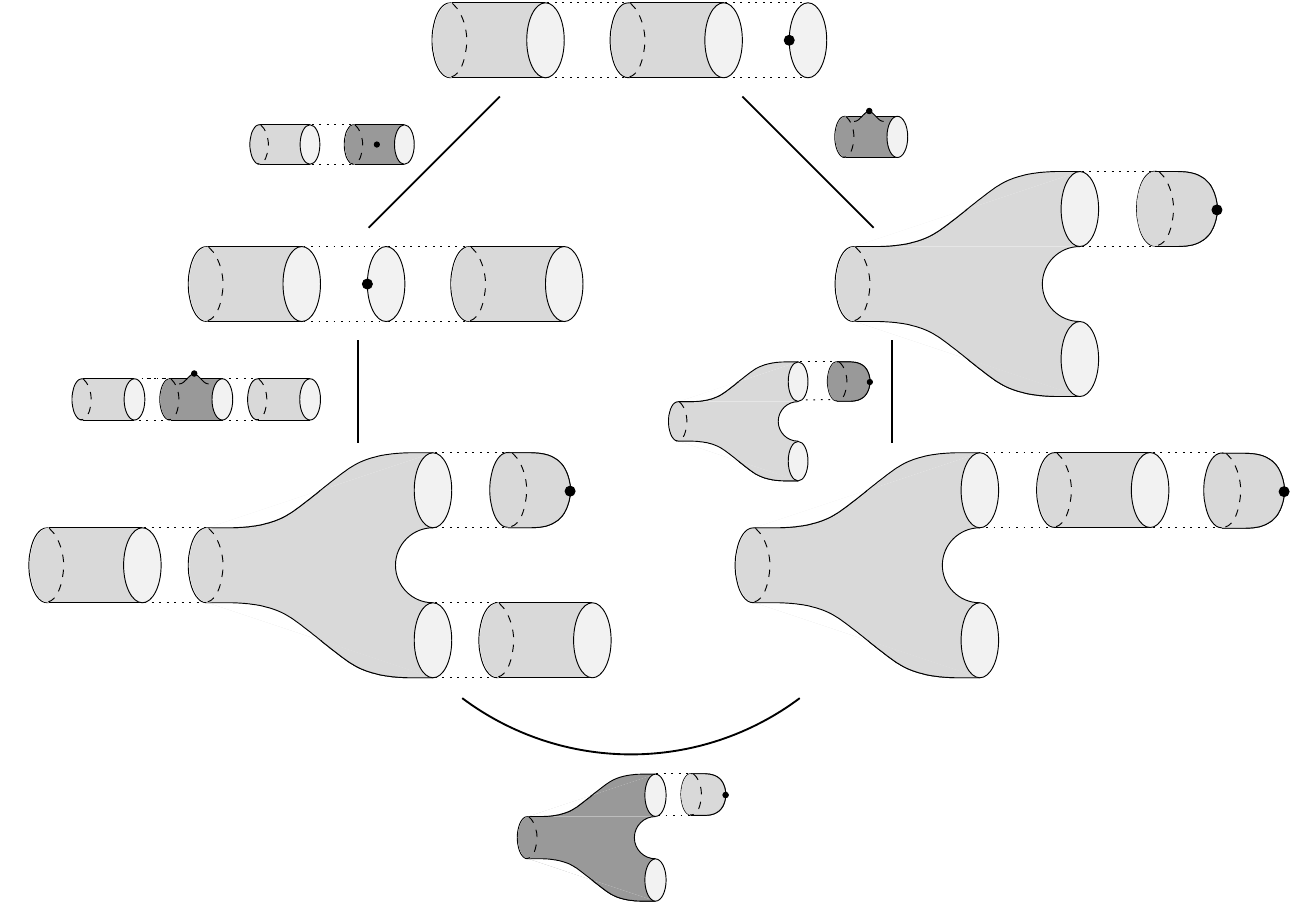}%
\end{picture}%
\setlength{\unitlength}{2368sp}%
\begingroup\makeatletter\ifx\SetFigFont\undefined%
\gdef\SetFigFont#1#2#3#4#5{%
  \reset@font\fontsize{#1}{#2pt}%
  \fontfamily{#3}\fontseries{#4}\fontshape{#5}%
  \selectfont}%
\fi\endgroup%
\begin{picture}(10322,7211)(-1964,-5760)
\put(1251,1064){\makebox(0,0)[lb]{\smash{{\SetFigFont{7}{8.4}{\rmdefault}{\mddefault}{\updefault}{$r_0$}%
}}}}
\put(4051,1064){\makebox(0,0)[lb]{\smash{{\SetFigFont{7}{8.4}{\rmdefault}{\mddefault}{\updefault}{$r_2$}%
}}}}
\put(2701,1064){\makebox(0,0)[lb]{\smash{{\SetFigFont{7}{8.4}{\rmdefault}{\mddefault}{\updefault}{$r$}%
}}}}
\put(1726,-2536){\makebox(0,0)[lb]{\smash{{\SetFigFont{7}{8.4}{\rmdefault}{\mddefault}{\updefault}{$1$}%
}}}}
\put(1726,-3736){\makebox(0,0)[lb]{\smash{{\SetFigFont{7}{8.4}{\rmdefault}{\mddefault}{\updefault}{$r$}%
}}}}
\put(-599,-3136){\makebox(0,0)[lb]{\smash{{\SetFigFont{7}{8.4}{\rmdefault}{\mddefault}{\updefault}{$r$}%
}}}}
\put(-1969,-3136){\makebox(0,0)[lb]{\smash{{\SetFigFont{7}{8.4}{\rmdefault}{\mddefault}{\updefault}{$r_0$}%
}}}}
\put(6151,-3736){\makebox(0,0)[lb]{\smash{{\SetFigFont{7}{8.4}{\rmdefault}{\mddefault}{\updefault}{$r_2$}%
}}}}
\put(3001,-3736){\makebox(0,0)[lb]{\smash{{\SetFigFont{7}{8.4}{\rmdefault}{\mddefault}{\updefault}{$r_2$}%
}}}}
\put(6901,-1486){\makebox(0,0)[lb]{\smash{{\SetFigFont{7}{8.4}{\rmdefault}{\mddefault}{\updefault}{$r_2$}%
}}}}
\put(751,-886){\makebox(0,0)[lb]{\smash{{\SetFigFont{7}{8.4}{\rmdefault}{\mddefault}{\updefault}{$r$}%
}}}}
\put(1426,-886){\makebox(0,0)[lb]{\smash{{\SetFigFont{7}{8.4}{\rmdefault}{\mddefault}{\updefault}{$r$}%
}}}}
\put(2776,-886){\makebox(0,0)[lb]{\smash{{\SetFigFont{7}{8.4}{\rmdefault}{\mddefault}{\updefault}{$r_2$}%
}}}}
\put(4801,1064){\makebox(0,0)[lb]{\smash{{\SetFigFont{7}{8.4}{\rmdefault}{\mddefault}{\updefault}{$r_2$}%
}}}}
\put(4476,-886){\makebox(0,0)[lb]{\smash{{\SetFigFont{7}{8.4}{\rmdefault}{\mddefault}{\updefault}{$r_0$}%
}}}}
\put(3676,-3121){\makebox(0,0)[lb]{\smash{{\SetFigFont{7}{8.4}{\rmdefault}{\mddefault}{\updefault}{$r_0$}%
}}}}
\put(6901,-261){\makebox(0,0)[lb]{\smash{{\SetFigFont{7}{8.4}{\rmdefault}{\mddefault}{\updefault}{$r_1$}%
}}}}
\put(6096,-2517){\makebox(0,0)[lb]{\smash{{\SetFigFont{7}{8.4}{\rmdefault}{\mddefault}{\updefault}{$r_1$}%
}}}}
\put(7497,-2540){\makebox(0,0)[lb]{\smash{{\SetFigFont{7}{8.4}{\rmdefault}{\mddefault}{\updefault}{$1$}%
}}}}
\put(-699,-886){\makebox(0,0)[lb]{\smash{{\SetFigFont{7}{8.4}{\rmdefault}{\mddefault}{\updefault}{$r_0$}%
}}}}
\end{picture}%
\caption{\label{fig:pentagon}The families of surfaces appearing in \eqref{eq:magic-pentagon}. The darker shaded pieces are those one which change along the sides of the diagram.}
\end{centering}
\end{figure}%

\section{Maps to the disc\label{sec:maps-to-the-disc}}
As a toy model for the subsequent construction of Hamiltonian Floer cohomology, we look at maps from Riemann surfaces to the hyperbolic disc, which satisfy an inhomogeneous Cauchy-Riemann equation. The inhomogeneous term comes from a family of infinitesimal automorphisms, hence is not compactly supported on the disc. We are mostly interested in what happens when the maps approach the boundary circle (at infinity).

\subsection{The hyperbolic disc}
Let $B \subset \bC$ be the open unit disc. The holomorphic automorphism group of $B$, and its Lie algebra of holomorphic vector fields, are given by:
\begin{equation} \label{eq:x-gamma}
\begin{aligned}
& G = \mathit{PU}(1,1) = \left\{ \textstyle g = \left(\begin{smallmatrix} a & b \\ \bar{b} & \bar{a} \end{smallmatrix}\right)\; : \;
|a|^2 - |b|^2 = 1 \right\}/\pm\!\Id,
&& 
g \mapsto \rho_g(w) = \frac{aw+b}{\bar{b}w + \bar{a}},
\\
& \frakg = \mathfrak{u}(1,1) = \left\{ \gamma = \left(\begin{smallmatrix} i\alpha & \beta \\ \bar{\beta} & -i\alpha \end{smallmatrix}\right) \;:\;
\alpha \in \bR, \;\beta \in \bC \right\}, &&
\gamma \mapsto X_\gamma = (\bar{\beta} + 2i\alpha w - \beta w^2) \, \partial_w.
\end{aligned}
\end{equation}
The action preserves the hyperbolic metric (in our convention, normalized to have curvature $-4$)
\begin{equation}
g_{\mathit{hyp}} = \frac{1}{(1-|w|^2)^2}  \mathit{dw}\, d\bar{w},
\end{equation}
hence also its symplectic form $\omega_{\mathit{hyp}}$, and distance function $\mathrm{dist}_{\mathit{hyp}}$. There is a unique choice of Hamiltonian function for each vector field $X_\gamma$ which is linear in $\gamma$ and compatible with Poisson brackets:
\begin{equation} \label{eq:hamiltonian}
H_\gamma 
= \frac{1}{1-|w|^2} \left( \half(1+|w|^2)\alpha - \mathrm{im}(\beta w) \right).
\end{equation}

The group action extends smoothly to the closed disc $\bar{B}$. Call $\gamma$ nonnegative if $X_\gamma$ points in nonnegative direction along $\partial \bar{B}$.  In terms of \eqref{eq:x-gamma}, the nonnegative cone defined in this way is
\begin{equation} \label{eq:h-pos}
\frakg_{\geq 0} = \{ \alpha \geq |\beta| \}.
\end{equation}
In terms of \eqref{eq:hamiltonian}, $\gamma$ is nonnegative iff $H_\gamma$ is nonnegative everywhere (moreover, if $\gamma \notin \frakg_{\geq 0}$, $H_\gamma$ is not even bounded below). 

\begin{convention}
As the reader may have noticed, we are switching notation: for the rest of the paper, $G$ and $\frakg$ will be as in \eqref{eq:x-gamma}. This is not a significant issue, since the two groups that appear under that name are actually isomorphic:
\begin{equation} \label{eq:change-g}
\mathit{PSL}_2(\bR) \iso \mathit{PU}(1,1),
\end{equation}
and such an isomorphism is canonical up to inner automorphisms (geometrically, it is induced by a holomorphic automorphism of $\bC \cup \{\infty\} = \bC P^1$ which carries $B$ to the open upper half plane). \eqref{eq:change-g} preserves traces; and the associated isomorphism of Lie algebras relates \eqref{eq:h-pos} to our previous notion of nonnegativity in $\mathfrak{sl}_2(\bR)$. Hence, the entire discussion from Sections \ref{sec:sl2}--\ref{sec:elliptic} carries over immediately to \eqref{eq:x-gamma}.
\end{convention}

Let $S$ be an oriented surface, equipped with a connection $A \in \Omega^1(S,\frakg)$. Via the homomorphisms from $\frakg$ to the Lie algebras of vector fields and of functions (with the Poisson bracket) on $B$, this yields one-forms valued in those Lie algebras. We write them as
\begin{equation}
\begin{aligned} \label{eq:ha}
& X_A \in \smooth(S \times B, T^*S \otimes_{\bR} \bC), \\
& H_A \in \smooth(S \times B, T^*S).
\end{aligned}
\end{equation}
One can apply the same process to the curvature $F_A \in \Omega^2(S,\frakg)$, which yields
\begin{equation} \label{eq:2-curvatures}
\begin{aligned}
& X_{F_A} \in \smooth(S \times B, \Lambda^2_{\bR} T^*S \otimes_{\bR} \bC), \\
& H_{F_A} \in \smooth(S \times B, \Lambda^2_{\bR} T^*S).
\end{aligned}
\end{equation}
What $X_A$ describes is the connection induced by $A$ on the trivial fibre bundle $S \times B \rightarrow S$. Like any connection on a fibre bundle, one can also think of this as given by a complementary subbundle to the fibrewise tangent bundle. The choice of such a subbundle determines a lift of $\omega_{\mathit{hyp}}$ to a fibrewise symplectic form $\omega_{A,\mathit{geom}} \in \Omega^2(S \times B)$, and this satisfies
\begin{equation} \label{eq:domega-a}
d\omega_{A,\mathit{geom}} = -dH_{F_A}.
\end{equation}
On the right hand side of \eqref{eq:domega-a}, we think of $H_{F_A}$ as a two-form on $S \times B$ via the projection $T(S \times B) \rightarrow TS$, and then take its exterior derivative. While $\omega_{A,\mathit{geom}}$ is not itself closed, there is an obvious closed replacement (this is an instance of the general theory of symplectic fibrations, see e.g.\ \cite[Chapter 6]{mcduff-salamon99}):
\begin{equation} \label{eq:top-omega}
\omega_{A,\mathit{top}} = \omega_{A,\mathit{geom}} + H_{F_A}.
\end{equation}
Similarly, the connection determines an almost complex structure $J_A$ on $S \times B$, such that projection to $S$ is $J_A$-holomorphic, and whose restriction to each fibre is the standard complex structure on $B$. This satisfies
\begin{equation} \label{eq:semi-omega}
\omega_{A,\mathit{geom}}(\eta ,J_A \eta) \geq 0 \quad \text{for all $\eta \in T(S \times B)$.}
\end{equation}
By construction, the induced connection on $S \times B \rightarrow S$ extends smoothly to the compactification $S \times \bar{B}$, and so does $J_A$.

One can be more explicit by working in a local complex coordinate $z = s+it$ on $S$. If $A = A_1(s,t)\mathit{ds} + A_2(s,t)\mathit{dt}$ with $A_1(s,t),\, A_2(s,t) \in \frakg$, then by definition
\begin{align}
& X_A = X_{A_1} \mathit{ds} + X_{A_2} \mathit{dt}, \\
& H_A = H_{A_1} \mathit{ds} + H_{A_2} \mathit{dt}, \\
& X_{F_A} = \big( X_{\partial_t A_1} - X_{\partial_s A_2} + X_{[A_1,A_2]} \big) \, \mathit{ds} \wedge \mathit{dt}, \\ 
& H_{F_A} = \big( H_{\partial_t A_1} - H_{\partial_s A_2} + \omega_{\mathit{hyp}}(X_{A_1},X_{A_2}) \big) \,\mathit{ds} \wedge \mathit{dt}.
\end{align}
The subbundle of $T(S \times B)$ which defines the connection on $S \times B \rightarrow S$ is
\begin{equation} \label{eq:horizontal-subspace}
\bR (\partial_s + X_{A_1} \partial_w) \oplus \bR (\partial_t + X_{A_2} \partial_w).
\end{equation}
The two-form $\omega_{A,\mathit{geom}}$ is characterized by having \eqref{eq:horizontal-subspace} as its nullspace. This means that
\begin{align}
& \omega_{A,\mathit{geom}} = \omega_{\mathit{hyp}} - dH_{A_1} \wedge \mathit{ds} - dH_{A_2} \wedge \mathit{dt} - \omega_{\mathit{hyp}}(X_{A_1},X_{A_2}) \mathit{ds} \wedge \mathit{dt}, \\
& \omega_{A,\mathit{top}} = \omega_{\mathit{hyp}} - d(H_{A_1} \mathit{ds}) - d(H_{A_2} \mathit{dt}).
\end{align}
The associated almost complex structure is similarly given by
\begin{equation} \label{eq:j-a}
\left\{
\begin{aligned}
& J_A (\partial_w) = i\partial_w, \\
& J_A (\partial_s + X_{A_1} \partial_w) = \partial_t + X_{A_2} \partial_w.
\end{aligned}
\right.
\end{equation}

Suppose that $A$ is nonnegatively curved. This means that, for any oriented basis $(\xi_1,\xi_2)$ in $TS$, $H_{F_A}(\xi_1,\xi_2)$ is a nonnegative function; and that along $\partial\bar{B}$, $X_{F_A}(\xi_1,\xi_2)$ points in clockwise direction. From \eqref{eq:top-omega} and \eqref{eq:semi-omega}, it then follows that
\begin{equation} \label{eq:o-o}
0 \leq \omega_{A,\mathit{geom}}(\eta,J_A \eta) \leq \omega_{A,\mathit{top}}(\eta, J_A \eta).
\end{equation}

\begin{lemma} \label{th:levi}
If $A$ has nonnegative curvature, the hypersurface $S \times \partial \bar{B} \subset S \times \bar{B}$ is Levi convex (meaning that its Levi form with respect to $J_A$ is nonnegative).
\end{lemma}

\begin{proof}
Take $\psi(z,w) = \half |w|^2: S \times \bar{B} \rightarrow \bR$. In local coordinates on $S$ as before, one gets
\begin{equation} \label{eq:ddc}
\begin{aligned}
-d(d\psi \circ J_A)| (S \times \partial \bar{B}) = &\;
- \mathrm{im}(X_{A_1} d\bar{w}) \wedge \mathit{ds} 
-\mathrm{im}(X_{A_2} d\bar{w}) \wedge \mathit{dt} 
\\ & 
+ \mathrm{im}(\bar{w} X_{\partial_t A_1}) \mathit{ds} \wedge \mathit{dt} 
- \mathrm{im}(\bar{w} X_{\partial_s A_2}) \mathit{ds} \wedge \mathit{dt} 
\\ &
- \mathrm{im}(\bar{w}\, dX_{A_1}) \wedge \mathit{ds} 
- \mathrm{im}(\bar{w}\, dX_{A_2}) \wedge \mathit{dt}.
\end{aligned}
\end{equation}
The Levi form $L_\psi$ is obtained by restricting \eqref{eq:ddc} to the space $T(S \times \partial \bar{B}) \cap J_A T(S \times \partial \bar{B})$, which is \eqref{eq:horizontal-subspace}. From \eqref{eq:ddc}, one gets
\begin{equation} \label{eq:levi}
L_\psi = \mathrm{im}( \bar{w} X_{F_A} ).
\end{equation}
Let's spell out the meaning of this. $X_{F_A}$, defined in \eqref{eq:2-curvatures}, is a complex-valued skew bilinear form on $TS$. In \eqref{eq:levi}, we identify that space with \eqref{eq:horizontal-subspace} by pullback. The desired property follows immediately from \eqref{eq:levi}.
\end{proof}

\subsection{The Cauchy-Riemann equation}
Now assume that $S$ is a (connected) Riemann surface. We will consider the inhomogeneous Cauchy-Riemann equation
\begin{equation} \label{eq:inhomogeneous}
\left\{
\begin{aligned} 
& u: S \longrightarrow B, \\
& \big( Du - X_A(u) \big)^{0,1} = 0.
\end{aligned}
\right.
\end{equation}
More precisely, $X_A$ is being evaluated at points $\sigma(z) = (z,u(z))$, and the $(0,1)$ part is taken with respect to the given complex structures on the source and target. Equivalently, this is the equation for $\sigma$ to be a pseudo-holomorphic section:
\begin{equation} \label{eq:homogeneous}
\left\{
\begin{aligned}
& \sigma = (\mathit{id},u): S \longrightarrow S \times B, \\
& D\sigma \circ i = J_A(\sigma) \circ d\sigma.
\end{aligned}
\right.
\end{equation}
In local coordinates on $S$, \eqref{eq:inhomogeneous} becomes
\begin{equation} \label{eq:concrete-inho}
\partial_s u + i \partial_t u = X_{A_1(s,t)}(u) + i X_{A_2(s,t)}(u).
\end{equation}
As one would expect from the appearance of connections, there is an action of the group of gauge transformations. Namely, if $A^\dag = \Phi_*A$ and $u^\dag(z) = \rho_{\Phi(z)}(u(z))$ for some $\Phi \in \smooth(S,G)$, then
\begin{equation}
D u^\dag - X_{A^\dag}(u^\dag) = D\rho_{\Phi} (Du - X_A(u)),
\end{equation}
where on the right hand side, $D\rho_{\Phi}$ is the derivative of $\rho_{\Phi(z)} \in \mathit{Aut}(B)$ at the point $u(z)$. Hence, passing from $(u,A)$ to $(u^\dag,A^\dag)$ transforms solutions of one equation \eqref{eq:inhomogeneous} to the other. Equivalently, the gauge transformation determines a diffeomorphism of $S \times B$ which is fibrewise a hyperbolic isometry, and that diffeomorphism relates the almost complex structures $J_A$ and $J_{A^\dag}$.

We also want to quickly review the notions of energy for solutions of \eqref{eq:inhomogeneous}. The geometric and topological energies are, respectively,
\begin{align} \label{eq:geometric-energy}
& E_{\mathit{geom}}(u) = \half \int_S |du - X_A(u)|^2_{\mathit{hyp}}, 
\\
& \label{eq:topological-energy}
E_{\mathit{top}}(u) = \int_S u^*\omega_{\mathit{hyp}} - d (H_A(u))
\end{align}
(for general maps $u$ and noncompact $S$, the integral \eqref{eq:topological-energy} may not converge, but we'll ignore that issue as it will irrelevant for our applications). The difference between the two notions is given by a curvature term:
\begin{equation} \label{eq:e-e}
E_{\mathit{top}}(u) = E_{\mathit{geom}}(u) + \int_S H_{F_A}(u).
\end{equation}
In particular, if $A$ is nonnegatively curved, 
\begin{equation}
E_{\mathit{geom}}(u) \leq E_{\mathit{top}}(u).
\end{equation}
%
In other words, the situation is as follows: the two notions of energy are obtained by integrating $\omega_{A,\mathit{geom}}$ and $\omega_{A,\mathit{top}}$ over $\sigma$. Then, \eqref{eq:top-omega} explains \eqref{eq:e-e}, and the behaviour for nonnegatively curved connections follows from \eqref{eq:o-o}.

\subsection{Convergence and compactness\label{subsec:schwarz}}
If the connection $A$ is flat, one can use a local gauge transformation to reduce \eqref{eq:inhomogeneous} to $\bar\partial u^\dag = 0$, hence to elementary complex analysis.

\begin{lemma} \label{th:schwarz}
Suppose that $S = B_r \subset \bC$ is the open disc of radius $r$, and that the connection $A$ is flat. Then for any solution of \eqref{eq:inhomogeneous}, we have
\begin{equation}
|du(z) - X_A(z,u(z))|_{\mathit{hyp}} \leq \frac{r}{r^2-|z|^2}.
\end{equation}
\end{lemma}

After a gauge transformation, this is just the Schwarz Lemma (holomorphic maps between hyperbolic Riemann surfaces can't increase distances).

\begin{lemma} \label{th:hyp-cylinder}
Suppose that $S = (-l/2,l/2) \times S^1$ is the open cylinder of length $l$, and that the connection $A$ is flat, with hyperbolic holonomy $g \in G$ around $\{0\} \times S^1$. Write $\mathrm{tr}(g) = \pm(\lambda+\lambda^{-1})$, for some $\lambda>1$. Then, a solution of \eqref{eq:inhomogeneous} can exist only if
\begin{equation} \label{eq:l-ineq}
l \leq \frac{\pi}{\log(\lambda^2)}.
\end{equation}
\end{lemma}

\begin{proof}
We again apply a gauge transformation to trivialize the connection, but this time working on the universal cover $(-l/2,l/2) \times \bR$. After doing that and then again quotienting out by deck transformations, a solution \eqref{eq:inhomogeneous} turns into a holomorphic map
\begin{equation}
S \longrightarrow B/\rho_g \iso (0,\pi/\log(\lambda^2)) \times S^1,
\end{equation}
which induces an isomorphism of fundamental groups. Such a map can't exist if the source cylinder is longer than the target one; this follows from the Schwarz Lemma, given that the hyperbolic metric on a cylinder of length $l$ has a ``waist'' of girth $\pi/2l$; see e.g.\ \cite[p.~13]{kobayashi05} (taking into account our scaling convention for the metric).
%
\end{proof}

\begin{lemma} \label{th:elliptic-cylinder}
Suppose that $S = (-l/2,l/2) \times S^1$ as before, and that the connection $A$ is flat, with elliptic holonomy $g \in G$ around $\{0\} \times S^1$. Write $\mathrm{tr}(g) = \pm 2\cos(\theta)$, for some $\theta \in (0,\pi)$. Let $w \in B$ be the unique fixed point of $\rho_g$. Then, any solution of \eqref{eq:inhomogeneous} satisfies
\begin{equation}
\sinh(2\,\mathrm{dist}_{\mathit{hyp}}(u(0,0),w)) \leq \frac{\sinh(\pi/2l)}{\sin \theta}.
\end{equation}
\end{lemma}

\begin{proof}
After the usual gauge transformation, we get a holomorphic map $u^\dag: (-l/2,l/2) \times \bR \rightarrow B$, which satisfies $u^\dag(0,0) = u(0,0)$ and $u^\dag(s,t+1) = \rho_g(u^\dag(s,t))$. In the hyperbolic metric on the domain, the distance between the points $(0,0)$ and $(0,1)$ is $\pi/2l$, as in Lemma \ref{th:hyp-cylinder}. By hyperbolic trigonometry, the distance between $p = u^\dag(0,0)$ and $q = u^\dag(0,1)$ satisfies
\begin{equation}
\sin \theta = \frac{\sinh( \mathrm{dist}_{\mathit{hyp}}(p,q))}{\sinh(2\,\mathrm{dist}_{\mathit{hyp}}(p,w))}.
\end{equation}
The rest is once again Schwarz' Lemma.
\end{proof}

\begin{lemma} \label{th:flat-convergence}
Let $S$ be any Riemann surface, equipped with a flat connection $A$, and $(u_k)$ a sequence of solutions of \eqref{eq:inhomogeneous}. Suppose that there is a sequence of points $z_k$, contained in a compact subset of $S$, such that $u_k(z_k)$ goes to $\partial \bar{B}$. Then $u_k$ goes to $\partial \bar{B}$, in the sense of uniform convergence on compact subsets. 
\end{lemma}

\begin{proof}
Since we can always replace $S$ by its universal cover, let's assume that it is simply-connected. After a gauge transformation, we get a sequence of bounded holomorphic functions $u_k^\dag$. This sequence has a convergent (on compact subsets) subsequence, and hence, so does $u_k$. Let $u_\infty^\dag, u_\infty: S \rightarrow \bar{B}$ be the limits. $u_\infty^\dag$ is a holomorphic function which meets $\partial \bar{B}$, hence must be constant at a point of $\partial \bar{B}$, by the open mapping theorem. It follows that $u_\infty$ also take values in $\partial \bar{B}$. By same argument, every subsequence of $(u_k)$ has a subsequence with the desired behaviour. From this, the result follows.
\end{proof}


To go beyond the flat case, we need to replace holomorphic function theory by more flexible analytic tools. We work with a fixed $S$ and connection $A$. Bounds for derivatives will be with respect to some metric on $S$ (the choice is irrelevant, since we will always work on compact subsets), and the Euclidean metric on $B$.

\begin{lemma} \label{th:elliptic}
Given a compact subset $K \subset S$ and an $r \geq 0$, there is a constant $C$ such that $\|u|K\|_{C^r} \leq C$ for all solutions $u$ of \eqref{eq:inhomogeneous}.
\end{lemma}

\begin{proof}[Sketch of proof]
Let's work in local coordinates on $S$. There is a uniform bound on the right hand side of \eqref{eq:concrete-inho}, and of course on $u$ itself. By elliptic regularity, we get a $W^{1,p}$ bound on $u$ (for any $1<p<\infty$). From there, one proceeds by bootstrapping.
\end{proof}

One can also view Lemma \ref{th:elliptic} as a consequence of the regularity theory for pseudo-holomorphic maps (see e.g.\ \cite[Appendix B]{mcduff-salamon-big}) to \eqref{eq:homogeneous}; the noncompactness of $S \times B$ does not matter here, since $J_A$ extends to $S \times \bar{B}$.
%

\begin{lemma} \label{th:levi-apply}
Suppose that $A$ has nonnegative curvature. Let $(u_k)$ be a sequence of solutions of \eqref{eq:inhomogeneous}. Suppose that $z_k$ is a sequence of points, contained in a compact subset of $S$, such that $u_k(z_k)$ goes to $\partial \bar{B}$. Then $u_k$ goes to $\partial \bar{B}$, in the sense of uniform convergence on compact subsets.
\end{lemma}

\begin{proof}
Lemma \ref{th:elliptic} shows that, after passing to a subsequence, the $u_k$ converge to some $u_\infty: S \rightarrow \bar{B}$. Consider the $J_A$-holomorphic map $\sigma_\infty(z) = (z,u_\infty(z))$. We know from Lemma \ref{th:levi} that $S \times \partial \bar{B}$ is Levi convex, which by \cite[Corollary 4.7]{diederich-sukhov08} implies that $\sigma_\infty^{-1}(S \times \partial \bar{B})$ is open and closed. But by assumption, that subset is also nonempty, which shows that $u_\infty$ takes values in $\partial \bar{B}$. The rest is as in Lemma \ref{th:flat-convergence}.
\end{proof}

\begin{remark} \label{th:sukhov}
The expository paper \cite{gaussier-sukhov11} provides helpful additional discussion of the result from \cite{diederich-sukhov08} which we have used. The strategy of proof goes as follows. Let $\bar{N}$ be a domain with smooth Levi convex boundary in an almost complex manifold, and $N$ its interior. Suppose that there is a pseudo-holomorphic map from a connected Riemann surface to $\bar{N}$, which meets the boundary but is not entirely contained in it. After restriction to a suitably chosen part of $S$, we get a pseudo-holomorphic map defined on the unit disc,
\begin{equation}
\left\{
\begin{aligned}
& \sigma: \bar{B} \longrightarrow \bar{N}, \\
& \sigma^{-1}(\partial \bar{N}) = \{1\}, \text{ and $\sigma$ is tangent to $\partial \bar{N}$ at that point.}
\end{aligned}
\right.
\end{equation}
Near $\sigma(1)$, there is a smooth local defining function $r$ for $\partial \bar{N}$, such that $-(-r)^\theta$ is (strictly) plurisubharmonic in $N$; the constant $0 < \theta < 1$ can be chosen arbitrarily \cite[Theorem 2.4]{gaussier-sukhov11}, and we will just need $\theta > 1/2$. After shrinking the domain further if necessary, we may assume that $r \circ \sigma$ is well-defined. Since it has vanishing derivative at $z = 1$, we have $r(\sigma(z)) \gtrsim -|z-1|^2$. Hence, $h = -(-r)^\theta \circ \sigma$ is a negative subharmonic function on $B$, which extends continuously to $\bar{B}$ and vanishes exactly at $1 \in \partial \bar{B}$, such that
\begin{equation}
h(z) \gtrsim -|z-1|^{2\theta}.
\end{equation}
But that contradicts the Hopf Lemma for such functions.
\end{remark}

\section{Floer cohomology\label{sec:floer}}
This section sets up Hamiltonian Floer cohomology theory for Lefschetz fibrations, and the operations on it. The basic approach is classical \cite{hofer-salamon95}. The issues that are specific to our context arise from the noncompactness of the symplectic manifold. Compared to previous versions of the same formalism (see e.g.\ \cite{seidel14b}), the main difference is that at infinity, our Hamiltonian automorphisms are modelled on isometries of the hyperbolic disc (following Section \ref{sec:maps-to-the-disc}), rather than area-preserving diffeomorphisms of the Euclidean plane. Even though the distinction may fundamentally be a technical one, the hyperbolic model turns out to be very helpful for our constructions. 

\subsection{Geometry of the target space}
Let $(E^{2n},\omega_E)$ be a symplectic manifold, together with a trivialization of its anticanonical bundle (for some compatible almost complex structure), and a proper map \eqref{eq:lefschetz} to the hyperbolic disc. Suppose that $\bar{N} \subset \bar{B}$ of $\partial \bar{B}$ is a closed collar neighbourhood. We then call $N = \bar{N} \cap B$ a neighbourhood of infinity (for instance, $\{ |w| \geq r\}$ for some $r<1$ is such a neighbourhood). Our main condition is:
\begin{itemize} \itemsep.5em
\item
For some neighbourhood of infinity $N = N_E$, there is a (necessarily unique) $\omega_E$-orthogonal splitting of the tangent space at any point $x \in \pi^{-1}(N)$ into horizontal and vertical parts,
\begin{equation} \label{eq:te-h-v}
\left\{
\begin{aligned}
& TE_x = TE_x^h \oplus TE_x^v, \\ 
& \mathit{TE}_x^v = \mathit{ker}(D\pi_x), \\
& D\pi_x: \mathit{TE}_x^h \longrightarrow TB_{\pi(x)} \text{ is a symplectic isomorphism.}
\end{aligned}
\right.
\end{equation}
\end{itemize}
In other words: $x$ is a regular point of $\pi$; $\mathit{TE}_x^v$ is a symplectic subspace; and the symplectic form on its $\omega_E$-orthogonal complement agrees with the pullback of $\omega_{\mathit{hyp}}$. As a consequence, \eqref{eq:lefschetz} restricts to a symplectic fibre bundle over $N$ with a preferred connection, given by $\mathit{TE}^h$, which is flat. Using that structure, one can construct a canonical compactification
\begin{equation} \label{eq:compactification}
\pi: \bar{E} \longrightarrow \bar{B}.
\end{equation}
The total space $\bar{E}$ is a compact manifold with boundary, whose interior is $E$, and with $\partial \bar{E} = \pi^{-1}(\partial \bar{B})$. The compactification again has the property that, over $\bar{N}$, it is a symplectic fibre bundle with a preferred flat connection. In other words, the splitting in \eqref{eq:te-h-v} extends to points of $\partial\bar{E}$, and so does the symplectic structure on the vertical part. Of course, $\omega_E$ itself does not extend to $\bar{E}$, because of the $\omega_{\mathit{hyp}}$ component, but there are slightly different symplectic forms that do:
\begin{equation} \label{eq:omega-bar-e}
\omega_{\bar{E}} = \omega_E - \pi^* (\psi \omega_{\mathit{hyp}}) + \pi^* (\psi \omega_{\bar{B}}), 
\end{equation}
where $\omega_{\bar{B}}$ is any (positive) symplectic form on $\bar{B}$, and $\psi: \bar{B} \rightarrow [0,1]$ is a function which vanishes outside $N$, and equals $1$ near $\partial \bar{B}$.
\begin{itemize}
\item
We denote by $\scrJ(E)$ the space of compatible almost complex structures $J$ on $E$ with the following two additional properties. First, there is some neighbourhood of infinity $N = N_J \subset N_E$, such that $\pi$ is $J$-holomorphic at all points of $\pi^{-1}(N)$. Secondly, $J$ extends to $\bar{E}$, and that extension is compatible with \eqref{eq:omega-bar-e}.
\end{itemize}
The last-mentioned condition (compatibility with $\omega_{\bar{E}})$ is independent of the choices made in \eqref{eq:omega-bar-e}. Even before imposing that condition, we already knew that $J$ preserves the extension of the splitting \eqref{eq:te-h-v} to points $x \in \partial \bar{E}$, and is compatible with the orientation of $T\bar{E}^h_x \iso T\bar{B}_{\pi(x)}$. Hence, the additional requirement is that the extension should be compatible with the symplectic structure on $T\bar{E}_x^v$.
We want to make one more observation concerning such almost complex structures. Because the hyperbolic metric scales up to infinity as one approaches $\partial\bar{B}$, there is a constant $C = C_J >0$ such that
\begin{equation} \label{eq:omega-bar-e-2}
\omega_{\bar{E}}(X,JX) \leq C\, \omega_E(X,JX) \quad \text{for all $X \in \mathit{TE}$.}
\end{equation}
In particular, the $\omega_E$-energy of non-constant $J$-holomorphic spheres is bounded below by a positive constant (since the same holds for $\omega_{\bar{E}}$, which lives on a compact space).
\begin{itemize} \itemsep.5em
\item For any $\gamma \in \frakg$, let $\scrH(E,\gamma)$ be the space of functions $H \in \smooth(E,\bR)$ whose restriction to $\pi^{-1}(N)$, for some neighbourhood of infinity $N = N_H \subset N_E$, equals the pullback of the function $H_\gamma$ from \eqref{eq:hamiltonian}. Similarly, given $a = a_t \mathit{dt} \in \Omega^1(S^1,\frakg)$, we write $\scrH(E,a)$ for the space of $H \in \smooth(S^1 \times E, \bR)$ such that $H_t \in \scrH(E,a_t)$ for all $t \in S^1$.
\end{itemize}
On $\pi^{-1}(N)$, the Hamiltonian vector field $X$ associated to $H \in \scrH(E,\gamma)$ is the unique horizontal lift of $X_\gamma$. Hence, $X$ extends smoothly to $\bar{E}$, and is tangent to $\partial\bar{E}$. 

\begin{remark} \label{th:families-of-neighbourhoods}
In our application, we will often encounter almost complex structures or functions which depend on additional parameters (i.e.\ families of such objects). In that case, it is always understood that the relevant neighbourhoods of infinity can be chosen locally constant with respect to the parameter. Equivalently, if $P$ is the parameter space, we are considering a closed collar neighbourhood of $P \times \partial \bar{B}$ inside $P \times \bar{B}$, and then taking its intersection with $P \times B$.
\end{remark}

Given $H = (H_t) \in \scrH(E,a)$ and its time-dependent vector field $X = (X_t)$, we will consider one-periodic orbits
\begin{equation}
\label{eq:periodic-orbit}
\left\{
\begin{aligned}
& x: S^1 \longrightarrow E, \\
& dx/dt = X_t.
\end{aligned}
\right.
\end{equation}

\begin{lemma}
If $a$ has nontrivial holonomy, all solutions of \eqref{eq:periodic-orbit} are contained in a compact subset of $E$. 
\end{lemma}

\begin{proof}
Otherwise, there would have to be a sequence of one-periodic orbits in $E$ which converge to $\partial \bar{E}$. By projection, the same would have to be true for the corresponding ODE on the hyperbolic disc,
\begin{equation}
\label{eq:downstairs-periodic-orbit}
\left\{
\begin{aligned}
& y: S^1 \longrightarrow B, \\
& dy/dt = X_{a_t}.
\end{aligned}
\right.
\end{equation}
But that is impossible, since such $y$ correspond to fixed points of the holonomy of $a$ acting on $B$ (of which there is at most one).
\end{proof}

\begin{remark}
When describing our results in Section \ref{sec:results}, we included a Lefschetz (nondegeneracy of critical points) condition, since that puts them into a familiar context. In fact, that condition is irrelevant for the constructions in this paper, and therefore doesn't appear in the definition we have just given.
\end{remark}

\subsection{Basic transversality and compactness}
In spite of the constraints imposed on the almost complex structures and Hamiltonian functions, the transversality results underlying the traditional construction of Floer theory still hold. Here is a sample of the kind of arguments that are required:

\begin{lemma} \label{th:j-spheres}
For generic $J \in \scrJ(E)$, all simple (non-multiply-covered) $J$-holomorphic spheres are regular.
\end{lemma}

\begin{proof}[Sketch of proof]
A given $J$ comes with some $N = N_J$ over which $D\pi$ is $J$-holomorphic. Genericity is understood to be with respect to perturbations which keep that subset fixed. This means that on $\pi^{-1}(N)$, we may change only the $TE^v$ component of $J$. Let $u: \bC P^1 \rightarrow E$ be a simple $J$-holomorphic sphere, and $v = \pi(u)$. At any point $z$ such that $v(z) \in N$, we have $\bar\partial v = 0$. Hence, one of the following applies:

(i) There is a point where $v(z) \notin N$. Since the choice of almost complex structure is free (except for the condition of $\omega_E$-compatibility) near $u(z)$, transversality for such $u$ follows from the standard argument.

(ii) There is a point where $v(z) \in N \setminus \partial N$. Note that the degree of $v$ over any point $w \in B$ is the same, hence zero. It follows that for each $w \in N \setminus \partial N$, $v^{-1}(w)$ is open and closed. Under our assumption, this means that $u$ must lie in a single fibre $F = E_w$. A parametrized transversality argument, using the fact that the almost complex structure in vertical direction can be varied freely, shows generic regularity for such maps. To explain the relevance of parametrized transversality, consider the commutative diagram (with exact columns)
\begin{equation} \label{eq:snake}
\xymatrix{
0 \ar[d] && 0 \ar[d] \\
\ar[d]
\{ \xi \in \smooth(\bC P^1, u^*TE) \,:\, D\pi(\xi) \text{ is constant} \}
\ar[rr] && \smooth(\bC P^1, \Omega^{0,1} \otimes_{\bC} u^*TF)\ar[d] \\
\ar[d]^-{D\pi}
\smooth(\bC P^1, u^*TE) \ar[rr]^-{D_{E,u}} && \smooth(\bC P^1, \Omega^{0,1} \otimes_{\bC} u^*TE) \ar[d]^-{D\pi} \\
\ar[d] \smooth(\bC P^1, \bC)/\mathit{constants} \ar[rr]^-{\bar\partial} && \smooth(\bC P^1,\Omega^{0,1}) \ar[d]
\\ 0 && 0.
}
\end{equation}
Here, $D_{E,u}$ is the linearization of $u$ as a pseudo-holomorphic map into $E$. The top $\rightarrow$ is the linearization of $u$ as a pseudo-holomorphic map into a variable fibre (its domain is an extension of $\smooth(\bC P^1,u^*TF)$ by $\bC = TB_w$). The parametrized transversality theory shows that the top $\rightarrow$ is onto for generic $J$. But since the bottom $\rightarrow$ is invertible, we get the same surjectivity result for $D_{E,u}$.

(iii) $v(z) \in \partial N$ for all $z$. In that case, $v$ is holomorphic, hence must again be constant. The same strategy as in (ii) applies.
\end{proof}

A similar result holds for one-periodic orbits:

\begin{lemma} \label{th:1-periodic}
Suppose that $a \in \Omega^1(S^1,\frakg)$ has nontrivial holonomy. Then, for generic choice of $H \in \scrH(E,a)$, all one-periodic orbits \eqref{eq:periodic-orbit} are nondegenerate.
\end{lemma}

\begin{proof}
As before, we need to be a bit more precise concerning the statement itself. Each $H_t \in \scrH(E,a_t)$ comes with some neighbourhood of infinity $N_{H_t}$, but we can fix an $N$ which is smaller than all of them (see Remark \ref{th:families-of-neighbourhoods}, and take into account the compactness of the parameter space $P = S^1$). In the case where the holonomy is elliptic, we assume that \eqref{eq:downstairs-periodic-orbit} has no solutions satisfying $y(t) \in N$ for all $t$ (for parabolic or hyperbolic holonomy, this condition would be vacuous, since their holonomy acts fixed point freely). That being given, it is clear that perturbing $H$ only outside $\pi^{-1}(N)$ is sufficient to achieve transversality.
\end{proof}

Let's keep the assumption that $a$ has nontrivial holonomy, and choose $H \in \scrH(E,a)$, satisfying the nondegeneracy condition from Lemma \ref{th:1-periodic}. Suppose also that we have a family $J = (J_t)$ of almost complex structures in $\scrJ(E)$. Given one-periodic orbits $x_{\pm}$ as in \eqref{eq:periodic-orbit}, we consider Floer trajectories
\begin{equation} \label{eq:e-floer}
\left\{
\begin{aligned}
& u: \bR \times S^1 \longrightarrow E, \\
& \partial_s u + J_t(u) (\partial_t u - X_t(u)) = 0, \\
& \textstyle \lim_{s \rightarrow \pm\infty} u(s,t) = x_{\pm}(t).
\end{aligned}
\right.
\end{equation}
The pair $\Phi = (J,H)$ will be called a ``Floer datum'', since this is the structure required to write down Floer's equation \eqref{eq:e-floer} (nondegeneracy of one-periodic orbits is part of the definition of Floer datum).

\begin{lemma} \label{th:floer-transversality}
For generic choice of $J$ (keeping $H$ fixed), all solutions of \eqref{eq:e-floer} are regular. Moreover, if $u$ is a solution whose linearized operator $D_u$ has index $\leq 2$, then $u(s,t)$ never lies on a nontrivial $J_t$-holomorphic sphere.
\end{lemma}

This is essentially the transversality result from \cite{floer-hofer-salamon94,hofer-salamon95}. Applied to solutions $u(s,t) = x(t)$ ($D_u$ is invertible in that case, because of the nondegeneracy of $x$), the result includes the statement that for any one-periodic orbit, $x(t)$ should not lie on a nontrivial $J_t$-holomorphic sphere; this is a version of Lemma \ref{th:j-spheres} including transversality of evaluations. For the remaining part of Lemma \ref{th:floer-transversality}, one starts with a fixed $\Phi = (J,H)$, and takes some neighbourhood of infinity $N$ which is smaller than $N_{J_t}$ and $N_{H_t}$ for any $t$, and which has the same property concerning \eqref{eq:downstairs-periodic-orbit} as in the proof of Lemma \ref{th:1-periodic}. Then, since the orbits $x_{\pm}$ aren't  contained in $\pi^{-1}(N)$, the maps $u$ also can't be contained in that subset, which means that perturbing $J$ outside $\pi^{-1}(N)$ gives one enough freedom.

To show that Floer trajectories can't escape to infinity, we will combine elementary complex analysis (as in Section \ref{subsec:schwarz}) and pseudo-holomorphic curve theory. Some basic ingredients are:

\begin{itemize}
\itemsep0.5em

\item The notion of energy of a Floer trajectory,
\begin{equation}
E(u) = \int_{\bR \times S^1} \|\partial_s u\|^2_{J_t} = \int_{\bR \times S^1} u^*\omega_E - \int_{S^1} H_t(x_+(t))\, \mathit{dt}  + \int_{S^1} H_t(x_-(t))\, \mathit{dt}.
\end{equation}

\item The ``gauge transformation'' which makes Floer's equation homogeneous. Namely,
let $(\phi_t)$ be the Hamiltonian isotopy generated by $(X_t)$. The change of variables $u(s,t) = \phi_t(u^\dag(s,t))$ turns \eqref{eq:e-floer} into
\begin{equation} \label{eq:u-star}
\left\{
\begin{aligned}
& u^\dag: \bR^2 \longrightarrow E, \\
& u^\dag(s,t) = \phi_1(u^\dag(s,t+1)), \\ 
& \partial_s u^\dag + J_t^\dag(u^\dag) \partial_t u^\dag = 0, \\
& \textstyle \lim_{s \rightarrow \pm \infty} u^\dag(s,t) = x^\dag_{\pm},
\end{aligned}
\right.
\end{equation}
where $J_t = (\phi_t)_* J_t^\dag$, and $x^\dag_{\pm} = x_{\pm}(0)$. Outside a compact subset, $\phi_t$ covers an isotopy in $\mathit{Aut}(B)$. Hence, $J_t^\dag$ again belongs to $\scrJ(E)$. 

\item Projection to the base. Wherever $u$ is suffficiently close to $\partial\bar{E}$, the projection $v = \pi(u)$ satisfies a special case of \eqref{eq:inhomogeneous}, where $A$ is the pullback of $a$ to the cylinder (hence flat):
\begin{equation} \label{eq:downstairs-floer}
\partial_s v + i(\partial_t v - X_{a_t}(v)) = 0.
\end{equation}
\end{itemize}

\begin{lemma} \label{th:e-bound-0}
Suppose that $(u_k)$ is a sequence of solutions of \eqref{eq:e-floer}, with an upper bound on the energy. Assume that there are points $(s_k,t_k) \in \bR \times S^1$ such that $|s_k|$ is bounded, and $u_k(s_k,t_k) \rightarrow \partial \bar{E}$. Then, the $u_k$ converge to $\partial \bar{E}$ uniformly on compact subsets. 
\end{lemma}

\begin{proof}
Gromov compactness (on compact subsets of the domain) can be applied to the $u_k$, seen as maps taking values in $\bar{E}$. This may seem suspicious, but is justified by the following argument. Consider the pseudo-holomorphic maps $u_k^\dag$ as in \eqref{eq:u-star}. On each compact subset of $\bR^2$, we have an upper bound on their energies with respect to $\omega_E$, hence by \eqref{eq:e-e} also for $\omega_{\bar{E}}$. When applying Gromov compactness, one thinks of the latter symplectic form on $\bar{E}$.

With this in mind, the Gromov limit (of a subsequence of the $u_k$) consists of a principal component $u_\infty: \bR \times S^1 \rightarrow \bar{E}$, which solves the Cauchy-Riemann equation from \eqref{eq:e-floer}, together with bubble components which are pseudo-holomorphic spheres (unlike the usual compactification of Floer trajectory spaces, the notion of Gromov limit we are working with here ignores any components which might split off over the ends $s \rightarrow \pm\infty$). By assumption, one of our components must intersect $\partial \bar{E}$. If the principal component does that, then by looking at its projection to $\bar{B}$ (as in Lemma \ref{th:flat-convergence}), one concludes that $u_\infty(\bR \times S^1) \subset \partial \bar{E}$. The same argument applies to bubble components. Because the Gromov limit is connected, a component-by-component induction shows that all of it lies in $\partial \bar{E}$. This implies that a subsequence of the $u_k$ converges to $\partial \bar{E}$ on compact subsets. Going from that to the entire sequence works as in Lemma \ref{th:flat-convergence}.
\end{proof}
%

\begin{lemma} \label{th:e-bound-1}
Assume that the holonomy of $a$ is elliptic or hyperbolic. Suppose that $(u_k)$ is a sequence of solutions of \eqref{eq:e-floer}, with an upper bound on the energy. Then, there is a compact subset of $E$ which contains the images of all the $u_k$.
\end{lemma}

\begin{proof}
Let's suppose that on the contrary, there are (possibly after passing to a subsequence) points $(s_k,t_k)$ such that $u_k(s_k,t_k) \rightarrow \partial \bar{E}$. After translation in $s$-direction, we may assume that $|s_k|$ is bounded. Lemma \ref{th:e-bound-0} implies that, on each relatively compact subset of $\bR \times S^1$, $v_k = \pi(u_k)$, for $k \gg 0$, is a solution of \eqref{eq:downstairs-floer}. In the hyperbolic case, one takes that subset to be a sufficiently long finite cylinder $(-l/2,l/2) \times S^1$, and gets a contradiction to Lemma \ref{th:hyp-cylinder}. In the slightly more complicated elliptic case, Lemma \ref{th:elliptic-cylinder} shows that $v_k(0,0)$ must lie in a specific compact subset of $B$, which leads to a contradiction to Lemma \ref{th:e-bound-0}.
\end{proof}

\subsection{General surfaces\label{subsec:worldsheets}}
We now extend \eqref{eq:e-floer} to other Riemann surfaces, building a familiar kind of TQFT framework (compare e.g.\ \cite{schwarz95}). Let $S$ be a punctured (connected, but not closed) Riemann surface. We arbitrarily divide the set $I$ of punctures into subsets $I_{\mathit{in}}$ and $I_{\mathit{out}}$. The surface should be decorated with the following additional data:
\begin{itemize} \itemsep.5em
\item Tubular ends surrounding the punctures, with pairwise disjoint images,
\begin{equation}
\left\{
\begin{aligned}
& \epsilon_i: [0,\infty) \times S^1 \hookrightarrow S, && i \in I_{\mathit{in}}, \\
& \epsilon_i: (-\infty,0] \times S^1 \hookrightarrow S, && i \in I_{\mathit{out}}.
\end{aligned}
\right.
\end{equation}

\item
A connection $A \in \Omega^1(S,\frakg)$ which is nonnegatively curved and compatible with the tubular ends, meaning that
\begin{equation} \label{eq:end-a}
\epsilon_i^*A = a_i = a_{i,t} \mathit{dt}
\end{equation}
for some $a_i \in \Omega^1(S^1,\frakg)$. Moreover, each $a_i$ should have nontrivial holonomy.
\end{itemize}
We write $Y = (\{\epsilon_i\},A)$, and refer to $(S,Y)$ as a decorated surface (this is closely related to, but not exactly the same as, the notion from Section \ref{sec:tqft}; see Section \ref{sec:end} for further discussion). 
The connection $A$ gives rise to an equation \eqref{eq:inhomogeneous} for maps $S \rightarrow B$, which has the form \eqref{eq:downstairs-floer} on each end. The corresponding construction for maps to $E$ requires the following additional choices:
\begin{itemize}
\itemsep0.5em
\item
For each end, we fix a Floer datum $\Phi_i = (J_i,H_i)$ in the class determined by $a_i$.
\item
Let $J = (J_z)_{z \in S}$ be a family of almost complex structures, $J_z \in \scrJ(E)$, which on the tubular ends satisfies the following asymptotic condition:
\begin{equation} \label{eq:exponential-asymptotics}
J_{\epsilon_i(s,t)} \longrightarrow J_{i,t} \quad \text{exponentially fast as $s \rightarrow \pm\infty$.}
\end{equation}
\end{itemize}
We need to make the last-mentioned condition more precise, in several respects. The difference between $J_{\epsilon_i(s,t)}$ and $J_{i,t}$ is measured after extension to the compactification $\bar{E}$. ``Exponentially fast'' means that in the $C^r$ sense for any $r$, the difference between the structures is bounded by a constant times $e^{-\lambda |s|}$, for some $\lambda>0$. Finally, for each end there should be a neighbourhood of infinity $N = N_{J,i} \subset N_E$, such that all $J_{\epsilon_i(s,t)}$ make $\pi$ pseudo-holomorphic on $\pi^{-1}(N)$.
\begin{itemize}
\item
Take $K \in \smooth(S \times E, T^*S)$, thought of as a one-form on $S$ with values in $\smooth(E,\bR)$. This must be such that $K(\xi) \in \scrH(E, A(\xi))$ for all $\xi \in TS$. Over the ends, we require the following (this time, a strict equality, rather than an asymptotic condition):
\begin{equation}
\epsilon_i^*K = H_{i,t} \mathit{dt}.
\end{equation}
\end{itemize}
We will refer to the pair $\Pi = (J,K)$ as a ``perturbation datum''. Let $X_K \in \smooth(S \times E, T^*S \otimes_{\bR} TE)$ be the Hamiltonian-vector-field-valued one-form associated to $K$. We will consider maps
\begin{equation} \label{eq:e-equation}
\left\{
\begin{aligned}
& u: S \longrightarrow E, \\
& (Du - X_K(u))^{0,1} = 0, \\
&\textstyle  \lim_{s \rightarrow \pm\infty} u(\epsilon_i(s,t)) = x_i(t).
\end{aligned}
\right.
\end{equation}
The Cauchy-Riemann equation in \eqref{eq:e-equation} says that $Du - X_K(u): TS_z \rightarrow TE_{u(z)}$ should be complex-linear with respect to $J_{z}$. On each tubular end, setting $u_i(s,t) = u(\epsilon_i(s,t))$, this equation reduces to
\begin{equation} \label{eq:semi-floer}
\partial_s u_i + J_{\epsilon_i(s,t)} (\partial_t u_i - X_{i,t}(u_i)) = 0.
\end{equation}
The limits $x_i$ are one-periodic orbits of $H_i$. In local coordinates $z=s+it$ on $S$ in which $K = K_1 \mathit{ds} + K_2 \mathit{dt}$, the equation is
\begin{equation}
\partial_s u + J_{s,t} \partial_t u = X_{K_1(s,t)} + J_{s,t} X_{K_2(s,t)}.
\end{equation}

\begin{lemma} \label{th:transversality-on-s}
For generic choice of perturbation datum $\Pi$, keeping $\Phi_i$ fixed, the spaces of solutions of \eqref{eq:e-equation} are regular. Moreover, if $u$ is a solution whose linearized operator $D_u$ has index $\leq 1$, then $u(z)$ never lies on a nontrivial $J_z$-holomorphc sphere.
\end{lemma}

\begin{proof}[Sketch of proof]
Let's fix a suitably small neighbourhood of infinity $N$, so that none of the $x_i$ is contained in $\pi^{-1}(N)$. Over each end, one of the following applies (this is a version of \cite[Lemma 4.1]{floer-hofer-salamon94} with additional $s$-dependence of the almost complex structure, but has the same proof):

(i) The set of points where $\partial_s u_i$ vanishes is discrete. If there is at least one end with this property, there are points $(s,t)$ where $\partial_s u_i \neq 0$ and $u_i(s,t) \notin \pi^{-1}(N)$. Transversality can then be achieved by perturbing the family $J_z$ near $z = \epsilon_i(s,t)$.

(ii) $\partial_s u_i$ vanishes identically. This means that $u_i(s,t) = x_i(t)$. In particular, there are $t$ such that $z =\epsilon_i(0,t)$ satisfies $u(z) \notin \pi^{-1}(N)$. Since this is an open condition, there are also points $z \in S$ which are disjoint from the tubular ends and such that $u(z) \notin \pi^{-1}(N)$. We can then achieve transversality by varying $(J,K)$ near such a point.
\end{proof}

We will now tackle the issue of compactness. This is based on similar ideas as in the case of Floer trajectories:
\begin{itemize} \itemsep.5em
\item There are two notions of energy for solutions of \eqref{eq:e-equation}, paralleling \eqref{eq:geometric-energy} and \eqref{eq:topological-energy} from our toy model discussion:
\begin{align}
& E_{\mathit{geom}}(u) = \half \int_S |Du - X_K(u)|^2, \\
& E_{\mathit{top}}(u) = \int_S u^*\omega_E - d(K(u)).
\end{align}
These are related by a curvature term,
\begin{equation} \label{eq:e-e-e}
E_{\mathit{top}}(u) = E_{\mathit{geom}}(u) + \int_S R_K(u).
\end{equation}
We do not want to spell out the definition of $R_K$, but on a suitable neighbourhood of infinity, it agrees with the pullback of the corresponding term in \eqref{eq:e-e}. Since we are assuming that $A$ is nonnegatively curved, it follows that the integrand in \eqref{eq:e-e-e} is bounded below. Moreover, it vanishes identically on the tubular ends. It follows that there is a constant $C = C_K$ such that for all $u$,
\begin{equation} \label{eq:e-e-bounds}
E_{\mathit{geom}}(u) \leq E_{\mathit{top}}(u) + C.
\end{equation}

\item To any perturbation datum $\Pi = (J,K)$ one can associate an almost complex structure $J_\Pi$ on the trivial fibre bundle $S \times E \rightarrow S$, such that \eqref{eq:e-equation} turns into the equation for $J_\Pi$-holomorphic sections. In local coordinates on $S$, this is given by the same kind of formula as in \eqref{eq:j-a}, 
\begin{equation}
\left\{ 
\begin{aligned}
& J_\Pi(Y) = J_{s,t} Y \quad \text{for $Y \in TE$,} \\
& J_\Pi(\partial_s + X_{K_1}) = \partial_t + X_{K_2}.
\end{aligned}
\right.
\end{equation}
Near $\partial \bar{E}$, the vector fields $X_{K_1}, X_{K_2} \in \smooth(E,TE)$ are the horizontal lifts of the corresponding vector fields $X_{A_1}, X_{A_2} \in \smooth(B,TB)$. Hence, they extend smoothly to $\bar{E}$, and so does $J_\Pi$.

\item At any point $z \in S$ such that $u(z)$ is sufficiently close to $\partial \bar{E}$, the projection $v = \pi(u)$ satisfies the equation \eqref{eq:inhomogeneous} associated to $A$. 
\end{itemize}

\begin{lemma} \label{th:u-bound-0}
Suppose that $(u_k)$ is a sequence of solutions of \eqref{eq:e-equation}, with an upper bound on the topological energy. Suppose that there are points $z_k$, contained in a compact subset of $S$, such that $u_k(z_k) \rightarrow \partial \bar{E}$. Then, the $u_k$ converge to $\partial \bar{E}$ uniformly on compact subsets. 
\end{lemma}

\begin{proof}
By \eqref{eq:e-e-bounds}, we have an upper bound on the geometric energy. From there on,
the proof follows the same strategy as its Floer-theoretic counterpart, Lemma \ref{th:e-bound-0}. This time, the applicability of Gromov compactness has to be explained in a slightly different way. Think of our maps as $J_\Pi$-holomorphic sections $S \rightarrow S \times \bar{E}$. From the bound on the geometric energy (and the fact that $X_K$ extends smoothly to $\bar{E}$), we get a $W^{1,2}$-bound on those sections, on each fixed compact subset of $S$. This already suffices for Gromov compactness, without having to introduce a symplectic structure on $S \times \bar{E}$ (see e.g.\ \cite{zinger17} for a suitable exposition). Alternatively, one can argue by the following familiar trick: since $J_\Pi | \{z\} \times \bar{E}$ is compatible with \eqref{eq:omega-bar-e} for each $z$, one can add a large positive two-form on $S$ to create a symplectic form on $S \times \bar{E}$ which at least tames $J_\Pi$. 

The Gromov limit (of a subsequence, in the sense of convergence on compact subsets) consists of a solution $u_\infty: S \rightarrow \bar{E}$ of the same $\bar\partial$-equation as in \eqref{eq:e-equation}, together with bubble components which are $J_z$-holomorphic spheres. By projecting to $\bar{B}$ and using the same convexity argument as in Lemma \ref{th:levi-apply}, one sees that $u_\infty^{-1}(\partial \bar{E})$ is open and closed. The same applies to the bubble components, as in Lemma \ref{th:e-bound-0}, and we again conclude that under our assumptions on $u_k$, the entire Gromov limit must lie in $\partial \bar{E}$.
\end{proof}

\begin{lemma} \label{th:u-bound-1}
Assume that the $a_i$ have elliptic or hyperbolic holonomies. Suppose that $(u_k)$ is a sequence of solutions of \eqref{eq:e-equation}, with an upper bound on the topological energy. Then, there is a compact subset of $E$ which contains the images of all the $u_k$.
\end{lemma}

\begin{proof}
This is modelled on Lemma \ref{th:e-bound-1}. Let's suppose that, on the contrary (after passing to a subsequence), there are $z_k$ such that $u_k(z_k) \rightarrow \partial \bar{E}$. 

(i) If  (again after passing to a subsequence) the $z_k$ are contained in a compact subset, we can apply Lemma \ref{th:u-bound-0}, which implies that $u_k \rightarrow \partial \bar{E}$ on compact subsets. In particular, this convergence holds for the restriction of $u_k$ to any fixed finite cylinder $\epsilon_i((s-l/2,s+l/2) \times S^1) \subset S$, which yields a sequence of solutions of \eqref{eq:semi-floer} that approach $\partial\bar{E}$. By taking $l$ to be sufficiently large, and looking at $v_k = \pi(u_k)$, one obtains a contradiction with Lemma \ref{th:hyp-cylinder} or Lemma \ref{th:elliptic-cylinder}.

(ii) The other case one has to deal with is where $z_k = \epsilon_i(s_k,t_k)$ for some $|s_k| \rightarrow \infty$. Consider the translated versions
\begin{equation} \label{eq:translated}
u_{k,i}(\cdot+s_k,\cdot): (-|s_k|,|s_k|) \times S^1 \longrightarrow E.
\end{equation}
By definition, the image of $(0,t_k)$ under \eqref{eq:translated} approaches $\partial \bar{E}$. Moreover, these maps satisfy a sequence of Cauchy-Riemann equations which, as $k \rightarrow \infty$, converge to an equation \eqref{eq:e-floer}. A version of the argument from Lemma \ref{th:u-bound-0} can be applied to this sequence; after that, one proceeds as in case (i).
\end{proof}

\begin{figure}
\begin{centering}
\begin{picture}(0,0)%
\includegraphics{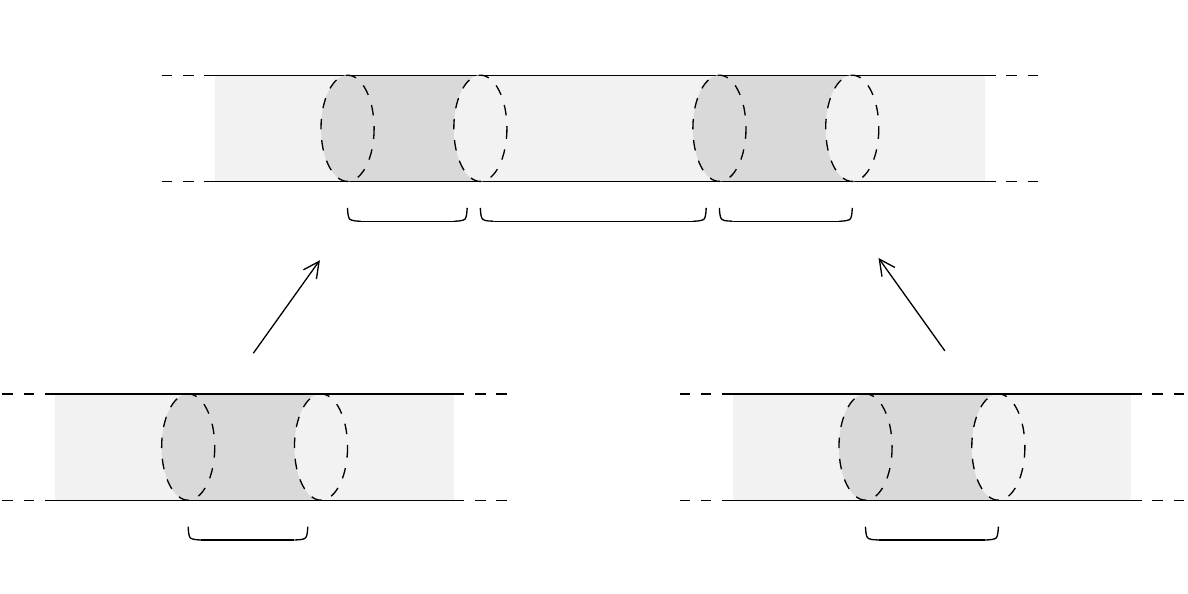}%
\end{picture}%
\setlength{\unitlength}{3355sp}%
\begingroup\makeatletter\ifx\SetFigFont\undefined%
\gdef\SetFigFont#1#2#3#4#5{%
  \reset@font\fontsize{#1}{#2pt}%
  \fontfamily{#3}\fontseries{#4}\fontshape{#5}%
  \selectfont}%
\fi\endgroup%
\begin{picture}(6699,3336)(-2261,-3439)
\put(2851,-3361){\makebox(0,0)[lb]{\smash{{\SetFigFont{10}{12.0}{\rmdefault}{\mddefault}{\updefault}{$l_{\mathit{right}}$}%
}}}}
\put(1876,-1561){\makebox(0,0)[lb]{\smash{{\SetFigFont{10}{12.0}{\rmdefault}{\mddefault}{\updefault}{$l_{\mathit{right}}$}%
}}}}
\put(976,-316){\makebox(0,0)[lb]{\smash{{\SetFigFont{10}{12.0}{\rmdefault}{\mddefault}{\updefault}{$S_k$}%
}}}}
\put(751,-1561){\makebox(0,0)[lb]{\smash{{\SetFigFont{10}{12.0}{\rmdefault}{\mddefault}{\updefault}{$l_{\mathit{neck},k}$}%
}}}}
\put(-1049,-3361){\makebox(0,0)[lb]{\smash{{\SetFigFont{10}{12.0}{\rmdefault}{\mddefault}{\updefault}{$l_{\mathit{left}}$}%
}}}}
\put(-149,-1561){\makebox(0,0)[lb]{\smash{{\SetFigFont{10}{12.0}{\rmdefault}{\mddefault}{\updefault}{$l_{\mathit{left}}$}%
}}}}
\put(-1649,-2161){\makebox(0,0)[lb]{\smash{{\SetFigFont{10}{12.0}{\rmdefault}{\mddefault}{\updefault}{$S_{\mathit{left}}$}%
}}}}
\put(1876,-2161){\makebox(0,0)[lb]{\smash{{\SetFigFont{10}{12.0}{\rmdefault}{\mddefault}{\updefault}{$S_{\mathit{right}}$}%
}}}}
\end{picture}%
\caption{\label{fig:stretch}Notation from the neck-stretching process.}
\end{centering}
\end{figure}
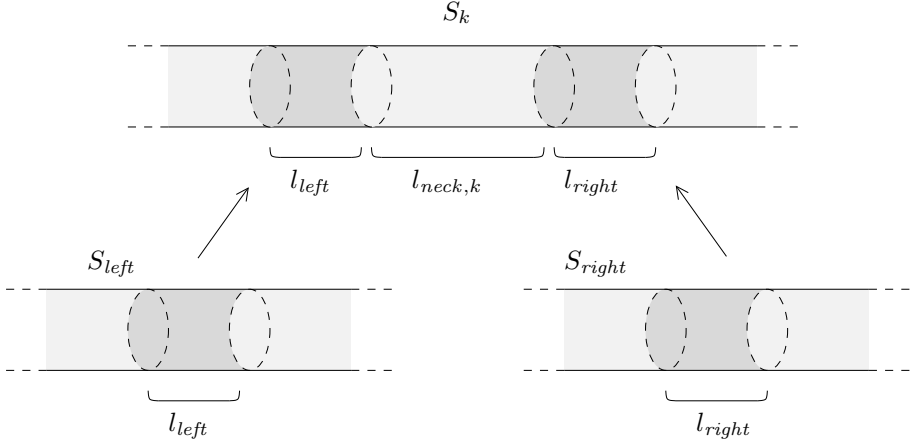
A frequently encountered class of generalizations is that where, instead of maps defined on a single Riemann surface, one has a family of surfaces. If the parameter space for the family is compact, the previous arguments carry over straightforwardly. A slightly more interesting situation is ``neck-stretching'', where the Riemann surfaces split into pieces. Let's consider the simplest example of such behaviour, which is one infinite cylinder splitting into two (the general case differs mainly in being even more cumbersome to formulate). The basic situation is: 
\begin{itemize} \itemsep.5em
\item
Fix $a_{\mathit{left}}, a_{\mathit{neck}}, a_{\mathit{right}} \in \Omega^1(S^1,\frakg)$, all of which have nontrivial holonomy.

\item
Fix constants $l_{\mathit{left}}, l_{\mathit{right}} > 0$, and a sequence $l_{\mathit{neck},k} > 0$, $l_{\mathit{neck},k} \rightarrow \infty$. These govern the length of various pieces of our cylinder, see Figure \ref{fig:stretch}.

\item
Let's write $S_k$ ($k = 1,2,\dots$), $S_{\mathit{left}}$ and $S_{\mathit{right}}$ for copies of $\bR \times S^1$. We will think of these as mapped to each other as follows:
\begin{equation} \label{eq:sigma-embeddings}
\left\{
\begin{aligned}
& \sigma_{\mathit{left}}: S_{\mathit{left}} \longrightarrow S_k, && \sigma_{\mathit{left}}(s,t) = (s-l_{\mathit{left}}/2-l_{\mathit{neck},k}/2,t), \\
& \sigma_{\mathit{right}}: S_{\mathit{right}} \longrightarrow S_k, && \sigma_{\mathit{right}}(s,t) = (s+l_{\mathit{right}}/2+l_{\mathit{neck},k}/2,t).
\end{aligned}
\right.
\end{equation}
We give $S_{\mathit{left}}$ tubular ends, whose images are $(-\infty,-l_{\mathit{left}}/2] \times S^1$ and $[l_{\mathit{left}}/2,\infty) \times S^1$; and correspondingly for $S_{\mathit{right}}$. The ends of $S_k$ have image $(-\infty,-l_{\mathit{left}}-l_{\mathit{neck},k}/2] \times S^1$ and $[l_{\mathit{right}}+l_{\mathit{neck},k}/2,\infty) \times S^1$. 

\item
Each $S_k$ carries a nonnegatively curved connection $A_k$ which restricts to $a_{\mathit{left}}$, $a_{\mathit{right}}$ on its tubular ends, and to $a_{\mathit{neck}}$ on $[-l_{\mathit{neck},k}/2, l_{\mathit{neck},k}/2] \times S^1$. Similarly, $S_{\mathit{left}}$ carries a nonnegatively curved connection $A_{\mathit{left}}$ which restricts to $a_{\mathit{left}}$, $a_{\mathit{neck}}$ on the tubular ends. Moreover, the pullbacks of $A_k$ by \eqref{eq:sigma-embeddings} converge to $A_{\mathit{left}}$ on compact subsets. On $S_{\mathit{right}}$, we have a connection $A_{\mathit{right}}$ with corresponding properties.
\end{itemize}
When setting up the pseudo-holomorphic curve theory, the perturbation data have to be similarly correlated:
\begin{itemize} \itemsep.5em \parindent0em \parskip1em
\item We suppose that Floer data $\Phi_{\mathit{left}} = (J_{\mathit{left},t},H_{\mathit{left},t})$, $\Phi_{\mathit{neck}} = (J_{\mathit{neck},t},H_{\mathit{neck},t})$, and $\Phi_{\mathit{right}} = (J_{\mathit{right},t},H_{\mathit{right},t})$ have been chosen. The Hamiltonians should lie in the classes specified by the previous $a$'s. We also choose perturbation data $\Pi_k = (J_k,K_k)$, $\Pi_{\mathit{left}} = (J_{\mathit{left}}, K_{\mathit{left}})$, and $\Phi_{\mathit{right}} = (J_{\mathit{right}}, K_{\mathit{right}})$, whose behaviour over the ends is governed by the appropriate Floer data. 

\item For the inhomogeneous terms $K_k$ which are part of $\Pi_k$, we additionally prescribe the behaviour on the neck:
\begin{equation}
K_{k,s,t} = 
H_{\mathit{neck},t} \mathit{dt} \quad \text{for $-l_{\mathit{neck},k}/2 \leq s \leq l_{\mathit{neck},k}/2$.}
\end{equation}
The pullback of $K_k$ by \eqref{eq:sigma-embeddings} should converge on compact subsets to $K_{\mathit{left}}$, $K_{\mathit{right}}$.

\item The requirements for the families of almost complex $J_k$ are the most complicated part of the definition, because of the asymptotic condition \eqref{eq:exponential-asymptotics}. Pulling back $J_k$ by $\sigma_{\mathit{left}}$, and then restricting to $(-\infty, -l_{\mathit{left}}/2] \times S^1$, should yield a family of almost complex structures which converges to $J_{\mathit{left}}$; and that convergence should be in any $C^r$ topology with sufficiently small exponential weights over the end $s \rightarrow -\infty$. Take the same pullback, and restrict to it $[-l_{\mathit{left}}/2, l_{\mathit{left}}/2+l_{\mathit{neck,k}}/2 + C] \times S^1$ for some constant $C$, which yields a family of almost complex structures parametrized by increasingly large parts of $S_{\mathit{left}}$. As $k \rightarrow \infty$, that family should converge to $J_{\mathit{left}}$; and that convergence should be in any $C^r$ topology with sufficiently small exponential weights (where the weight concerns the direction $s \rightarrow +\infty$). Parallel requirements apply to the pullback by $\sigma_{\mathit{right}}$.
\end{itemize}

This notion of neck-stretching is designed to allow suitable gluing arguments to go through, while preserving our previous transversality and compactness statements. More precisely, the appropriate version of Lemma \ref{th:transversality-on-s} says that, keeping the $\Phi$'s fixed, a generic perturbation of the entire sequence $\Pi_k = (J_k,K_k)$ (within the allowed class) makes the associated moduli spaces regular for all $k$. Similarly, one gets a version of Lemma \ref{th:u-bound-1}, assuming that the holonomies involved (of $a_{\mathit{left}}, a_{\mathit{neck}}, a_{\mathit{right}}$) are elliptic or hyperbolic. The strategy of proof remains the same, with the breakdown into cases understood in an appropriate way. Namely, take maps $u_k: S_k \rightarrow E$ and points $z_k \in S_k$ such that $u_k(z_k) \rightarrow \partial \bar{E}$. (i) If (possibly after passing to a subsequence) the $z_k$ have preimages under \eqref{eq:sigma-embeddings} which lie in a compact subset of $S_{\mathit{left}}$, then one considers a Gromov limit which is a solution of an equation \eqref{eq:e-equation} involving $\Pi_{\mathit{left}}$; and similarly for $S_{\mathit{right}}$. (ii) In all other cases, one gets a limit which is a Floer trajectory for one of our chosen Floer data.

\subsection{Formal aspects of Floer cohomology}
The rest of the construction of Floer cohomology is standard, and we will keep the discussion brief. From now on, we will only allow $a \in \Omega^1(S^1,\frakg)$ with elliptic or hyperbolic holonomy. We also restrict our notion of decorated surface, by requiring the same property for all $a_i$.

Fix an $a$. Choose an associated Floer datum $\Phi = (H,J)$, which satisfies the transversality statement from Lemma \ref{th:floer-transversality}. We define a $\bZ$-graded cochain complex over the Novikov field \eqref{eq:novikov}:
\begin{equation}
\mathit{CF}^*(E,H) = \bigoplus_x \bK_x,
\end{equation}
where the sum is over the (finitely many) one-periodic orbits \eqref{eq:periodic-orbit}. Each $\bK_x$ is an orientation space, which means that it is identified with $\bK$ in a way which is canonical up to sign. To simplify the notation, we make an arbitrary choice of such an identification, and denote the resulting basis element in $\mathit{CF}^*(E,H)$ simply by $x$. The grading is given by Conley-Zehnder indices (which are well-defined since we have chosen a trivialization of the anticanonical bundle of $E$). The Floer differential is
\begin{equation} \label{eq:floer-diff}
d(x_+) = \sum_{u} \pm q^{E(u)} x_-,
\end{equation}
where the sum is over all solutions of $u$ of \eqref{eq:e-floer} which are isolated up to translation. Lemma \ref{th:e-bound-1}, together with more standard compactness arguments, shows that there are only finitely many such $u$ whose energy is lower than any given bound, hence that \eqref{eq:floer-diff} makes sense. The sign with which each $u$ contributes is determined by gluing of determinant lines. One defines $\mathit{HF}^*(E,H)$ to be the cohomology of $(\mathit{CF}^*(E,H),d)$.

Suppose that we have a decorated surface $(S,Y)$, and for each end, a Floer datum $\Phi_i$ whose Floer complex is well-defined. Choose a perturbation datum $\Pi$ satisfying Lemma \ref{th:transversality-on-s}. Then, counting solutions of \eqref{eq:e-equation} yields a chain map of even degree,
\begin{equation} \label{eq:floer-gamma}
\begin{aligned}
& \phi: \bigotimes_{i \in I_{\mathit{in}}} \mathit{CF}^*(E,H_i) \longrightarrow
\Big( \bigotimes_{i \in I_{\mathit{out}}} \mathit{CF}^*(E,H_i) \Big) [n(\chi(S)+|I_{\mathit{in}}|-|I_{\mathit{out}}|)], \\
& \phi(\otimes_{i \in I_{\mathit{in}}} x_i) = \sum_u \pm q^{E_{\mathit{top}}(u)} (\otimes_{i \in I_\mathit{out}} x_i).
\end{aligned}
\end{equation}
Here, the sum is over solutions with limits $\{x_i\}$ which are isolated; the topological energy is used, since that is locally constant in families of solutions; and there is a sign, exactly as in \eqref{eq:floer-diff}. The same holds more generally if we have a family of such surfaces parametrized by a closed oriented manifold $P$ (the Floer data are parameter-independent, but the perturbation data obviously change). The only difference is that the degree of the associated map $\phi_P$ is $\mathrm{dim}(P)$ less than in \eqref{eq:floer-gamma}. 

The next simplest case is where the parameter space is a compact oriented manifold with corners. In that case, the associated operation is no longer a chain map, but instead a nullhomotopy for the operations associated to the boundary faces. Let's suppose for simplicity that exactly one end is an output, say $I_{\mathit{out}} = \{0\}$ and $I_{\mathit{in}} = \{1,\dots,m\}$; this will be the case in all our applications. Then, the precise expression is exactly as in \eqref{eq:boundary-sign}.

It remains to to take a brief look at gluing. A simple case would be where we have a parameter space $P$ with $\partial P = P_1 \times P_2$. The associated family of surfaces is defined only over $P \setminus \partial P$, and as we approach the boundary, it stretches along a neck, with the limit being two families over $P_1$ and $P_2$, respectively. The behaviour of the perturbation data must be as discussed at the end of Section \ref{subsec:worldsheets}. One can then still define an operation associated to $P$, and that satisfies (under the same one-output assumption as before)
\begin{equation} \label{eq:boundary-stretching}
\begin{aligned}
& (-1)^{\mathrm{dim}(P)} d\phi_P(x_1,\dots,x_m) - \phi_P(dx_1,\dots,x_m) - \cdots 
- (-1)^{|x_1|+\cdots+|x_{m-1}|} \phi_P(x_1,\dots,dx_m) \\ & \qquad + (-1)^{(\mathrm{dim}(P_1) + |x_1| + \cdots + |x_i|) \mathrm{dim}(P_2)} \phi_{P_1}(x_1,\dots,\phi_{P_2}(x_{i+1},\dots,x_{i+m_2}), \dots, x_m) = 0,
\end{aligned}
\end{equation}
where $m = m_1+m_2-1$; note that $(-1)^{\mathrm{\dim}(P)} = (-1)^{|\phi_P|}$. As a final example, consider a mixture of the two previously encountered situations, where $\partial P = P_0 \sqcup -(P_1 \times P_2)$ (the minus sign indicates orientation-reversal). The associated family of surfaces should extend over $P_0$, but on the other boundary component, we have neck-stretching as before. Then, one gets 
\begin{equation} \label{eq:boundary-stretching-2}
\begin{aligned}
& (-1)^{\mathrm{dim}(P)} d\phi_P(x_1,\dots,x_m) - \phi_P(dx_1,\dots,x_m) - \cdots \\ & \qquad \cdots
- (-1)^{|x_1|+\cdots+|x_{m-1}|} \phi_P(x_1,\dots,dx_m) + \phi_{P_0}(x_1,\dots,x_m) \\ & \qquad - (-1)^{(\mathrm{dim}(P_1) + |x_1| + \cdots + |x_i|) \mathrm{dim}(P_2)} \phi_{P_1}(x_1,\dots,\phi_{P_2}(x_{i+1},\dots,x_{i+m_2}), \dots, x_m) = 0.
\end{aligned}
\end{equation}
In the simplest case where all the $P_k$ are closed and $m _k = 1$, this means that $(-1)^{\mathrm{dim}(P)}\phi_P$ yields a chain homotopy 
\begin{equation}
\phi_{P_0} \htp (-1)^{|\phi_1| \cdot |\phi_2|} \phi_{P_1}\phi_{P_2}.
\end{equation}
Comparing this with our previous formal TQFT discussion, \eqref{eq:boundary-stretching} and \eqref{eq:boundary-stretching-2} are a combination of \eqref{eq:boundary-sign} and \eqref{eq:tqft} (the difference reflects the fact that ``gluing'' is an asymptotic process in the present context). Of course, much more complicated degenerations, where several ends are being stretched, are routinely used in Floer theory (and we will do the same). The discussion above was intended merely as a reminder, including sign conventions. 

\section{Conclusion\label{sec:end}}

While Floer cohomology is not strictly speaking an instance of the TQFT formalism from Sections \ref{sec:tqft}--\ref{sec:elliptic}, translating the formal arguments into Floer-theoretic ones is mostly a straightforward task (compare \cite{seidel16}, which was structured in the same way). We will give a brief sketch of how this leads to the operations from Propositions \ref{th:operations-1}--\ref{th:operations-3}. After that, we will discuss connections, and how the material from Section \ref{sec:diff-axiom} gives rise to the proofs of Propositions \ref{th:diff-q} and \ref{th:two-connections}.

\subsection{Floer cohomology groups associated to Lefschetz fibrations}
In Section \ref{sec:floer}, we explained how, given a connection $a \in \Omega^1(S^1,\frakg)$ with elliptic or hyperbolic holonomy, one chooses a compatible Floer datum $\Phi = (H,J)$ and obtains a Floer cochain complex $\mathit{CF}^*(E,H)$. We want to make a few additional comments concerning that construction:
\begin{itemize} \itemsep.5em
\item
Suppose that $a$ has constant coefficients, which are sufficiently small. Then, one can choose $H$ with the same properties, and a version of the argument from \cite{hofer-salamon95} yields an isomorphism between the Morse and Floer cochain complexes,
\begin{equation} \label{eq:pss}
\mathit{CM}^*(E,H) \iso \mathit{CF}^*(E,H).
\end{equation}

\item
The time variable $t \in S^1$ in the definition of Floer cohomology can be shifted. Given $a = a_t \mathit{dt}$, consider ${}^\tau a = a_{t+\tau} \mathit{dt}$, where $\tau \in S^1$ is an arbitrary constant. If $\Phi = (H_t,J_t)$ is a Floer datum for $a$, ${}^\tau \Phi = (H_{t+\tau}, J_{t+\tau})$ is a Floer datum for ${}^\tau a$. One-periodic orbits and Floer trajectories can be identified accordingly, leading to an isomorphism of chain complexes 
\begin{equation} \label{eq:e1}
\mathit{CF}^*(E,{}^\tau H) \iso \mathit{CF}^*(E,H). 
\end{equation}

\item
In a similar spirit, let's consider reversing the orientation of $S^1$, which in the notation of \eqref{eq:a01-cover} means passing from $a$ to $a^{-1}$. Given a Floer datum $\Phi = (H_t,J_t)$ for $a$, consider $\Phi^{-1} = (-H_{-t}, J_{-t})$. The flow of the new Hamiltonian is inverse to the old one, and one can similarly relate Floer trajectories by setting $u^{-1}(s,t) = u(-s,-t)$. Taking signs and grading into account, the outcome is an isomorphism of chain complexes 
\begin{equation} \label{eq:e2}
\mathit{CF}^*(E,H^{-1}) \iso \mathit{CF}^{2n-*}(E,H)^\vee. 
\end{equation}
\end{itemize}
Let's move on to the maps \eqref{eq:floer-gamma} between Floer complexes. Following a standard strategy, one uses them first of all to construct continuation maps which show that, up to canonical isomorphism, the choice of Floer datum doesn't affect Floer cohomology. That involves taking a cylinder $S = \bR \times S^1$, with $A$ the pullback of $a$, and a perturbation datum which agrees with one choice of Floer datum on each end. We refer to \cite{salamon-zehnder92} for the original exposition of this process. Once it has been carried out, one can write the Floer cohomology groups as $\mathit{HF}^*(E,a)$. It is elementary to show that \eqref{eq:e1} and \eqref{eq:e2} are compatible with continuation maps. Therefore, one has canonical isomorphisms
\begin{align}
\label{eq:rotate-time}
& \mathit{HF}^*(E,{}^\tau a) \iso \mathit{HF}^*(E,a) \quad \text{for any $\tau \in S^1$,} \\
\label{eq:flip-time}
& \mathit{HF}^*(E,a^{-1}) \iso \mathit{HF}^{2n-*}(E,a)^\vee.
\end{align}

\begin{example} \label{th:fractional-rotation-2}
The isomorphism \eqref{eq:rotate-time} is particularly interesting for a connection which is a $\mu$-fold pullback, in the sense of \eqref{eq:a01-cover}, for $\mu>1$. In that case, taking $\tau \in (\mu^{-1}\bZ)/\bZ$ yields ${}^\tau a^\mu = a^\mu$, resulting in a $\bZ/\mu$-action on $\mathit{HF}^*(E,a^\mu)$. Note that even though we are working with a connection which has a $\bZ/\mu$-symmetry, the associated Floer datum does not have to obey the same restriction (thereby avoiding equivariant transversality problems). Concretely, the chain maps underlying the action are
\begin{equation}
\xymatrix{
\mathit{CF}^*(E,H) \ar[r]^-{\eqref{eq:e1}}_-{\iso} & \mathit{CF}^*(E,{}^\tau H)  \ar[rr]^-{\text{continuation}}_-{\htp} && \mathit{CF}^*(E,H),
}
\end{equation}
where ${}^\tau H \neq H$ in general; this is not a (chain level) isomorphism, but only a homotopy equivalence. Similarly, the group composition law is satisfied only up to homotopy.
\end{example}

Given $a_0,a_1 \in \Omega^1(S^1,\frakg)$ with the same rotation number, one can find a nonnegatively curved connection $A \in \Omega^1(S,\frakg)$ which agrees with the pullbacks of $a_0$ and $a_1$ on the ends (in the hyperbolic case, this is Proposition \ref{th:nonnegative-connection}; in the elliptic case, the holonomies are conjugate, and one can choose $A$ to be flat).
Via \eqref{eq:floer-gamma}, that gives rise to a map
\begin{equation} \label{eq:hf-map}
\mathit{HF}^*(E,a_1) \longrightarrow \mathit{HF}^*(E,a_0).
\end{equation}
To understand these maps, it is convenient to allow a little more freedom in the choice of $A$. Namely, let's require that
\begin{equation} \label{eq:rotate-one-end}
A_{s,t} = \begin{cases} a_0 = a_{0,t} \mathit{dt} & s \ll 0, \\
{}^\tau a_1 = a_{1,t+\tau} \mathit{dt} & s \gg 0, \text{ for some $\tau \in S^1$,}
\end{cases}
\end{equation}
and also choose the perturbation datum to have the correspondingly rotated asymptotic behaviour as $s \rightarrow +\infty$. In view of \eqref{eq:rotate-time}, this still gives rise to maps \eqref{eq:hf-map}. An equivalent formulation would be to say that we still consider $S$ as a surface with tubular ends, but where the parametrization of one end has been modified; with respect to that, it still carries a connection that reduces to $a_0$ and $a_1$ over the ends. The implications depend on the kind of connections involved:
\begin{itemize}
\itemsep.5em
\item For connections with hyperbolic holonomy and rotation number $r \in \bZ \setminus \{0\}$, there is a $\bZ/r$ worth of choices of \eqref{eq:hf-map}, by Proposition \ref{th:x-annulus}. A gluing argument (with geometrically nontrivial content, since the curvature condition means that one can't just reverse orientation on $S$ to construct the connection underlying the inverse map) shows that all such maps are isomorphisms. The outcome is just as in Propositions \ref{th:auto} and \ref{th:well-defined}. Hence, we may write $\mathit{HF}^*(E,r)$ for the associated Floer cohomology group, which is well-defined up to isomorphism (up to canonical isomorphism if one is willing to additionally fix a choice of one-parameter subgroup, to which the holonomy of all connections will belong), and carries a canonical $\bZ/r$-action. If we choose our connection to be an $r$-fold covering, the $\bZ/r$-action is that from Example \ref{th:fractional-rotation-2} (compare Example \ref{th:fractional-rotation}).

\item What we've said above also applies to the $r = 0$ case, but there is an additional quirk. The abstract theory (Proposition \ref{th:auto}) predicts that $\mathit{HF}^*(E,0)$ has an endomorphism of degree $-1$, obtained by treating $\tau$ as a variable in \eqref{eq:rotate-one-end}, and looking at the corresponding parametrized problem (this is the BV operator in Hamiltonian Floer cohomology). However, by looking at the specific choice of connection and Floer datum giving rise to \eqref{eq:pss}, one sees that the BV operator vanishes.

\item For connections with elliptic holonomy, each group $\mathit{HF}^*(E,r)$ depends only on $r \in \bR \setminus \bZ$, up to canonical isomorphism, and carries a BV operator; compare Proposition \ref{th:only-elliptic}(i) and (ii). 
\end{itemize}

\begin{remark} \label{th:geometric2}
The vanishing of the BV operator on $\mathit{HF}^*(E,0)$, and the duality \eqref{eq:flip-time}, can fail in closely related geometric situations. Namely, suppose that we considered a map $\pi : E \rightarrow B$ with noncompact fibres, and which itself has a noncompact critical locus (say, of complex Morse-Bott type). We could use a version of Floer cohomology which ``wraps'' in fibre direction. In this context, neither property would hold any more, even though the formal TQFT framework remains relevant. Therefore, the failure of the abstract theory to predict these behaviours was to be expected.
\end{remark}

Along the same lines, one can consider continuation maps $\mathit{HF}^*(E,r_1) \rightarrow \mathit{HF}^*(E,r_0)$ which increase the rotation number, $r_0 > r_1$. After adapting the corresponding TQFT results, we find that such maps are canonically defined in  the following instances:
\begin{itemize}
\itemsep.5em
\item $r_1 \in \bZ$, and $r_0 = r_1+1$ (hyperbolic to hyperbolic, meaning the holonomies of the connections concerned); compare Proposition \ref{th:canonical-continuation}.
\item $r_0,r_1 \in \bR \setminus \bZ$, and $(r_1,r_0)$ contains no integer (elliptic to elliptic); see Proposition \ref{th:only-elliptic}(iii). In this case, the continuation map turns out to be an isomorphism (as in Remark \ref{th:geometric2}, this is not a formal TQFT property, but standard from a geometric viewpoint). 

\item $r_0,r_1 \in \bR \setminus \bZ$, and $(r_1,r_0)$ contains exactly one integer (elliptic to elliptic). The corresponding TQFT result is again Proposition \ref{th:only-elliptic}(iii)

\item $r_0 \in \bZ$, and $r_1 \in (r_0-1,r_0)$ (elliptic to hyperbolic), see Proposition \ref{th:mixed-continuation-2}(i).

\item $r_1 \in \bZ$, and $r_0 \in (r_1,r_1+1)$ (hyperbolic to elliptic), see Proposition \ref{th:mixed-continuation-2}(ii).
\end{itemize}
Moreover, as long as one stays within this list of cases, the composition of continuation maps is again a continuation map. In particular, since $\mathit{SH}^*(E)$ is defined as a direct limit over $\mathit{HF}^*(E,r)$ for $r \notin \bZ$, Proposition \ref{th:limit} follows immediately, as in Proposition \ref{th:mixed-continuation-2}(iii). Along the same lines, we get maps \eqref{eq:continuation-2}, which factor through the BV operator as in Proposition \ref{th:mixed-continuation-2}(iv).

\subsection{The algebraic structure of Floer cohomology}
Consider operations induced by surfaces other than the cylinder. The translation from the TQFT framework happens in the same way as before: we replace compact surfaces with boundary \eqref{eq:surface} by their analogues with tubular ends (extending the connections in the obvious way), and choose perturbation data. The naive TQFT gluing operation is replaced by neck-stretching. Then,
\begin{itemize} \itemsep.5em
\item Proposition \ref{th:gerstenhaber-algebra} carries over to give $\mathit{HF}^*(E,1)$ the structure of a Gerstenhaber algebra structure, proving Proposition \ref{th:operations-1}. One can show as in Addendum \ref{th:all-a} that this is well-defined (compatible with the canonical isomorphisms between different choices made in constructing the Floer cohomology group).

\item Proposition \ref{th:gerstenhaber-module} similarly yields a Gerstenhaber module structure on any $\mathit{HF}^*(E,r)$, which is Proposition \ref{th:operations-1b} (and moreover, this module satisfies the additional relation from Proposition \ref{th:strange-relation}). As before, this structure is compatible with the canonical isomorphisms that relate different constructions of $\mathit{HF}^*(E,1)$, and also with the corresponding non-canonical isomorphisms for $\mathit{HF}^*(E,r)$.

\item We can define all the $\mathit{HF}^*(E,r)$, $r \in \{1,2,\dots\}$, using covers $a^r$ of the same underlying connection $a^1 = a$ with rotation number $1$. Then, by proceeding as in Proposition \ref{th:bigraded-algebra}, we get the algebra structure from Proposition \ref{th:operations-2}. The same observation as in Example \ref{th:fractional-rotation-2} applies: using $a^r$ is useful because it makes it easier to characterize which connections on surfaces to use (via Addendum \ref{th:partial-fix}), but the Floer data do not have to be $\bZ/r$-invariant (or related to each other for different $r$).
\end{itemize}
The remaining item is Proposition \ref{th:operations-3}, which deserves a little more discussion. The main point can be summarized as follows: the original TQFT construction, in Proposition \ref{th:bigraded-lie}, used \eqref{eq:covering-axiom}, which was formulated as a strict chain level identity. This is potentially suspicious in a Floer-theoretic context, where everything depends on additional choices; however, we will ultimately only need the cohomology level relation, which is unproblematic.
Fix positive integers $r_i$ ($i = 0,1,2$), with $r_0 = r_1+r_2-1$. Choose connections $a^{r_i}$ with rotation number $r_i$ and which are $r_i$-fold covers (not necessarily of the same connection). Also, choose corresponding Floer data $\Phi_i$. To define
\begin{equation} \label{eq:bi-bracket}
[\cdot,\cdot]: \mathit{HF}^*(E,r_1)^{\bZ/r_1} \otimes \mathit{HF}^*(E,r_2)^{\bZ/r_2} \longrightarrow \mathit{HF}^{*-1}(E,r_0)^{\bZ/r_0}, 
\end{equation}
we start with a family of surfaces parametrized by $p \in P = S^1$. Each such surface $S_p$ is a pair-of-pants, equipped with tubular ends
\begin{equation}
\left\{
\begin{aligned}
&
\epsilon_{0,p}: (-\infty,0] \times S^1 \longrightarrow S_p, \\
&
\epsilon_{1,p},\epsilon_{2,p}: [0,\infty) \times S^1 \longrightarrow S_p.
\end{aligned}
\right.
\end{equation}
of which those for $i = 1,2$ are determined only up to rotations $\epsilon_{i,p}(s,t) \mapsto \epsilon_{i,p}(s,t+k_i r_i^{-1})$, $k_i \in \bZ/r_i$. The $p$-dependence follows the model from \eqref{eq:exact-framings}. One can arrange that there are isomorphisms
\begin{equation} \label{eq:p-shift-map}
\left\{
\begin{aligned}
& \phi_{0,p}: S_p \longrightarrow S_{p+r_0^{-1}}, \\
& \phi_{0,p} \epsilon_{0,p}(s,t) = \epsilon_{0,p+r_0^{-1}}(s,t-r_0^{-1}), \\
& \phi_{0,p} \epsilon_{i,p} = \epsilon_{i,p+r_0^{-1}} \quad \text{for $i = 1,2$,}
\end{aligned}
\right.
\end{equation}
which generate a $\bZ/r_0$-action on the family; compare \eqref{eq:rotated-conf}. We equip our surfaces with connections $A_p$ satisfying $\epsilon_{i,p}^*A_p = a_i$, and which are compatible with \eqref{eq:p-shift-map}. Again following \eqref{eq:exact-framings}, we consider the $\bZ/r_1 \times \bZ/r_2$-covering $\tilde{P} \rightarrow P$ associated to the element
\begin{equation}
(1-r_2, 1-r_1) = (-r_0,-r_0) \in H^1(P;\bZ/r_1 \times \bZ/r_2) = \bZ/r_1 \times \bZ/r_2.
\end{equation}
The pullback family of surfaces $S_{\tilde{p}}$ comes with well-defined tubular ends $\epsilon_{i,\tilde{p}}$ for $i = 0,1,2$, removing the previous ambiguity. The covering group acts on that family, fractionally rotating the ends. Moreover, the $\bZ/r_0$-action from \eqref{eq:p-shift-map} also lifts to $\tilde{P}$, resulting in a total action of $\bZ/r_0 \times \bZ/r_1 \times \bZ/r_2$. Because that action rotates the tubular ends, it is not compatible with our given Floer data, which means that there is no meaningful notion of equivariant perturbation datum. Instead, we make an arbitrary choice of perturbation datum, and then divide the resulting operation by $r_1r_2$; the outcome is invariant under the action of $\bZ/r_0 \times \bZ/r_1 \times \bZ/r_2$ on the source and target Floer cohomology groups, and we restrict to the invariant parts to get \eqref{eq:bi-bracket}. This follows \eqref{eq:divide-operation}, except that we remain on the cohomology level (since the Floer cochain complexes do not carry group actions; see Example \ref{th:fractional-rotation-2}). 

Now consider the composition of two such brackets, say
\begin{equation} \label{eq:composition-bracket}
\begin{aligned}
& \mathit{HF}^*(E,a^{r_1})^{\bZ/r_1} \otimes \mathit{HF}^*(E,a^{r_2})^{\bZ/r_2} \otimes \mathit{HF}^*(E,a^{r_3})^{\bZ/r_3} \\ & \qquad
\xrightarrow{[\cdot,\cdot] \otimes \mathit{id}} \mathit{HF}^{*-1}(E,a^{r_1+r_2-1})^{\bZ/r_1+r_2-1}
\otimes \mathit{HF}^*(E,a^{r_3})^{\bZ/r_3} \\\ & \qquad \xrightarrow{[\cdot,\cdot]} \mathit{HF}^{*-2}(E,a^{r_1+r_2+r_3-2})^{\bZ/r_1+r_2+r_3-2}.
\\
\end{aligned}
\end{equation}
Let's write $\tilde{P}_1$ and $\tilde{P}_2$ for the parameter spaces that occur in the two steps of \eqref{eq:composition-bracket}. Both spaces carry actions of $\bZ/r_0$, where $r_0 = r_1+r_2-1$. To prove the Jacobi identity following the strategy from Proposition \ref{th:bigraded-lie}, one has to relate \eqref{eq:composition-bracket} to the operation derived from the quotient space $\tilde{P}_1 \times_{\bZ/r_0} \tilde{P}_2$. For that, one argues as follows: gluing writes \eqref{eq:composition-bracket} as an operation arising from a family of surfaces over $\tilde{P}_1 \times \tilde{P}_2$ (for an arbitrary choice of perturbation data). On that family, the $\bZ/r_0$-action no longer affects the tubular ends, and hence, we can choose the perturbation data to be equivariant. In other words, our family descends to $\tilde{P}_1 \times_{\bZ/r_0} \tilde{P}_2$, including tubular ends; we can make a generic choice of perturbation data on that quotient space, and then pull that choice back to $\tilde{P}_1 \times \tilde{P}_2$. For those specific choices, we have an analogue of \eqref{eq:covering-axiom} on the Floer cochain level. Hence, the corresponding Floer cohomology level identity holds.

\subsection{The differentiation axiom revisited\label{subsec:diff-axiom-2}}
We will now implement differentiation with respect to the Novikov parameter $q$ in Floer theory. This closely follows \cite[Section 9]{seidel16}. As a preliminary step, it is convenient to modify slightly the definition of the Floer differential and of the operations \eqref{eq:floer-gamma}; the new version is isomorphic to the original one (admittedly, the isomorphism is canonical only up to powers of $q$, but that is hopefully a venial sin; in any case, all previous constructions also go through with the new version, without any other modifications). Let's fix an $\bR$-cycle of codimension two, which means a formal finite linear combination
\begin{equation}
\Omega = \sum_j m_j\Omega_j,
\end{equation}
where $\Omega_j \subset E$ is a properly embedded codimension two submanifold, and the multiplicities are $m_j \in \bR$. This should be such that the Poincar\'e dual to $\Omega$ is the symplectic class $[\omega_E]$. When carrying out any Floer-theoretic construction, we assume that the one-periodic orbits of the relevant Hamiltonians are disjoint from $\Omega$; and when counting maps $u: S \rightarrow E$ satisfying a Cauchy-Riemann equation with suitable asymptotics, we replace the expression $q^{\int_u \omega_E}$ with the intersection number
\begin{equation} \label{eq:usual-power}
q^{u \cdot \Omega} = q^{\sum_j m_j (u \cdot \Omega_j)}. 
\end{equation}
The naive differentiation operator on any Floer cochain group is
\begin{equation}
\begin{aligned}
& \qabla: \mathit{CF}^*(E,H) \longrightarrow \mathit{CF}^*(E,H), \\
& \qabla(\textstyle \sum_x c_x x) = \sum_x (\partial_q c_x) x,
\end{aligned}
\end{equation}
where the coefficients are $c_x \in \bK$. Let's assume that our Floer datum $\Phi = (H,J)$ has been chosen in such a way that the isolated solutions $u$ of Floer's equation intersect $\Omega$ transversally, and no $J_t$-holomorphic sphere intersects $\Omega$ (these conditions are satisfied generically, thanks to transversality of evaluations). Then, one can write
\begin{equation} \label{eq:define-r}
\begin{aligned}
& R = \qabla d - d \qabla: \mathit{CF}^*(E,H) \longrightarrow \mathit{CF}^{*+1}(E,H), \\
& R (x_+) = \sum_{x_-} \Big( \sum_u \pm (u \cdot \Omega) q^{u \cdot \Omega-1} \Big) x_-
\end{aligned}
\end{equation}
as an operation associated to a certain parametrized moduli space. Namely, for each $\Omega_j$, we consider pairs $(p,u)$ consisting of $p \in P = S^1$ and a Floer trajectory $u$, such that at the point $\zeta_p = (0,-p)$, we have $u(\zeta_p) \in \Omega_j$. We count those those solutions with multiplicity $q^{-1}m_j$ times as the usual power \eqref{eq:usual-power} of $q$, and add up over $j$; see \cite[Section 9a]{seidel16}. In future, whenever such a weighted count is used, we will refer to it by the shorthand notation ``counting those $u$ such that $u(\zeta_p) \in q^{-1}\Omega$''.

For expository reasons, let's briefly consider the Floer-theoretic analogue of \eqref{eq:e}, namely the unit class in $\mathit{HF}^0(E,1)$. This is constructed as follows. Choose some $a_1 \in \Omega^1(S^1,\frakg)$ with $\mathrm{rot}(a_1) = 1$. Take $S = \bC$, considered as a Riemann surface with an output end. We equip that with a nonnegatively curved connection which reduces to $a_1$ on the end. Choose a Floer datum $\Phi_1 = (H_1,J_1)$ associated to $a_1$, and a corresponding perturbation datum. Counting solutions of the resulting equation \eqref{eq:e-equation} then yields a cocycle representative for the unit in $\mathit{CF}^0(E,H_1)$. Now let's additionally fix a point $\zeta \in S$, and impose the condition
\begin{equation}
u(\zeta) \in q^{-1}\Omega
\end{equation}
on the solutions of our equation \eqref{eq:e-equation}. The outcome is the analogue of the Kodaira-Spencer class in $\mathit{HF}^2(E,1)$, represented by a Floer cocycle
\begin{equation} \label{eq:floer-k-s}
k \in \mathit{CF}^2(E,H_1).
\end{equation}
In parallel with Lemma \ref{th:bracket-with-k}, one has a chain homotopy
\begin{equation} \label{eq:geometric-rho}
\begin{aligned}
& \rho: \mathit{CF}^*(E,H_2) \longrightarrow \mathit{CF}^*(E,H_2), \\
& d\rho(x) - \rho(dx) - R(x) + [k,x] = 0.
\end{aligned}
\end{equation}
The context for this is that one takes a connection $a_2$ with hyperbolic holonomy and arbitrary rotation number $r$, as well as an associated Floer datum $(H_2,J_2)$, chosen so that the operator $R$ on $\mathit{CF}^*(E,H_2)$ admits the enumerative interpretation described above. The definition of \eqref{eq:geometric-rho} involves a family of Riemann surfaces with marked point $(S_p,\zeta_p)$, parametrized by $p = (p_1,p_2) \in P =[0,\infty) \times S^1$, which for $p_1 = 0$ reduce to those defining \eqref{eq:define-r}. At the other end $p_1 \rightarrow \infty$, we ``pull out'' the marked point so that the surface degenerates into two pieces, one defining the Lie bracket and the other \eqref{eq:floer-k-s}. The process is shown schematically in Figure \ref{fig:pull-out}. Of course, our surfaces need to carry connections, and those are constructed exactly as in Lemma \ref{th:bracket-with-elliptic-k}. In particular, this construction constrains how the ends of the limiting pair-of-pants ($p_1 = \infty$) change depending on the remaining parameter $p_2 \in S^1$ (see Lemma \ref{th:rot-bracket}).
\begin{figure}
\begin{centering}
\begin{picture}(0,0)%
\includegraphics{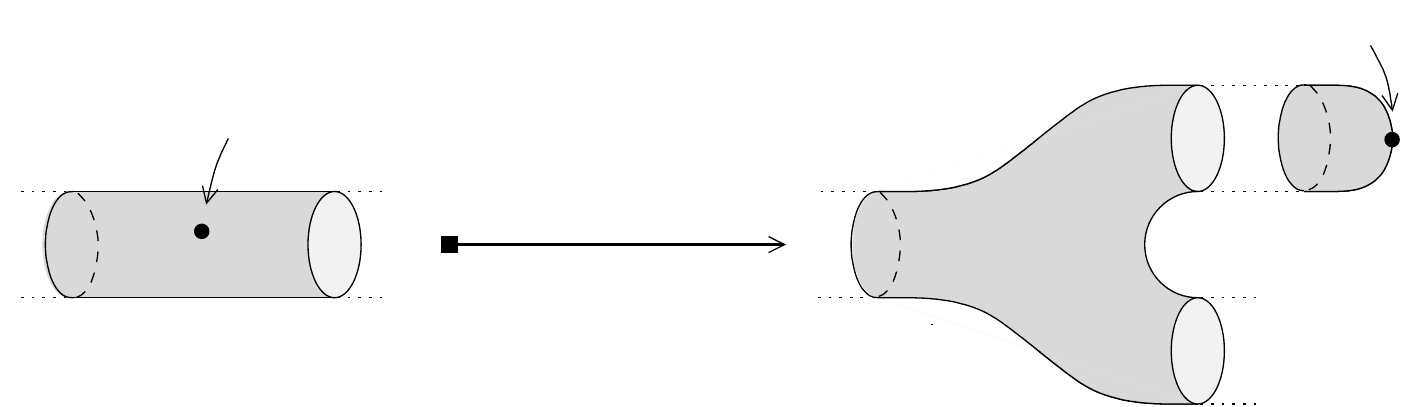}%
\end{picture}%
\setlength{\unitlength}{3355sp}%
\begingroup\makeatletter\ifx\SetFigFont\undefined%
\gdef\SetFigFont#1#2#3#4#5{%
  \reset@font\fontsize{#1}{#2pt}%
  \fontfamily{#3}\fontseries{#4}\fontshape{#5}%
  \selectfont}%
\fi\endgroup%
\begin{picture}(7910,2283)(-239,-1723)
\put(1876,-886){\makebox(0,0)[lb]{\smash{{\SetFigFont{10}{12.0}{\rmdefault}{\mddefault}{\updefault}{$r$}%
}}}}
\put(3451,-661){\makebox(0,0)[lb]{\smash{{\SetFigFont{10}{12.0}{\rmdefault}{\mddefault}{\updefault}{$p_1 \rightarrow \infty$}%
}}}}
\put(2251,-661){\makebox(0,0)[lb]{\smash{{\SetFigFont{10}{12.0}{\rmdefault}{\mddefault}{\updefault}{$p_1 = 0$}%
}}}}
\put(7126,389){\makebox(0,0)[lb]{\smash{{\SetFigFont{10}{12.0}{\rmdefault}{\mddefault}{\updefault}{$q^{-1}\Omega$}%
}}}}
\put(-154,-886){\makebox(0,0)[lb]{\smash{{\SetFigFont{10}{12.0}{\rmdefault}{\mddefault}{\updefault}{$r$}%
}}}}
\put(826,-136){\makebox(0,0)[lb]{\smash{{\SetFigFont{10}{12.0}{\rmdefault}{\mddefault}{\updefault}{$q^{-1}\Omega$}%
}}}}
\put(6751,-1486){\makebox(0,0)[lb]{\smash{{\SetFigFont{10}{12.0}{\rmdefault}{\mddefault}{\updefault}{$r$}%
}}}}
\put(4396,-886){\makebox(0,0)[lb]{\smash{{\SetFigFont{10}{12.0}{\rmdefault}{\mddefault}{\updefault}{$r$}%
}}}}
\put(6771,-291){\makebox(0,0)[lb]{\smash{{\SetFigFont{10}{12.0}{\rmdefault}{\mddefault}{\updefault}{$1$}%
}}}}
\end{picture}%
\caption{\label{fig:pull-out}The geometry behind \eqref{eq:geometric-rho}, with the rotation numbers of the connections associated to the ends. Note that there is another para\-meter $p_2 \in S^1$, not shown here.}
\end{centering}
\end{figure}
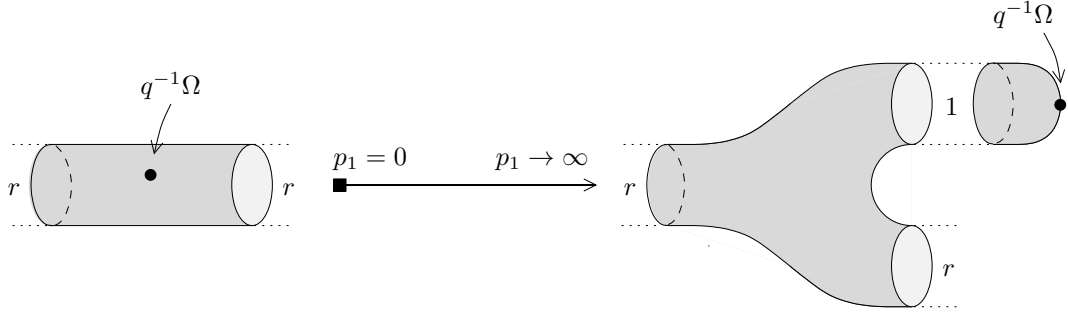%

The corresponding construction for connections with elliptic holonomy, which is the subject of \cite{seidel16}, is slightly more complicated. In that case, one can't express $R$ itself as a Lie bracket, but has to combine it with a continuation map, so as to obtain the analogue of \eqref{eq:rho-homotopy-elliptic}:
\begin{equation} \label{eq:geometric-rho-c}
\begin{aligned}
& \rho_c: \mathit{CF}^*(E,H_2) \longrightarrow \mathit{CF}^*(E,H_0), \\
& d\rho_c(x)- \rho_c(dx) - C(R(x)) + [k,x]_c = 0.
\end{aligned}
\end{equation}
The context for this equation is as follows. Throughout, we work with connections of the form \eqref{eq:standard-connection} and \eqref{eq:d-beta}. We fix numbers $r_i$ as in Lemma \ref{th:bracket-with-elliptic-k}, and Floer data $(H_i,J_i)$ that are compatible with the connections $r_i \alpha\, \mathit{dt}$. This means that on a neighbourhood of infinity, the Hamiltonian vector field of $H_i$ is (time-independent, and) the horizontal lift of the infinitesimal rotation of $B$ with speed $r_i$ (this is a version of the standard setup for symplectic cohomology). The family of surfaces which defines \eqref{eq:geometric-rho-c} is parametrized by $p = (p_1,p_2) \in P = \bR \times S^1$, and splits as we approach either end $p_1 \rightarrow \pm\infty$ (see Figure \ref{fig:pullout-2}). The parameter $c \in \bZ$ affects only the geometry of the limit $p_1 \rightarrow +\infty$. In that limit, the cylinder splits into a pair-of-pants and a copy of $\bC$. The copy of $\bC$ always carries the data defining a version of the Kodaira-Spencer cocycle, lying in $\mathit{CF}^2(E,H_1)$, which is independent of the remaining parameter $p_2$. In contrast, the pair-of-pants does depend on $p_2$ and on $c$ (taking the tubular ends into account). It is convenient to write it as
\begin{equation}
S_{p_2} = (\bR \times S^1) \setminus \{(0,-p_2)\},
\end{equation}
with coordinates $(s,t)$. Then, the two ends asymptotic to $s = \pm \infty$ can be taken to be standard, hence independent of $p_2$; while the parametrization of the end asymptotic to $(0,p_2)$ rotates $(c+1)$ times as $p_2$ goes around the circle.
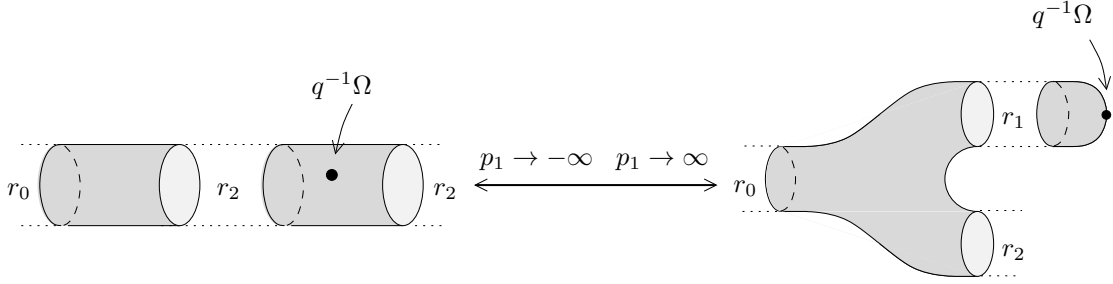
\begin{figure}
\begin{centering}
\begin{picture}(0,0)%
\includegraphics{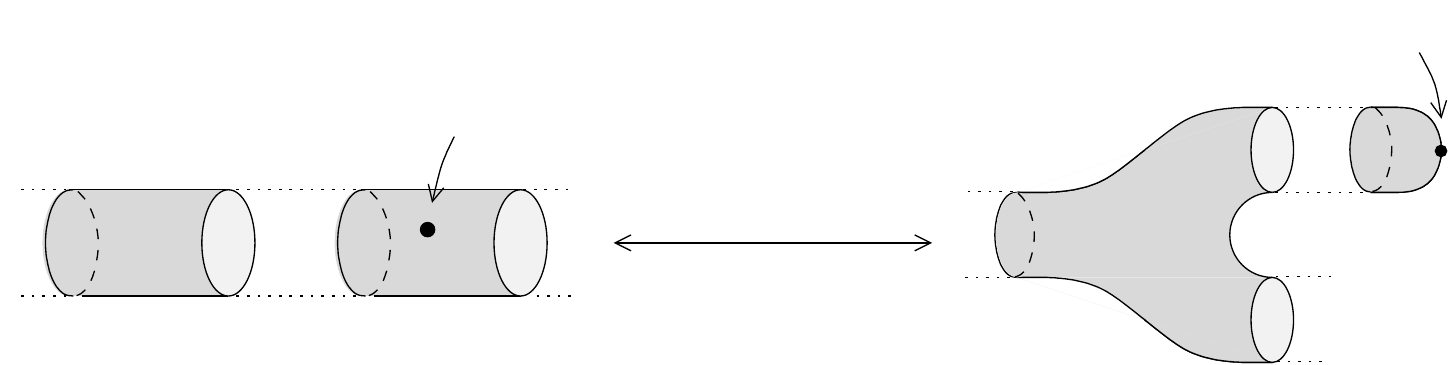}%
\end{picture}%
\setlength{\unitlength}{3355sp}%
\begingroup\makeatletter\ifx\SetFigFont\undefined%
\gdef\SetFigFont#1#2#3#4#5{%
  \reset@font\fontsize{#1}{#2pt}%
  \fontfamily{#3}\fontseries{#4}\fontshape{#5}%
  \selectfont}%
\fi\endgroup%
\begin{picture}(8181,2058)(-1064,-1498)
\put(3451,-661){\makebox(0,0)[lb]{\smash{{\SetFigFont{10}{12.0}{\rmdefault}{\mddefault}{\updefault}{$p_1 \rightarrow \infty$}%
}}}}
\put(4318,-863){\makebox(0,0)[lb]{\smash{{\SetFigFont{10}{12.0}{\rmdefault}{\mddefault}{\updefault}{$r_0$}%
}}}}
\put(2446,-661){\makebox(0,0)[lb]{\smash{{\SetFigFont{10}{12.0}{\rmdefault}{\mddefault}{\updefault}{$p_1 \rightarrow -\infty$}%
}}}}
\put(6301,-1336){\makebox(0,0)[lb]{\smash{{\SetFigFont{10}{12.0}{\rmdefault}{\mddefault}{\updefault}{$r_2$}%
}}}}
\put(6301,-361){\makebox(0,0)[lb]{\smash{{\SetFigFont{10}{12.0}{\rmdefault}{\mddefault}{\updefault}{$r_1$}%
}}}}
\put(6526,389){\makebox(0,0)[lb]{\smash{{\SetFigFont{10}{12.0}{\rmdefault}{\mddefault}{\updefault}{$q^{-1}\Omega$}%
}}}}
\put(1201,-136){\makebox(0,0)[lb]{\smash{{\SetFigFont{10}{12.0}{\rmdefault}{\mddefault}{\updefault}{$q^{-1}\Omega$}%
}}}}
\put(2101,-886){\makebox(0,0)[lb]{\smash{{\SetFigFont{10}{12.0}{\rmdefault}{\mddefault}{\updefault}{$r_2$}%
}}}}
\put(501,-886){\makebox(0,0)[lb]{\smash{{\SetFigFont{10}{12.0}{\rmdefault}{\mddefault}{\updefault}{$r_2$}%
}}}}
\put(-1049,-886){\makebox(0,0)[lb]{\smash{{\SetFigFont{10}{12.0}{\rmdefault}{\mddefault}{\updefault}{$r_0$}%
}}}}
\end{picture}%
\caption{\label{fig:pullout-2}The geometric construction of \eqref{eq:geometric-rho-c}.}
\end{centering}
\end{figure}%

In the proof of Proposition \ref{th:two-connections} (modelled on that of Proposition \ref{th:comparison-of-connections}), a variant $\rho_{\mathit{broken},c}$ of \eqref{eq:geometric-rho-c} occurs, in which the continuation map is replaced by a composition of two such maps. As far as the Floer-theoretic construction is concerned, the only difference is that in the limit $p_1 \rightarrow -\infty$, we have a ``double stretching'' of two necks as the same time (the relative speed at which this happens is irrelevant, as long as both necks stretch unboundedly). The comparison between the two constructions is given by a higher homotopy, which is the counterpart of \eqref{eq:magic-triangle}. The three-dimensional parameter space involved is drawn schematically in Figure \ref{fig:triangle2}. Topologically, this space can be thought of as a triangle, with a vertex and the opposite side removed, times $S^1$. There are two codimension $1$ boundary faces (along which the surfaces do not degenerate), plus one more such face at infinity, which would appear when compactifying it (and along which a standard stretching process happens). We want to pay particular attention to what happens at the vertices of our triangle:
\begin{itemize}
\itemsep.5em
\item 
Near the vertex on the right (as drawn in Figure \ref{fig:triangle2}), the parameter space has local coordinates $(p_1,p_2,p_3) \in \bR \times[0,1] \times S^1$, where $p_1 \gg 0$. We can choose the Riemann surfaces and all their auxiliary data to be independent of $p_2$, while preserving regularity of the parametrized moduli space. In the region where that holds, no isolated points of that moduli space can occur; and any one-dimensional components that occur are of the form $\{p_1\} \times [0,1] \times \{p_3\}$.

\item
Near the top left vertex, we have local coordinates $(p_1,p_2,p_3) \in \bR \times (-\infty,0] \times S^1$, where $p_1 \ll 0$ and $p_2$ is small. Here, $p_2 = 0$ is the adjacent boundary face, and $p_1 = -\infty$ the ``face at infinity''. We again choose all data to be $p_2$-independent, with similar consequences as before. 

\item
Near the bottom left vertex, we have local coordinates $(p_1,p_2,p_3) \in \bR \times [0,\infty) \times S^1$, where $p_1 \ll 0$ and $p_2$ is small. Here, we have two ``gluing parameters'' or stretching lengths. Along the ``face at infinity'' $p_1 = -\infty$, the $p_2$ parameter governs the gluing of the left two of the three cylinders (the last cylinder remains the same). Along the boundary face $p_2 = 0$, we carry out both gluing processes simultaneously (with some relation between the gluing parameters). This is, roughly speaking, half of the more standard process near a codimension two corner (since we do not consider a region where only the right two cylinders are glued together).
\end{itemize}
\begin{figure}
\begin{centering}
\begin{picture}(0,0)%
\includegraphics{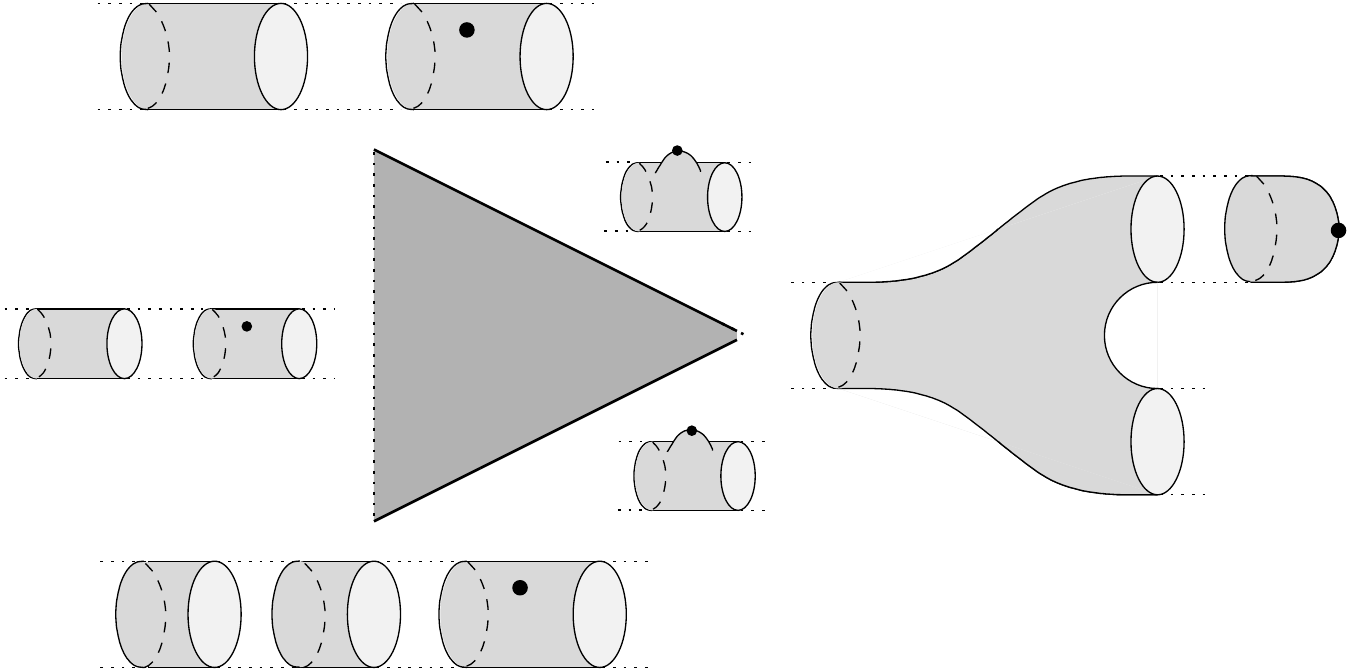}%
\end{picture}%
\setlength{\unitlength}{3355sp}%
\begingroup\makeatletter\ifx\SetFigFont\undefined%
\gdef\SetFigFont#1#2#3#4#5{%
  \reset@font\fontsize{#1}{#2pt}%
  \fontfamily{#3}\fontseries{#4}\fontshape{#5}%
  \selectfont}%
\fi\endgroup%
\begin{picture}(7607,3777)(-9311,-3451)
\put(-6974,-1636){\makebox(0,0)[lb]{\smash{{\SetFigFont{10}{12.0}{\rmdefault}{\mddefault}{\updefault}{parameter space}%
}}}}
\end{picture}%
\caption{\label{fig:triangle2}Relating $\rho_c$ and $\rho_{\mathit{broken},c}$. As in Figure \ref{fig:pull-out}, there is an additional $S^1$ degree of freedom, which we have not represented.}
\end{centering}
\end{figure}
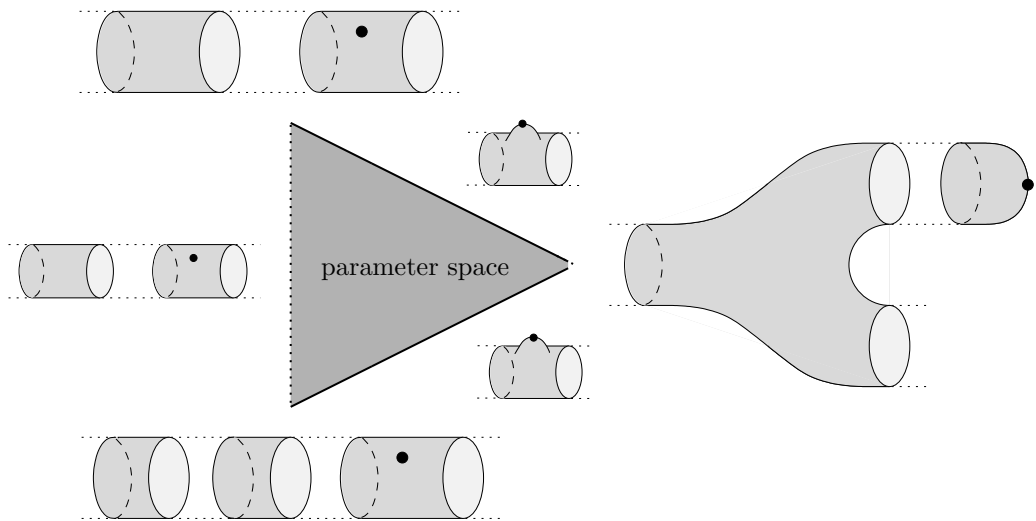

For the other parameter spaces appearing in the proof of Proposition \ref{th:two-connections} (in particular that from Figure \ref{fig:pentagon}), similar considerations apply when translating them into Floer-theoretic moduli spaces.
 

\end{document}